\pdfoutput=1

\documentclass[10pt]{article}

\usepackage{url}
\usepackage{mathtools}
\usepackage{amssymb}
\usepackage{amsthm}
\usepackage{latexsym}
\usepackage[shortlabels]{enumitem}
\usepackage{dsfont}
\usepackage{appendix}
\usepackage{color} 
\usepackage[utf8]{inputenc}
\usepackage[T1]{fontenc}
\usepackage{geometry}
\usepackage{todonotes}
\usepackage{lmodern}
\usepackage{anyfontsize}
\usepackage{stmaryrd}
\usepackage{comment}
\usepackage{bm}
\usepackage{mathrsfs}
\usepackage{mdframed}
\usepackage{natbib}
\usepackage[english]{babel}
\usepackage{cases}
\usepackage{braket}

\mdfsetup{middlelinecolor=blue,middlelinewidth=2pt,linewidth=0pt,backgroundcolor=blue!15,,roundcorner=10pt}

\allowdisplaybreaks

\setcitestyle{numbers,open={[},close={]}}

\definecolor{red}{rgb}{0.7,0.15,0.15}
\definecolor{green}{rgb}{0,0.5,0}
\definecolor{blue}{rgb}{0,0,0.7}

\usepackage{hyperref}
\hypersetup{colorlinks, linkcolor={red}, citecolor={green}, urlcolor={blue}}
\numberwithin{equation}{section}

\newtheorem{theorem}{Theorem}[section]
\newtheorem{assumption}[theorem]{Assumption}
\newtheorem{standing_assumption}[theorem]{Standing assumption}

\newtheorem{example}[theorem]{Example}

\newtheorem{lemma}[theorem]{Lemma}
\newtheorem{proposition}[theorem]{Proposition}

\newtheorem{definition}[theorem]{Definition}
\newtheorem{remark}[theorem]{Remark}

\usepackage[capitalize,nameinlink,noabbrev]{cleveref}

\crefname{assumption}{assumption}{assumptions}
\Crefname{assumption}{Assumption}{Assumptions}
\crefname{assumptio}{assumption}{assumptions}
\Crefname{assumptio}{Assumption}{Assumptions}
\crefname{assumptionalt}{assumption}{assumptions}
\Crefname{assumptionalt}{Assumption}{Assumptions}
\crefname{lemma}{lemma}{lemmas}
\Crefname{lemma}{Lemma}{Lemmas}
\crefname{proposition}{proposition}{propositions}
\Crefname{proposition}{Proposition}{Propositions}
\crefname{definition}{definition}{definitions}
\Crefname{definition}{Definition}{Definitions}
\crefname{remark}{remark}{remarks}
\Crefname{remark}{Remark}{Remarks}
\crefname{appendix}{appendix}{appendices}
\Crefname{appendix}{Appendix}{Appendices}
\crefname{corollary}{corollary}{corollaries}
\Crefname{corollary}{Corollary}{Corollaries}
\crefname{example}{example}{examples}
\Crefname{example}{Example}{Examples}

\AddToHook{env/assumption/begin}{\crefalias{theorem}{assumption}}
\AddToHook{env/lemma/begin}{\crefalias{theorem}{lemma}}
\AddToHook{env/proposition/begin}{\crefalias{theorem}{proposition}}
\AddToHook{env/definition/begin}{\crefalias{theorem}{definition}}
\AddToHook{env/remark/begin}{\crefalias{theorem}{remark}}
\AddToHook{env/corollary/begin}{\crefalias{theorem}{corollary}}
\AddToHook{env/example/begin}{\crefalias{theorem}{example}}

\newcommand\cB{\mathcal B}

\newcommand\cE{\mathcal E}
\newcommand\cF{\mathcal F}
\newcommand\cG{\mathcal G}
\newcommand\cH{\mathcal H}

\newcommand\cL{\mathcal L}
\newcommand\cM{\mathcal M}
\newcommand\cN{\mathcal N}

\newcommand\cP{\mathcal P}

\newcommand\cS{\mathcal S}
\newcommand\cT{\mathcal T}
\newcommand\cU{\mathcal U}

\newcommand\cY{\mathcal Y}
\newcommand\cZ{\mathcal Z}

\newcommand\sL{\mathscr L}

\newcommand\sN{\mathscr N}

\newcommand\sT{\mathscr T}

\newcommand\sfB{\mathsf B}

\newcommand\frM{\mathfrak M}

\newcommand\fH{\mathfrak H}
\newcommand\fP{\mathfrak{P}}

\newcommand\ff{\mathfrak f}
\newcommand\fg{\mathfrak g}

\newcommand{\smallertext}[1]{\text{\fontsize{5}{5}\selectfont$#1$}}
\newcommand{\smalltext}[1]{\text{\fontsize{4}{4}\selectfont$#1$}}
\newcommand{\tinytext}[1]{\text{\fontsize{3}{3}\selectfont$#1$}}

\newcommand{\vertiii}[1]{{\left\vert\kern-0.25ex\left\vert\kern-0.25ex\left\vert #1 \right\vert\kern-0.25ex\right\vert\kern-0.25ex\right\vert}}

\def \D{\mathbb{D}}
\def \E{\mathbb{E}}
\def \F{\mathbb{F}}
\def \G{\mathbb{G}}
\def \H{\mathbb{H}}

\def \L{\mathbb{L}}

\def \N{\mathbb{N}}
\def \P{\mathbb{P}}
\def \Q{\mathbb{Q}}
\def \R{\mathbb{R}}
\def \S{\mathbb{S}}

\def\Ec{\mathcal{E}}

\newcommand{\bcdot}{\boldsymbol{\cdot}}

\newcommand{\1}{\mathbf{1}}

\def\d{\mathrm{d}}

\DeclareMathOperator*{\esssup}{ess\,sup}
\DeclareMathOperator*{\essinf}{ess\,inf}

\newcommand\sgn{\text{sgn}}

\begin{document}
	
	\title{Mind the jumps: well-posedness of semi-martingale 2BSDEs}

	\author{Dylan {\sc Possama\"{i}}\footnote{ETH Z\"{u}rich, Department of Mathematics, Switzerland, dylan.possamai@math.ethz.ch. This author gratefully acknowledges partial support by the SNF project MINT 205121-219818.}\and Marco {\sc Rodrigues}\footnote{TU Berlin and Weierstrass Institute, Germany, marco.rodrigues@tu-berlin.de and marco.rodrigues@wias-berlin.de. This author gratefully acknowledges support from the Swiss National Science Foundation project MINT 205121-219818 and the Deutsche Forschungsgemeinschaft CRC/TRR 388 Project ID 516748464.} \and Alexandros {\sc Saplaouras}\footnote{University of the Aegean, Department of Statistics and Actuarial--Financial Mathematics, Greece, alsapl@aegean.gr.  This author gratefully acknowledges the financial support from the Hellenic foundation for research and innovation grant 235 `Stability and numerics for BSDEs under model uncertainty and applications' (2nd call for H.F.R.I. research projects to support post-doctoral researchers).}}
	
	\date{September 24, 2026}
	
	\maketitle
		
	\begin{abstract}
		We develop a well-posedness theory for second-order backward stochastic differential equations (2BSDEs) with jumps. This work covers two complementary notions of solution: an extrinsic and an intrinsic one. We also discuss the obstruction to aggregating jump integrands. The chosen framework allows for controlled diffusions with jumps, pure-jump processes, and discrete-time processes in a unified setting.
	\end{abstract}

\section{Introduction}
	
In our companion work \cite{possamai2026aggregation}, we constructed a c\`adl\`ag and `adapted' aggregator $\widehat{Y}$ for the family of value processes stemming from optimising solutions of BSDEs with jumps in weak formulation and over a set $\fP_0$ of possibly mutually singular probability measures. We subsequently derived corresponding dynamic programming properties. In this work, we identify the semi-martingale decomposition of $\widehat{Y}$ under every $\P\in\fP_0$ and characterise $\widehat{Y}$ together with its decomposition as the unique solution to a second-order BSDE (2BSDE) system with jumps, which, schematically, has dynamics of the form
\begin{align}\label{eq_2BSDE_intro}
	\displaystyle Y_t &= \xi + \int_t^T f^\P_s\big(Y_s,Y_{s\smallertext{-}},Z_s, U^\P_s(\cdot)\big) \d C_s- \bigg(\int_t^T Z_s \d X^{c,\P}_s\bigg)^{(\P)} \nonumber\\
	&\quad - \bigg(\int_t^T\int_{\R^\smalltext{d}} U^\P_s(x)\tilde\mu^{X,\P}(\d s,\d x)\bigg)^{(\P)}- \int_t^T \d N^\P_s + K^\P_T - K^\P_t, \; t \in [0,\infty], \; \textnormal{$\P$--a.s.}, \; \P\in\fP_0.
\end{align}
Here, $Y$ and $Z$ are universally defined, whereas $U^\P$, $N^\P$, and $K^\P$ depend on the probability $\P\in\fP_0$. The reason $Z$ can be defined universally is that it can be identified with the second characteristic of the joint semi-martingale $(Y,X)$, which can be defined independently of the probability measure; see \citeauthor*{neufeld2014measurability} \cite{neufeld2014measurability}. To enforce uniqueness of the solution, we introduce two independent conditions, leading to extrinsic and intrinsic formulations of the 2BSDE system, with the latter resembling the already existing minimality condition on $(K^\P)_{\P\in\fP_0}$ in the literature.

\medskip
Since BSDEs lie at the core of the construction of $\widehat{Y}$, we provide some background. Backward stochastic differential equations and their second-order extensions have become central tools in numerous applications, such as probabilistic representations for solutions of nonlinear partial integro-differential equations, associated probabilistic numerical methods, stochastic control, stochastic differential games, theoretical economics, and mathematical finance. Over the last three and a half decades, the theory’s development has been driven by these applications, which have, in turn, spurred further theoretical advancements. The primary motivation for this work is to extend the well-posedness of second-order BSDEs to a level of generality that enables a unified treatment of controlled diffusions, pure-jump processes, and discrete-time processes, while also overcoming the technical restrictions imposed by earlier approaches.

\medskip
Linear BSDEs\footnote{The term `BSDEs' was coined much later, in 1990, by \citeauthor*{pardoux1990adapted} \cite{pardoux1990adapted}.} date back to the early 1970s. They were introduced by \citeauthor*{davis1973dynamic} \cite{davis1973dynamic} to study stochastic control problems with drift control. Around the same time, they also appeared in the works of \citeauthor*{kushner1972necessary} \cite{kushner1972necessary} and \citeauthor*{bismut1973analyse} \cite{bismut1973analyse, bismut1973conjugate}, where they served as adjoint equations in the Pontryagin stochastic maximum principle. It is important to note that as early as \cite{davis1973dynamic}, the value process of certain stochastic control problems was identified as the $y$-component of the solution to the associated BSDE. This insight, combined with comparison principles, enabled characterisations of optimal strategies for controlling stochastic systems. Linear BSDEs continued to play a central role in the context of the stochastic maximum principle, as seen in the works of \citeauthor*{haussmann1976general} \cite{haussmann1976general}, \citeauthor*{kabanov1978on} \cite{kabanov1978on}, \citeauthor*{arkin1979necessary} \cite{arkin1979necessary}, and \citeauthor*{bensoussan1983maximum} \cite{bensoussan1983lectures, bensoussan1983maximum}. Nonlinear BSDEs, however, began to emerge later in the works of \citeauthor*{bismut1978controle} \cite{bismut1978controle, bismut1978introductory}, \citeauthor*{chitashvili1983martingale} \cite{chitashvili1983martingale}, and \citeauthor*{chitashvili1987optimal} \cite{chitashvili1987optimal, chitashvili1987optimal2}. The first systematic treatment of BSDEs with Lipschitz-continuous generators was in the seminal works of \citeauthor*{pardoux1992backward} \cite{pardoux1990adapted, pardoux1992backward}, in which they showed the well-posedness of BSDEs with $\mathbb{L}^2$-data. We refer to the illuminating survey by \citeauthor*{elkaroui1997backward} \cite{elkaroui1997backward} for various applications in the early stages of the theory.

\medskip
In subsequent years, most of the research focused on BSDEs for which $X$ is a $\P$--Brownian motion and $\d C_s = \d s$. Various extensions were explored, for example, random horizons, monotone generators, quadratic or super-linear growth in the $z$-variable, and $\L^p$-data, to name but a few. In many of these settings, BSDEs have also been shown to provide probabilistic representations of PDEs. Among further generalisations, reflected BSDEs represent a class of equations in which the $Y$-component is constrained to stay above a fixed obstacle process. They were introduced by \citeauthor*{elkaroui1997reflected} \cite{elkaroui1997reflected,elkaroui1997reflected2}; it was also shown that these equations provide probabilistic representations of certain obstacle problems for PDEs. In parallel, the literature on BSDEs including jumps also progressed, albeit less actively, along similar themes. There, the BSDEs are most often driven by a Brownian motion and a compensated Poisson random measure; these equations also provide probabilistic representations of certain partial integro-differential equations. BSDEs driven by general c\`adl\`ag martingales and possibly a general compensated random measure were also studied more recently by \citeauthor*{papapantoleon2016existence} \cite{papapantoleon2016existence,papapantoleon2021stability} and \citeauthor*{possamai2024reflections} \cite{possamai2024reflections}; the latter work studies BSDEs and reflected BSDEs, and both generalises and resolves issues that appeared in \cite{papapantoleon2016existence}. For a more detailed history of BSDEs and reflected BSDEs, we refer to these latter sources.

\medskip
While \citeauthor*{pardoux1999forward} \cite{pardoux1999forward} provided a probabilistic representation for solutions of second-order quasi-linear PDEs using coupled systems of forward--backward differential equations, one significant limitation remained: the literature on BSDEs could not address second-order fully nonlinear PDEs, that is, equations involving nonlinear dependence on the Hessian matrix. As discussed in the BSDE approach to weak control previously, this limitation is to be expected: controlling the volatility leads to mutually singular laws, see \cite[Section 1]{possamai2026aggregation}, and to nonlinear second-order terms in the PDE approach to stochastic control. To illustrate the point, we consider the simplest Markovian BSDE driven by a continuous process, namely
\begin{align}\label{eq_markov_bsde}
Y_t = g(X_T) + \int_t^T f(s,X_s,Y_s,Z_s)\mathrm{d}s - \int_t^T Z_s\mathrm{d} X_s,
\end{align}
where $X$ is now a Brownian motion and the generator $f$ is deterministic and globally Lipschitz-continuous. \citeauthor*{pardoux1992backward} \cite{pardoux1992backward} and \citeauthor*{pardoux1997probabilistic} \cite{pardoux1997probabilistic} provided a probabilistic representation of the solution of the semi-linear PDE
\begin{align*}
\partial_t u(t,x) +   \frac{1}{2}\nabla^2_x u(t,x) + f\big(t,x,u(t,x), \nabla_x u (t,x)\big) = 0,\; \text{on}\; [0,T)\times \mathbb{R}^d,\; \text{with}\; u(T,x)=g(x),\;  \text{for}\; x\in\mathbb{R}^d,
\end{align*}
in terms of the $Y$-component of the solution to \eqref{eq_markov_bsde}. Supposing now that the function $f$ also depends on the Hessian matrix $\nabla^2_x u(t,x)$, consider the non-linear partial differential equation 
\begin{align}\label{eq_PDE_nonlinear}
\partial_t u(t,x) + f\big(t,x,u(t,x), \nabla_x u (t,x), \nabla^2_x u(t,x)\big) = 0,\; \text{on}\; [0,T)\times \R^d,\; \text{with}\; u(T,x)=g(x),\; \text{for}\; x\in\R^d.
\end{align}

The first\footnote{This is not \emph{stricto sensu} the case: there were earlier attempts to obtain representations for certain fully non-linear HJB equations using value functions of control problems whose criterion was given by the solution of a BSDE, see \citeauthor*{peng1992generalized} \cite{peng1992generalized}, \citeauthor*{duffie1992stochastic} \cite{duffie1992stochastic}, \citeauthor*{buckdahn2010probabilistic} \cite{buckdahn2010probabilistic}.} attempt to obtain a probabilistic representation for the above was carried out by \citeauthor*{cheridito2007second} \cite{cheridito2007second}. They connected the PDE \eqref{eq_PDE_nonlinear} to the BSDE
\begin{align}\label{eq_bsde_gamma}
	Y_t &= g(X_T) + \int_t^T f\big(s,X_s,Y_{s},Z_s, \Gamma_s\big) \mathrm{d}s - \int_t^T Z_s \circ\mathrm{d} X_s, \;
	\d\langle Z, X \rangle_s = \Gamma_s \d s,
\end{align}
where `$\circ$' denotes the Fisk--Stratonovich integral. The process $\Gamma$ essentially corresponds to the Hessian matrix. It is because of the dependence on the process $\Gamma$ that the term `second-order BSDEs' was coined in \cite{cheridito2007second}.

\medskip
The approach described in the previous lines has a notable drawback. Specifically, the existence and uniqueness of solutions to \eqref{eq_bsde_gamma} could, in general, only be ensured through the existence of a smooth solution and the comparison principle for the associated PDE.\footnote{These limitations have been partially addressed in recent works by \citeauthor*{bouchard2019second} \cite{bouchard2019second} and \citeauthor*{bouchard2022understanding} \cite{bouchard2022understanding}.} Nonetheless, this approach provides two key insights. First, it underscores the need to reconsider the class of BSDEs under study, allowing for generators $f$ that explicitly depend on processes linked to the Hessian matrix. Second, it introduces the auxiliary stochastic target problems discussed in \cite{cheridito2007second}, which emerge in the course of identifying the appropriate class in which solutions ought to exist.

\medskip
The first systematic study of 2BSDEs was then carried out in a follow-up work by \citeauthor*{soner2012wellposedness} \cite{soner2012wellposedness,soner2013dual}; there, the authors established the well-posedness of a second-order extension of BSDEs, which led to probabilistic representations of the fully nonlinear PDE \eqref{eq_PDE_nonlinear}. The strong regularity assumptions imposed on the generator and terminal condition were lifted by \citeauthor*{possamai2018stochastic} \cite{possamai2018stochastic} through the use of selection theorems in a Brownian framework. 

\medskip
Naturally, the development of the 2BSDE theory followed paths similar to those of classical BSDEs, particularly in relaxing the Lipschitz-continuity assumption on the generator. Monotonicity conditions in the $y$-variable were introduced by \citeauthor*{possamai2013second1} \cite{possamai2013second1}. These were later extended by \citeauthor*{popier2019second} \cite{popier2019second}, who dropped the linear growth assumption and allowed for singular terminal conditions. A further contribution by \citeauthor*{o2020representation} \cite{o2020representation} employed a novel approach using $\mathbb{L}^p$-estimates for  $p \in [1,2)$  (instead of the traditional $\mathbb{L}^2$-estimates), ensuring the well-posedness of 2BSDEs under an extended monotonicity condition. Another research direction involves assuming super-linear growth in the $z$-variable; see \citeauthor*{possamai2013second} \cite{possamai2013second},  \citeauthor*{lin2016new} \cite{lin2016new}, and \citeauthor*{sheng2022quadratic} \cite{sheng2022quadratic}, which focused on the quadratic case. \citeauthor*{matoussi2013second} \cite{matoussi2013second, matoussi2021corrigendum} deal with a lower c\`adl\`ag obstacle on the first component, while \citeauthor*{matoussi2014second} \cite{matoussi2014second} address the case of two completely separated obstacles. For additional developments, see the PhD thesis of \citeauthor*{noubiagain2017equations} \cite{noubiagain2017equations} and the work of \citeauthor*{o2024wellposedness} \cite{o2024wellposedness}, which introduces a new formulation of reflected 2BSDEs based on non--Skorokhod-type minimality conditions. For more general constraints on the solution to 2BSDEs, we refer to \citeauthor*{fabre2012some} \cite[Chapter 3]{fabre2012some}. Further extensions also include time horizons that are stopping times, as in \citeauthor*{lin2020second} \cite{lin2020second}, time horizons that are random times, as in \citeauthor*{gennaro20252bsde} \cite{gennaro20252bsde}, or weaker integrability conditions, as in \citeauthor*{ren2022nonlinear} \cite{ren2022nonlinear}. Applications of 2BSDEs span several fields: zero-sum differential games in \citeauthor*{possamai2020zero} \cite{possamai2020zero}, optimal planning problems in \citeauthor*{ren2023entropic} \cite{ren2023entropic}, stochastic partial differential equations in \citeauthor*{matoussi2014probabilistic} \cite{matoussi2014probabilistic}, and mean-field games in \citeauthor*{barrasso2022controlled} \cite{barrasso2022controlled}. They have also been instrumental in studying Stackelberg games in \citeauthor*{hernandez2026closed} \cite{hernandez2026closed} and time-inconsistent stochastic control in \citeauthor*{hernandez2023me} \cite{hernandez2023me}.

\medskip
Numerical schemes with rigorous convergence proofs for approximating 2BSDE solutions have been provided by \citeauthor*{possamai2015weak} \cite{possamai2015weak} and \citeauthor*{ren2015convergence} \cite{ren2015convergence}. A natural continuation of \cite{possamai2015weak}, focusing on 2RBSDEs driven by a continuous canonical process, can be found in \citeauthor*{noubiagain2017equations} \cite[Chapter 5]{noubiagain2017equations}.\footnote{It is worth noting that crucial properties were assumed rather than proven in this work; see \cite[Remark 5.3.3]{noubiagain2017equations}.} More informal numerical schemes have also been explored in \citeauthor*{beck2019machine} \cite{beck2019machine}, \citeauthor*{pak2025nonequidistant} \cite{pak2025nonequidistant} and \citeauthor*{xiao2024numerical} \cite{xiao2024numerical}.

\medskip
Few works have ventured beyond the Brownian framework. Notably, \citeauthor*{kazi2015second} \cite{kazi2015second, kazi2015second2} addressed the well-posedness of 2BSDEs in a Brownian--Poisson framework under regularity assumptions similar to \cite{soner2012wellposedness,soner2013dual}. During the final stages of preparing this manuscript, a preprint by \citeauthor*{denis2024second} \cite{denis2024second} appeared, in which the strategy from \cite{possamai2018stochastic} was extended to 2BSDEs with jumps, as introduced in \cite{kazi2015second}. 

\medskip
We turn to the applications of 2BSDEs, which are primarily concentrated in two domains: model uncertainty in mathematical finance and contract theory. In mathematical finance, \citeauthor*{matoussi2015robust} \cite{matoussi2015robust} study utility maximisation under volatility uncertainty, while \cite{matoussi2013second} and \citeauthor*{matoussi2014second} \cite{matoussi2014second, matoussi2021corrigendum} address super-hedging problems for American and game options under volatility uncertainty in incomplete and nonlinear markets (see also \cite{possamai2018stochastic}). Additionally, 2BSDEs naturally arise in second-order stochastic target problems, as explored in \cite{bouchard2019second,soner2013dual}. An optimal liquidation problem under model uncertainty has also been studied in \cite{popier2019second} and \cite[Chapter 2]{fabre2012some}. In contract theory, 2BSDEs formed the foundation for the first comprehensive treatment of continuous-time principal-agent problems with moral hazard by \citeauthor*{cvitanic2017moral} \cite{cvitanic2017moral, cvitanic2018dynamic}; extensions to random horizons have been considered by \citeauthor*{lin2022random} \cite{lin2022random}.\footnote{\citeauthor*{chiusolo2026new} \cite{chiusolo2026new} have recently demonstrated that the 2BSDE framework can be bypassed in this context.} These ideas have since been applied to incentive structures in electricity markets by \citeauthor*{aid2022optimal} \cite{aid2022optimal}, \citeauthor*{aid2023principal} \cite{aid2023principal}, and \citeauthor*{elie2021mean} \cite{elie2021mean}, to green bond markets by \citeauthor*{baldacci2022governmental} \cite{baldacci2022governmental}, and to asset pricing by \citeauthor*{cvitanic2018asset} \cite{cvitanic2018asset}. Ambiguity aversion on volatility was studied in \citeauthor*{hernandez2019moral} \cite{hernandez2019moral}, \citeauthor*{mastrolia2015moral} \cite{mastrolia2015moral}, and \citeauthor*{sung2022optimal} \cite{sung2022optimal}. 2BSDEs can also be applied to optimal contracting problems with (multi-)hierarchical structures, see \citeauthor*{hubert2023continuous} \cite{hubert2023continuous}.
	
\medskip
As mentioned earlier, one of the major applications of 2BSDEs is in problems involving models with uncertain volatility. The foundational $G$-expectation framework is of particular relevance in this area; see \citeauthor*{peng2019nonlinear} \cite{peng2019nonlinear}. Both $G$-expectation theory and 2BSDE theory share a common foundational influence: the work of \citeauthor*{denis2006theoretical} \cite{denis2006theoretical}. However, their approaches were different. The theory of $G$-expectation builds on solutions to PDEs, which require bounds on volatility uncertainty, making the corresponding results more restrictive compared to 2BSDEs; see \cite[Section 3.3]{soner2012wellposedness} and \Cref{sec:GBSDE}. Nonetheless, the $G$-expectation theory has progressed immensely. For example, a complete $G$--stochastic calculus has emerged, with fundamental results analogous to those in (classical) stochastic calculus. Moreover, \citeauthor*{soner2011martingale} \cite{soner2011martingale} proved a martingale representation result under $G$-expectation, complementing earlier partial results by \citeauthor*{xu2009martingale} \cite{xu2009martingale} and \citeauthor*{song2011some} \cite{song2011some} for symmetric $G$-martingales; see also \citeauthor*{peng2014complete} \cite{peng2014complete}. Naturally, $G$-BSDEs followed; see \citeauthor*{hu2014backward} \cite{hu2014backward,hu2014comparison} and \citeauthor*{liu2019multi} \cite{liu2019multi}. Further results include a $G$--Doob--Meyer decomposition by \citeauthor*{chen2015g} \cite{chen2015g}, Jensen’s inequality by \citeauthor*{hu2015jensen} \cite{hu2015jensen}, and a study on $G$--Sobolev spaces by \citeauthor*{peng2015g} \cite{peng2015g}. Recent works have relaxed the Lipschitz-continuity assumption on generators in $G$-BSDEs. For instance, \citeauthor*{hu2020bsde} \cite{hu2020bsde} considered time-varying Lipschitz-continuous generators, \citeauthor*{zhang2021solutions} \cite{zhang2021solutions} examined $G$-BSDEs whose generator is not Lipschitz in the $y$-variable, and \citeauthor*{wang2021bsdes} \cite{wang2021bsdes} considered uniformly continuous generators. Another area in $G$-BSDE theory involves obstacle problems, motivated by their potential applications. For example, \citeauthor*{yang2017multi} \cite{yang2017multi} studied $G$-BSDEs with a sub-differential operator, while subsequent works examined classical cases of reflection: lower obstacle in \citeauthor*{li2018reflected} \cite{li2018reflected}, upper obstacle in \citeauthor*{li2020reflected} \cite{li2020reflected}, double obstacles with approximate Skorokhod condition in \citeauthor*{li2020doubly} \cite{li2020doubly}, and mean reflection in \citeauthor*{liu2019bsde} \cite{liu2019bsde}. Further extensions were studied by \citeauthor*{he2024mean} \cite{he2024mean}, \citeauthor*{li2022reflected} \cite{li2022reflected}, and \citeauthor*{li2024multi} \cite{li2024multi}. Generators with quadratic growth in the $z$-variable were also considered in \citeauthor*{hu2018quadratic} \cite{hu2018quadratic}, \citeauthor*{hu2022quadratic} \cite{hu2022quadratic}, and \citeauthor*{sun2024g} \cite{sun2024g}. Connections to nonlinear PDEs were explored in \citeauthor*{hu2018stochastic} \cite{hu2018stochastic} and \citeauthor*{hu2024bsde} \cite{hu2024bsde}, and generators with mean-field terms were studied in \citeauthor*{sun2020mean} \cite{sun2020mean2,sun2020mean,sun2025mean}. Other variations include ergodic $G$-BSDEs in \citeauthor*{hu2018ergodic} \cite{hu2018ergodic}, stochastic recursive optimal control problems in \citeauthor*{hu2017dynamic} \cite{hu2017dynamic}, and large deviation principles for the $Y$-component of $G$-BSDE solutions in \citeauthor*{dakaou2021large} \cite{dakaou2021large}.
	
\medskip
The literature on $G$-BSDEs is essentially limited to BSDEs driven by $G$--Brownian motion. To the best of our knowledge, the only work that extends the $G$-calculus framework to processes with jumps is \citeauthor*{hu2021g} \cite{hu2021g}, which introduces the so-called $G$--Lévy processes. However, this work exclusively deals with jumps of finite variation, and subsequent research on $G$--Lévy processes remains confined to this specific setting. Slightly more relevant to the topic of $G$-BSDEs are the unpublished works by \citeauthor*{paczka2014gmartingale} \cite{paczka2014gmartingale, paczka2014ito}, which investigate $G$--martingale representations in a jump setting. These can be considered as `proto-BSDEs' with a zero-generator. The lack of further development in this area stems from intrinsic restrictions in the construction of the sublinear $G$-expectation operator, which relies heavily on the use of PDEs. Interestingly, these limitations were in part addressed by a probabilistic approach inspired by the construction of nonlinear L\'evy processes by \citeauthor{neufeld2016nonlinear} \cite{neufeld2016nonlinear}. The construction is based on results by \citeauthor*{neufeld2014measurability} \cite{neufeld2014measurability}, who show that one can establish versions of the semi-martingale characteristics which are measurable in the probability parameter $\P$. In our companion paper \cite{possamai2026aggregation}, and in \Cref{sec_2bsde_X} below, new but related measurability issues arise, which are addressed through adaptations of \cite{neufeld2014measurability}.
	
\medskip
The aim of this work is to unify the ideas of \cite{neufeld2014measurability,neufeld2016nonlinear,possamai2024reflections} to establish a general well-posedness result for 2BSDEs. By intertwining these approaches with the techniques in \cite{nutz2013constructing,possamai2018stochastic}, we relax the strong continuity assumptions on the terminal condition and generator from \cite{kazi2015second,kazi2015second2}, while also departing from the semi-martingale characteristics associated with the quasi-left-continuous case. To put our efforts into context, a few comments are necessary. At first glance, it may seem that our work is merely a straightforward generalisation of the 2BSDE theory for continuous processes, especially with regard to the main hierarchy used to prove well-posedness:
	\begin{enumerate}
		\item[$(i)$] define the appropriate candidate value function corresponding to the supremum of conditional expectations of BSDEs;
		\item[$(ii)$] prove that the value function admits left- and right-hand side limits on a dense subset of $[0,T]$, outside a negligible set;
		\item[$(iii)$] show that the path-regularisation is a nonlinear super-martingale;
		\item[$(iv)$] use reflected BSDEs to find the decomposition of the regularisation;
		\item[$(v)$] characterise the regularisation and its decomposition as the unique solution to a system which resembles 2BSDEs as introduced in the literature.
	\end{enumerate}
	However, there are substantial technical obstacles that need to be overcome for the programme above to go through. In particular, this is due to the fact that the processes we are able to accommodate can be of jump--diffusion, pure-jump, or even discrete-time nature.
	
	\medskip
	The key challenges with regard to the above programme are the following. Steps $(i)$--$(iii)$ were established in our companion work \cite{possamai2026aggregation} and necessitated novel measurability results for integrands of stochastic integrals and for integrands of stochastic integrals of compensated random measures, as well as stability results for BSDEs.
	In step $(iv)$, we apply the well-posedness result for reflected BSDEs from \cite{possamai2024reflections} to derive the decomposition of the regularised value function. Finally, in step $(v)$, which is the main goal of this work, we provide two complementary characterisations of the regularised value function and its decomposition. We refer to them as the \textit{extrinsic} and \textit{intrinsic} 2BSDE systems. The former requires auxiliary BSDEs in the definition, while the latter does not and coincides in spirit with the usual notion of 2BSDEs in the literature, but requires additional assumptions.

	\medskip
	Nevertheless, when comparing the notion of 2BSDE in this work (either in the intrinsic or extrinsic sense) to \cite{denis2024second,kazi2015second,kazi2015second2}, one notices that we consider a family of integrands $({U}^\P)_{\P\in\fP_\smalltext{0}}$ for the stochastic integrals of the compensated jump measures, rather than a single integrand ${U}$ that acts as an aggregator of the former family. The claim that the family $({U}^\P)_{\P\in\fP_\smalltext{0}}$ can be defined independently of the underlying probability measures $\P\in\fP_0$ appears in the literature, but, to the best of our knowledge, lacks a rigorous proof; compare the discussion following \cite[Remark 4.8]{kazi2015second2} or the proof of \cite[Lemma 2.11]{denis2024second}. We were not able to confirm this claim and doubt that the aggregation of these integrands is possible in general. Our doubts are reinforced by the fact that, when considering semi-martingale characteristics, the only term that can be defined independently of the probability measure is the diffusion term, which is related only to the integrand of the usual continuous It\^o--stochastic integral in the dynamics of the 2BSDE. The same issue does not appear in the recent work of \citeauthor*{gennaro20252bsde} \cite{gennaro20252bsde} because the particular 2BSDE with jumps considered there can be reduced to a 2BSDE driven by continuous processes; see \cite[Theorem~1]{gennaro20252bsde}. Lastly, the proofs in \cite{denis2024second} primarily rely on earlier arguments from the literature, without addressing or correcting the issues we have identified in those previous proofs. These issues are resolved in the present work and in \cite{possamai2026aggregation}.
	
	\medskip
	One might naturally wonder whether the issue of aggregation can be resolved by considering a different class of 2BSDEs than the one in \Cref{eq_2BSDE_intro}, namely those that are not driven by the continuous part $X^{c,\P}$ of $X$ and its compensated jump measure $\tilde{\mu}^{X,\P}$, but by $X$ itself, provided the latter is an appropriate (local) martingale. However, the assumptions required in a genuine 2BSDE setting, where the generator is not identically zero and depends on the semi-martingale characteristics of $X$, imply that $X$ is $\P$--a.s. continuous. We consider these assumptions to be minimal, and we are not aware of any weaker conditions in a jump setting under which aggregation of the integrand with respect to $X$ would be possible. As a result, any alternative decomposition of $X \leadsto (X^\textnormal{I},\mu^\textnormal{J})$ into an It\^o integrator $X^{\textnormal{I}}$ and a jump measure $\mu^{\textnormal{J}}$, generating a compensated jump measure $\tilde{\mu}^{\textnormal{J},\P}$, does not appear to resolve the issue either. As before, $X^{\textnormal{I}}$ is again forced to be continuous in a genuine 2BSDE setting, and the integrand with respect to $\tilde{\mu}^{\textnormal{J},\P}$ still cannot be aggregated. However, we do not exclude the possibility that the aggregation issue for such integrands in a jump setting may be resolved through alternative tools or methods; our claim is merely that this does not seem possible with the techniques currently available.

	\medskip
	Second-order BSDEs have natural applications to control problems involving uncertainty. In such cases, the value of the problem with uncertainty can usually be associated with the first component $Y$ of the solution to the 2BSDE, and an optimal control can be constructed by maximising the generator along the solution of the 2BSDE. If all the terms inside the generator do not depend on the probability measure $\P$, then the control $\alpha^\star$ constructed in this way will be robust in the sense that, regardless of the realisation of the probability law $\P$, $\alpha^\star$ will be the best response to the uncertainty. When the uncertainty lies in the characteristic triplet of the semi-martingale $X$, but the control affects only the drift of this process, the Hamiltonian of the problem, which corresponds to the generator of the associated 2BSDE, can, in many cases, be independent of the probability measure if it depends only on the $Z$-integrand. However, as soon as the control affects the compensator of the jump measure $\mu^X$, the generator of the 2BSDE will, in general, also depend on the $U$-integrand. Since this leads to the aggregation problem outlined previously, the optimal control $\alpha^\star$, if it can be constructed by maximising the generator, will depend on the probability law $\P$, and is thus not robust in the sense described previously. 
	
	\medskip
	In contract theory, and in particular in principal--agent problems, similar issues appear. The reduction method, which uses (2)BSDEs to reduce this layered Stackelberg-type game to a more conventional stochastic control problem, currently seems to rely, in the case of continuous processes, on the fact that the solution of a particular 2BSDE does not depend on the probability law $\P$; compare with the proof of \cite[Theorem 3.6]{cvitanic2018dynamic}. However, whether aggregation can be circumvented, and whether optimal contracts by the principal or optimal responses by the agent can be described by 2BSDEs without the need to aggregate (particularly when introducing jumps, since this seems to cause issues) will be part of our future research endeavours.

\medskip
\textbf{Outline:} \Cref{sec_preliminaries} introduces the setting and fixes the notation and terminology. In \Cref{sec_wellposed_2bsde}, we present our main results, namely the well-posedness of our 2BSDE systems with jumps, along with certain key properties. The proofs are given in \Cref{sec_proofs}. \Cref{sec_2bsde_X} introduces a variant of our 2BSDE with jumps: there is no random measure, but the driving It\^o-integrator may jump. We discuss the difficulties of aggregating the integrand of the compensated random measure in \Cref{sec_aggregation_problems}. Lastly, we make a comparison with $G$-BSDEs in \Cref{sec:GBSDE}.

\medskip
{\small\textbf{\small Notation and terminology:} we follow \cite{possamai2026aggregation}. We fix a positive integer $d$. Let $\N$ and $\R$ denote the nonnegative integers and real numbers, respectively. Let $\D_\smallertext{+} = \{k 2^{-n}:k,n\in\N\}$ be the collection of nonnegative dyadic numbers. We write $\S^d_\smallertext{+}$ for the set of positive semi-definite, symmetric and real $d \times d$ matrices. Points in $\R^d$ are understood to be column vectors. For a matrix $A$, we denote its transpose by $A^\top$, its Moore–Penrose (pseudo-)inverse by $A^\oplus$, and, in the case that $A$ is a square matrix, we denote its trace by $\textnormal{Tr}[A]$. For $(a, b) \in [-\infty,\infty]^2$, we write $a \lor b \coloneqq \max\{a,b\}$, $a \land b \coloneqq \min\{a,b\}$, $a^+ \coloneqq a \lor 0$ and $a^- \coloneqq (-a)^+$. 
		
\smallskip
For a c\`adl\`ag function $X$ defined on some interval $I \subseteq [0,\infty]$ containing $0$, we set $\Delta X_0 \coloneqq 0$ and $\Delta X_t \coloneqq X_t - X_{t-}$ for $t \in I\setminus\{0\}$. For $t \in (0,\infty]$, a limit of the form $s \uparrow\uparrow t$ means that $s \longrightarrow t$ along $s < t$. For $t \in [0,\infty)$, we define  $s \downarrow\downarrow t$ analogously. We use the conventions $\sup\varnothing = - \infty$ and $0 / 0 \coloneqq 0$, if not stated otherwise. 
		
\smallskip
We write $\mu \ll \nu$ if the measure $\mu$ is absolutely continuous with respect to the measure $\nu$. For a family of measures $\mathfrak{M}$ on the same underlying measurable space $(\Omega,\cF)$, a property $P$ holds $\mathfrak{M}$--quasi-surely (abbreviated as $\mathfrak{M}$--q.s.) if the subset of $\Omega$ on which $P$ does not hold is an  $(\cF,\P)$--null set for every $\P\in\mathfrak{M}$; a subset $A \subseteq \Omega$ is an $(\cF,\P)$--null set, if there exists $B \in \cF$ such that $A \subseteq B$ and $\P[B] = 0$. Then $\xi : \Omega \longrightarrow [-\infty,\infty]$ is $\mathfrak{M}$--essentially bounded if there exists $\mathfrak{C} \in (0,\infty)$ such that $|\xi| \leq \mathfrak{C}$ holds $\mathfrak{M}$--quasi-surely. We always use the convention $\infty - \infty \coloneqq -\infty$. In particular, if $\xi : \Omega \longrightarrow [-\infty,\infty]$ is $\cF$-measurable, then $\E^\P[\xi] \coloneqq \E^\P[\xi\lor 0] - \E^\P[(-\xi) \lor 0] = -\infty$ in case $\E^\P[(-\xi) \lor 0] = \infty$.
		
\smallskip
Let $(\Omega,\cF)$ and $(\Omega^\prime,\cF^\prime)$ be measurable spaces. A kernel $K$ on $(\Omega^\prime,\cF^\prime)$ given $(\Omega,\cF)$ is a map $K : \Omega \times \cF^\prime \longrightarrow [0,\infty]$ such that $\omega \longmapsto K_\omega(A)$ is $\cF$-measurable for every $A \in \cF^\prime$ and $K_\omega(\cdot)$ is a measure on $(\Omega^\prime,\cF^\prime)$ for each $\omega \in \Omega$; if $K_\omega(\Omega^\prime) = 1$ holds for every $\omega \in \Omega$, then $K$ is a stochastic kernel.
		
\smallskip
An integral never includes the points $0$ or $\infty$ in the domain of integration, and $\int_a^b$ is always understood as $\int_{(a,b]}$ and $\int_{a\smallertext{-}}^b$ as $\int_{[a,b]}$.}

\section{Preliminaries}\label{sec_preliminaries}

We adopt the setup and conventions of \cite{possamai2026aggregation} and recall what is needed.

\subsection{Setting and semi-martingale laws}\label{sec::semimartingale_laws}

Let $\Omega = D_0([0,\infty);\R^d)$ be the space of all c\`adl\`ag paths $\omega : [0,\infty) \longrightarrow \R^d$ with $\omega_0 = 0$ endowed with the Skorokhod topology. We write $\cF$ for the corresponding Borel $\sigma$-algebra. We denote by $\F = (\cF_t)_{t \in [0,\infty)}$ the filtration generated by the canonical process $X = (X_t)_{t \in [0,\infty)}$ on $\Omega$. Then $\cF = \cF_\infty  \coloneqq \sigma(\cup_{t\in[0,\infty)}\cF_t)$; see \cite[Theorem VI.1.14]{jacod2003limit}. We denote by $\F_\smallertext{+}$ the right-continuous and by $\F_\smallertext{-}$ the left-continuous modification of $\F$, respectively. We use the conventions $\cF_{0\smallertext{-}}=\cF_{(0\smallertext{+})\smallertext{-}} = \{\varnothing,\Omega\}$. The predictable $\sigma$-algebra on $\Omega\times[0,\infty)$ relative to $\F$, and denoted $\cP(\F)$, is generated by all real-valued, $\F_\smallertext{-}$-adapted processes which are left-continuous on $(0,\infty)$. For a probability measure $\P$ on $(\Omega,\cF)$, we let $\sN^\P$ be the collection of all $(\cF,\P)$--null sets. We write $\F^\P = (\cF^\P_t)_{t\in[0,\infty)}$ and $\F^\P_\smallertext{+} = (\cF^\P_{t\smallertext{+}})_{t\in[0,\infty)}$ for the augmentations of $\F$ and $\F_\smallertext{+}$ relative to $\P$, respectively, that is, $\cF^\P_{t} \coloneqq \sigma(\cF_t,\sN^\P)$ and $\cF^\P_{t\smallertext{+}} \coloneqq \sigma(\cF_{t\smallertext{+}},\sN^\P)$, and we use the convention $\cF^{\P}_{0\smallertext{-}} = \cF^{\P}_{(0\smallertext{+})\smallertext{-}} = \sigma(\sN^\P)$. 
Then also $\cP(\F^\P) = \cP(\F^\P_{\smallertext{+}})$. We write $\cP \coloneqq \cP(\F)$ and 
$\widetilde{\cP} \coloneqq \cP(\F)\otimes\cB(\R^d)$ for simplicity.
Moreover, we say that a property holds outside a $\P$--null set (equivalently, $\P$--a.s.\ or for $\P$--a.e. $\omega \in \Omega$) if it holds outside an $(\cF,\P)$--null set.

\medskip
We denote by $\fP(\Omega)$ the set of probability measures on $(\Omega,\cF)$, and endow $\fP(\Omega)$ with the topology of weak convergence. Since $\Omega$ is Polish, so is $\fP(\Omega)$, and the Borel $\sigma$-algebra $\cB(\fP(\Omega))$ is the $\sigma$-algebra generated by all functions $\fP(\Omega) \ni \P \longmapsto \P[A] \in [0,1]$, $A \in \cB(\Omega)$; see \cite[Propositions 7.23 and 7.25]{bertsekas1978stochastic}. 
We write $\fP_\textnormal{sem}$ for the Borel set in $\fP(\Omega)$ (see \cite[Theorem 2.5]{neufeld2014measurability}) of all probability measures $\P\in\fP(\Omega)$ for which $X$ is an $(\F,\P)$--semi-martingale; $X$ is a semi-martingale relative to $(\F,\P)$ if and only if it is one relative to $(\F_\smallertext{+},\P)$, which is equivalent to $X$ being an $(\F^\P_\smallertext{+},\P)$--semi-martingale (see \cite[Proposition~2.2]{neufeld2014measurability}).
	
	\medskip
	Fix $t \in [0,\infty)$. The concatenation $\omega \otimes_t \tilde\omega$ of two paths $(\omega,\tilde\omega) \in \Omega \times \Omega$ is defined by
	\begin{equation*}
		(\omega \otimes_t \tilde\omega)_s \coloneqq \omega_s \1_{\{0 \leq s < t\}} + (\omega_{t} + \tilde\omega_{s-t})\1_{\{t \leq s\}}, \; s \in [0,\infty),
	\end{equation*}
	and for a map $\xi$ on $\Omega$, we write
	\[
		\xi^{t,\omega}(\tilde{\omega}) \coloneqq \xi(\omega\otimes_t \tilde{\omega}), \; \tilde{\omega} \in \Omega.
	\]
	For $\P \in \fP(\Omega)$, there exists a ($\P$--a.s. unique) regular conditional probability distribution $(\P^t_\omega)_{\omega \in \Omega}$ of $\P$ given $\cF_t$ that satisfies
	\begin{equation*}
		\P^t_\omega\big[X_{\cdot \land t} = \omega_{\cdot\land t}\big] = 1, \; \omega \in \Omega.
	\end{equation*}
	Then, we define $\P^{t,\omega} \in \fP(\Omega)$ by
	\[
		\P^{t,\omega}[A] \coloneqq \P^{t}_{\omega}[\omega\otimes_t A], \; A \in \cF,
	\]
	where $\omega\otimes_t A \coloneqq \{\omega\otimes_t\tilde\omega : \tilde\omega \in A\}$. Then $\E^{\P^{t,\omega}}[\xi^{t,\omega}] = \E^{\P^\smalltext{t}_\smalltext{\omega}}[\xi]$, $\omega \in \Omega$.
	For a process $Z$ indexed by $[0,\infty)$ or $[0,\infty]$, and $\omega \in \Omega$, we define the process $Z^{t,\omega}_{t\smallertext{+}\smallertext{\cdot}} \coloneqq (Z_{t\smallertext{+}\smallertext{\cdot}})^{t,\omega}$ by 
	\[
	(Z^{t,\omega}_{t\smallertext{+}\smallertext{\cdot}})_r(\tilde\omega) \coloneqq Z_{t\smallertext{+}r}(\omega\otimes_t \tilde\omega), \; (\tilde\omega,r) \in \Omega \times [0,\infty).
	\]

	\medskip
	We fix a bounded Borel function $h: \R^d \longrightarrow \R^d$ with $h(x) = x$ in an open neighbourhood of the origin. Let $\P \in \fP_\textnormal{sem}$. We denote by $(\mathsf{B}^{\P},\mathsf{C}^{\P},\nu^{\P})$ the $(\F,\P)$--semi-martingale characteristics of $X$ relative to $h$. 
	To be precise, this triplet consists of processes $\mathsf{B}^{\P} = (\mathsf{B}^\P_t)_{t \in [0,\infty)}$ and $\mathsf{C}^{\P} = (\mathsf{C}^\P_t)_{t \in [0,\infty)}$, and a collection $\nu^\P = \{\nu^\P(\omega;\d t, \d x): \omega \in \Omega\}$ of (nonnegative) measures on $[0,\infty)\times\R^d$ such that, $\P$--a.s., $\mathsf{B}^{\P}$ coincides with the predictable finite variation component in the canonical decomposition of $X - \sum_{s \in (0,\cdot]}(\Delta X_s - h(\Delta X_s))$, which is a special semi-martingale, $\mathsf{C}^\P$ coincides with the predictable quadratic covariation of the continuous local martingale part of $X$, and $\nu^\P$ is equal to the compensator of the jump measure $\mu^X$ of $X$. The characteristics of $X$ with respect to $\F$, $\F_\smallertext{+}$, and $\F^\P_\smallertext{+}$ coincide outside some $\P$--null set by \cite[Proposition~2.2]{neufeld2014measurability}.

	\medskip
	The following assumption holds throughout this work.

	\begin{standing_assumption}
		$C = (C_t)_{t \in [0,\infty)}$ is a real-valued, right-continuous, non-decreasing, and $\F$-predictable process starting at zero.
	\end{standing_assumption}

	For $s \in [0,\infty)$, we define
	\begin{equation*}
		\widehat{\Omega}^C_s \coloneqq \big\{(\omega,\P) \in \Omega \times \fP_\textnormal{sem} : (\mathsf{B}^{\P},\mathsf{C}^{\P},\nu^{\P}) \ll (C^{s,\omega}_{s\smallertext{+}\smallertext{\cdot}}-C_{s}(\omega)), \; \text{$\P$--a.s.} \big\}.
	\end{equation*}
	For $\nu^\P$, this means $(|x|^2 \land 1)\ast\nu^\P \ll (C^{s,\omega}_{s\smallertext{+}\smallertext{\cdot}}-C_{s}(\omega))$ up to a $\P$--null set. By \cite[Lemma A.7]{possamai2026aggregation}, this is equivalent to $\nu^{\P}(\d t, \d x)= \mathsf{K}^{s,\omega,\P}_{\cdot,t}(\d x)\d (C^{s,\omega}_{s\smallertext{+}\smallertext{\cdot}}-C_{s}(\omega))_t$, $\P$--a.s., where $\mathsf{K}^{s,\omega,\P}$ is a kernel on $(\R^d,\cB(\R^d))$ given $(\Omega\times[0,\infty),\cP)$. For $(\omega,\P)\in\widehat{\Omega}^C_s$, the corresponding derivatives $(\mathsf{b}^{s,\omega,\P},\mathsf{a}^{s,\omega,\P},\mathsf{K}^{s,\omega,\P})$ with respect to $C^{s,\omega}_{s\smallertext{+}\smallertext{\cdot}}-C_{s}(\omega)$ are $\P \otimes \d (C^{s,\omega}_{s\smallertext{+}\smallertext{\cdot}}-C_{s}(\omega))$--a.e. uniquely defined; they are the $(\F,\P)$--differential characteristics of $X$ relative to $C^{s,\omega}_{s\smallertext{+}\smallertext{\cdot}}-C_{s}(\omega)$. 

	\medskip
	As is the case for the differential characteristics with respect to Lebesgue measure in \cite[Theorem~2.6]{neufeld2014measurability}, the above differential characteristics can be chosen measurable in the probability parameter $\P$; see \cite[Lemma 2.4]{possamai2026aggregation}. The codomain of $(\mathsf{b}^{s,\omega,\P},\mathsf{a}^{s,\omega,\P})$ is $\R^d\times\S^d_\smallertext{+}$, where $\S^d_\smallertext{+}$ denotes the collection of all real, positive semi-definite $d\times d$ matrices. The codomain for $\mathsf{K}^{s,\omega,\P}$ is $\cL$, defined as follows. Let $\frM(\R^d)$ be the collection of all (nonnegative) measures on $(\R^d,\cB(\R^d))$, and let
	\begin{equation*}
		\cL\coloneqq\bigg\{\kappa \in \frM(\R^d) : \int_{\R^\smalltext{d}}(|x|^2 \land 1) \kappa(\d x) < \infty \; \text{and} \; \kappa(\{0\}) = 0\bigg\}
	\end{equation*}
	be the set of L\'evy measures.
	We endow $\cL$ with the Borel structure described in \cite[Section 2.1]{neufeld2014measurability}.

\subsection{BSDE data and weighted spaces}\label{sec::data}
	
For every $(\omega,s) \in \Omega \times [0,\infty)$, we fix
	\[
		\fP(s,\omega) \subseteq \big\{\P\in \fP_\textnormal{sem} : (\omega,\P)\in\widehat{\Omega}^C_s \big\},
	\]
	and we suppose that
	\begin{equation*}
		\fP(s,\omega) = \fP(s,\omega_{\cdot\land s}), \; s \in [0,\infty),\; \omega\in\Omega.
	\end{equation*}
	Since $\omega_0 = 0$ for every $\omega \in \Omega$, the collection $\fP(0,\omega)$ is independent of $\omega \in \Omega$ and we therefore denote it by $\fP_0$. We suppose that $\fP_0 \neq \varnothing$.

	\medskip
	We now state the main conditions on the family $(\fP(s,\omega))_{(\omega,s)\in\Omega\times[0,\infty)}$ that are sufficient for our purposes; see \cite[Assumption 2.8]{possamai2026aggregation}, and also \cite[Assumption 2.1]{nutz2013constructing}, \cite[Condition (A)]{neufeld2016nonlinear}, and \cite[Condition (A)]{nutz2015robust}. Recall that a subset of a Polish space is analytic if it is the image of a continuous map defined on some Polish space.

	\begin{assumption}\label{ass::probabilities2}
		For all $(s,t) \in [0,\infty)^2$ with $s \leq t$, $\bar\omega \in \Omega$, and $\P \in \fP(s,\bar{\omega})$,
		\begin{enumerate}
			\item[$(i)$] $\{(\omega,\P^\prime) :\omega \in \Omega,  \; \P^\prime \in \fP(s,\omega)\} \subseteq \Omega \times \fP(\Omega)$ is analytic$;$
			\item[$(ii)$] $\P^{t-s,\omega} \in \fP(t,\bar{\omega}\otimes_s \omega)$ for {\rm$\P$--a.e.} $\omega\in\Omega;$
			\item[$(iii)$] if $\Q$ is a stochastic kernel on $(\Omega,\cF)$ given $(\Omega,\cF_{t\smallertext{-}s})$ such that $\Q(\omega) \in \fP(t,\bar{\omega} \otimes_s\omega)$ for {\rm$\P$--a.e.} $\omega \in \Omega$, then $\overline{\P}\in\fP(s,\bar{\omega})$, where
			\begin{align*}
				\overline{\P}[A] \coloneqq & \iint \big(\1_A\big)^{t\smallertext{-}s,\omega}(\omega^\prime)\Q(\omega;\d\omega^\prime)\P(\d\omega), \; A \in \cF.
			\end{align*}
		\end{enumerate}
	\end{assumption}

	\begin{example}\label{ex::differential_characteristics}
		\textnormal{\Cref{ass::probabilities2}} is satisfied, for example, when $C_t \coloneqq t$ and $\fP(s,\omega) \equiv \fP_\Theta$ consists of all semi-martingale measures $\P$ for which $X$ has differential characteristics $(\mathsf{b}_t,\mathsf{c}_t,\mathsf{K}_{\cdot,t})$ $($relative to Lebesgue measure$)$ with  $(\mathsf{b}_t,\mathsf{c}_t,\mathsf{K}_{\cdot,t}) \in\Theta$, \textnormal{$\P\otimes\d t$--a.e.}, for a Borel subset $\Theta \subseteq \R^d \times \S^d_\smallertext{+}\times \cL$$;$ see \textnormal{\cite[Theorem 2.1.$(i)$]{neufeld2016nonlinear}} and also \textnormal{\cite[Example~2.9]{possamai2026aggregation}}.  
	\end{example}

	For $(\omega,s,K) \in \Omega \times [0,\infty)\times\cL$ and a Borel-measurable map $\cU : \R^d \longrightarrow \R$, we write
	\begin{align*}
		\|\cU(\cdot)\|^2_{\hat{\L}^\smalltext{2}_{\smalltext{\omega}\smalltext{,}\smalltext{s}}(K)} 
		& \coloneqq\int_{\R^\smalltext{d}}\bigg(\cU(x) - \int_{\R^\smalltext{d}}\cU(x)K(\d x) \Delta C_s(\omega)\bigg)^2 K(\d x) + (1-K(\R^d)\Delta C_s(\omega))^+\bigg(\int_{\R^\smalltext{d}} \cU(x) K(\d x)\bigg)^2 \Delta C_s(\omega),
	\end{align*}
	and denote by $\widehat\L^2_{\omega,s}(K)$ the set of $\cB(\R^d)$-measurable maps $\cU : \R^d \longrightarrow \R$ with $\|\cU(\cdot)\|_{\hat{\L}^\smalltext{2}_{\smalltext{\omega}\smalltext{,}\smalltext{s}}(K)} < \infty$. 
	
	\medskip
	We formulate our main results under the following conditions on the data $(T,\xi,f)$. They appeared in \cite[Assumption 2.11]{possamai2026aggregation}, except for the stronger condition $(v)$, which we discuss below.

		\begin{assumption}\label{ass::generator2}
		$(i)$ $T$ is an $\F$--stopping time, and $\xi$ is real-valued and $\cF_T$-measurable.
		
		\medskip
		$(ii)$ The generator $f:\bigsqcup_{ (\omega,t,K) \in \Omega \times [0,\infty) \times \cL} \R \times \R \times \R^d \times \R^d \times \S^d_{\smallertext{+}} \times {\widehat\L}^2_{\omega,t}(K) \longrightarrow \R$ satisfies
		\begin{align*}
			&\big|f\big(t,\tilde\omega,y,\mathrm{y},z,u(\cdot),b,a,K\big) - f\big(t,\tilde\omega,y^\prime,\mathrm{y}^\prime,z^\prime,u^{\prime}(\cdot),b,a,K\big)\big|^2 \\
			& \leq r_t(\tilde\omega) |y-y^\prime|^2 + \mathrm{r}_t(\tilde\omega) |\mathrm{y}-\mathrm{y}^\prime|^2 + \theta^X_t(\tilde\omega) (z-z^\prime)^\top a(z-z^\prime) + \theta^\mu_t(\tilde\omega)  \|u(\cdot) - u^{\prime}(\cdot)\|^2_{{\hat\L}^\smalltext{2}_{\smalltext{\tilde\omega}\smalltext{,}\smalltext{t}}(K)}, \; (\tilde\omega,t) \in \llparenthesis 0,T\rrbracket,
		\end{align*}
		for all $(y,y^\prime,\mathrm{y},\mathrm{y}^\prime,z,z^\prime,u(\cdot),u^\prime(\cdot),b,a) \in \R^2 \times \R^2 \times (\R^d)^2 \times \big(\widehat\L^2_{\tilde\omega,t}(K)\big)^2 \times \R^d \times \S^d_\smallertext{+}$, for some $\F$-predictable, $[0,\infty)$-valued processes $(r,\mathrm{r},\theta^X,\theta^\mu) = (r_t,\mathrm{r}_t,\theta^X_t,\theta^\mu_t)_{t \in [0,\infty)}$.
		
		\medskip
		$(iii)$ For any $s \in [0,\infty)$ and any $\cB(\widehat{\Omega}^C_s)\otimes\widetilde{\cP}$-measurable function
		\[
			\widehat{\Omega}^C_s \times \Omega \times [0,\infty) \times \R^d \ni (\omega,\P,\tilde{\omega},r,x) \longmapsto u^\P_r(\tilde{\omega};x) \in \R,
		\]
		satisfying $u^\P_r(\tilde{\omega};\cdot) \in \widehat{\L}^2_{\omega\otimes_\smalltext{s}\tilde\omega,s+r}(\mathsf{K}^{s,\omega,\P}_{\tilde\omega,r})$, the function 
		\begin{equation*}
			\widehat\Omega^C_s \times \Omega \times [0,\infty) \times \R \times \R \times \R^d \times \R^d \times \S^d_\smallertext{+}\ni (\omega,\P,\tilde\omega,r,y,\mathrm{y},z,b,a) \longmapsto f\big(s+r,\omega\otimes_s\tilde\omega,y,\mathrm{y},z,u^\P_r(\tilde\omega;\cdot),b,a,\mathsf{K}^{s,\omega,\P}_{\tilde\omega,r}\big) \in \R
		\end{equation*}
		is Borel-measurable, and, for every fixed $(\omega,\P) \in \widehat\Omega^C_s$, the function
		\begin{equation*}
			\Omega \times [0,\infty) \times \R \times \R \times \R^d \times \R^d \times \S^d_\smallertext{+} \ni (\tilde\omega,r,y,\mathrm{y},z,b,a) \longmapsto f\big(s+r,\omega\otimes_s\tilde\omega,y,\mathrm{y},z,u^\P_r(\tilde\omega;\cdot),b,a,\mathsf{K}^{s,\omega,\P}_{\tilde\omega,r}\big) \in \R
		\end{equation*}
		is $\textnormal{Prog}(\F_\smallertext{+})\otimes\cB(\R)\otimes\cB(\R)\otimes\cB(\R^d)\otimes\cB(\R^d)\otimes\cB(\S^d_\smallertext{+})$-measurable.
		
		\medskip
		$(iv)$ There exists $\hat\beta \in (0,\infty)$ such that for all $(\omega,s) \in \Omega\times[0,\infty)$,
		\begin{align*}
			\sup_{\P\in\fP(s,\omega)}\E^\P \bigg[&\cE\big(\hat\beta A^{s,\omega}_{s\smallertext{+}\smallertext{\cdot}})_{(T\smallertext{-}s\land T)^{\smalltext{s}\smalltext{,}\smalltext{\omega}}}|\xi^{s,\omega}|^2 + \int_0^{(T\smallertext{-}s\land T)^{\smalltext{s}\smalltext{,}\smalltext{\omega}}} \cE(\hat\beta A^{s,\omega}_{s\smallertext{+}\smallertext{\cdot}})_r \frac{|f^{s,\omega,\P}_r(0,0,0,\mathbf{0})|^2}{|\alpha^{s,\omega}_r|^2} \d (C^{s,\omega}_{s\smallertext{+}\smallertext{\cdot}})_r\bigg] < \infty,
		\end{align*}
		where
		\[
			f^{s,\omega,\P}_r(\tilde\omega,0,0,0,\mathbf{0}) \coloneqq f\big(s+r,\omega\otimes_s\tilde\omega,0,0,0,\mathbf{0},\mathsf{b}^{s,\omega,\P}_r(\tilde\omega),\mathsf{a}^{s,\omega}_r(\tilde\omega),\mathsf{K}^{s,\omega,\P}_{\tilde{\omega},r}(\d x)\big),
		\]
		 $\alpha^2 \coloneqq\max\{\sqrt{r},\sqrt{\mathrm{r}},\theta^X,\theta^\mu\} > 0$, the process $A = (A_t)_{t \in [0,\infty)}$ given by $A_t \coloneqq \int_0^{t \land T} \alpha^2_s\d C_s$ is assumed to be finite-valued, and
		\begin{equation*}
			\cE(\hat\beta A)_r \coloneqq 
			\begin{cases}
				1, & r = 0, \\[0.5em]
				 \mathrm{e}^{\hat\beta (A_\smalltext{r}\smallertext{-}A_\smalltext{0})} \prod_{s \in (0,r]} (1+\hat\beta\Delta A_s)\mathrm{e}^{\smallertext{-}\hat\beta\Delta A_\smalltext{s}}, & r \in (0,\infty).
			\end{cases}
		\end{equation*}
		
		\medskip
		$(v)$ There exists $\Phi \in [0,1)$ such that $\Delta A \leq \Phi$ and $M^\Phi_1(\hat\beta) <1$, where
		\begin{equation}\label{eq::contraction_reflected_bsde}
			M_1^\Phi(\beta) \coloneqq \ff^\Phi(\beta) + \frac{1}{\beta} + \max\bigg\{1,\frac{(1+\beta\Phi)}{\beta}\bigg\}\bigg(\frac{5}{\beta} +\frac{4}{\beta}(1+\beta\Phi)^{1/2} + \beta\fg^\Phi(\beta)\bigg), \; \beta \in (0,\infty),
		\end{equation}
		with
		\begin{align*}
			\ff^\Phi(\beta) 
			\coloneqq \frac{4(1+\beta\Phi)}{\beta^2},\;
			\fg^\Phi(\beta)
			\coloneqq \frac{4}{\beta^2} \mathbf{1}_{\{\Phi = 0\}} + \frac{\Phi^2\sqrt{1+\beta\Phi}}{\Big(1+\beta\Phi - \sqrt{1+\beta\Phi}\Big)\Big(\sqrt{1+\beta\Phi} - 1\Big)}\mathbf{1}_{\{\Phi > 0\}}.
		\end{align*}
	\end{assumption}
	
\begin{remark}
	Note that $M^\Phi_1(\beta)$ is strictly decreasing in $\beta$. By \textnormal{\cite[Lemma B.2]{possamai2024reflections}}, it follows that $\lim_{\beta \uparrow \uparrow \infty} M^\Phi_1(\beta) = \max\{1, \Phi\} \Phi$. Therefore, for $\Phi < 1$, there exists a unique number $\beta^\star$ such that $M^\Phi_1(\beta^\star) = 1$.
\end{remark}

	For $(\omega,\P) \in \widehat{\Omega}^C_s$ and $u(\cdot) \in \widehat{\L}^2_{\omega\otimes_\smalltext{s}\tilde\omega,s+r}(\mathsf{K}^{s,\omega,\P}_{\tilde\omega,r})$, we write
	\begin{equation*}
		f^{s,\omega,\P}_r\big(\tilde\omega,y,\mathrm{y},z,u(\cdot)\big) \coloneqq f\big(s+r,\omega\otimes_s\tilde\omega,y,\mathrm{y},z,u(\cdot),\mathsf{b}^{s,\omega,\P}_r(\tilde\omega),\mathsf{a}^{s,\omega}_r(\tilde\omega),\mathsf{K}^{s,\omega,\P}_{\tilde{\omega},r}(\d x)\big).
	\end{equation*}
	If $s = 0$, we simply write $f^{\P}_r(\tilde\omega,y,\mathrm{y},z,u(\cdot))$.

	\begin{remark}
		$(i)$ The slightly awkward assumptions on our generator seem unavoidable when dealing with non-dominated families of measures$;$ see \textnormal{\cite[Remark~2.12]{possamai2026aggregation}} and compare with \textnormal{\cite[Hypothesis~3.1 and Remark~3.2]{confortola2013backward}}.

		\medskip
		$(ii)$ If we compare \textnormal{\Cref{ass::generator2}} with \textnormal{\cite[Assumption~2.11]{possamai2026aggregation}}, then the only difference is that in the latter, we assumed that $\widetilde M^\Phi_1(\hat\beta) <1$ instead of $M^\Phi_1(\hat{\beta}) < 1$, where
		\begin{equation*}
			\widetilde M^\Phi_1(\hat\beta) \coloneqq \ff^\Phi(\hat\beta) + \frac{1}{\hat\beta} + \max\bigg\{1,\frac{1+\hat\beta\Phi}{\hat\beta}\bigg\}\bigg(\frac{1}{\hat\beta} + \hat\beta\fg^\Phi(\hat\beta)\bigg).
		\end{equation*}
		The condition $\widetilde M^\Phi_1(\hat\beta) < 1$ ensures existence and uniqueness of solutions to the \textnormal{BSDEs} that appear in this work. Since the well-posedness proof of the \textnormal{2BSDE} in this work relies on a reflected \textnormal{BSDE} argument, we impose the stronger condition $M^\Phi_1(\hat\beta) < 1$$;$ see \textnormal{\cite[Theorem~3.4 and Remark~3.5]{possamai2024reflections}}. As $\widetilde{M}^\Phi_1(\hat{\beta}) \leq M^\Phi_1(\hat{\beta})$, we are still in the setting of \textnormal{\cite{possamai2026aggregation}}.
	\end{remark}
	
	We introduce the weighted solution spaces that already appear in \cite{possamai2024reflections}. Let $\cG$ be an arbitrary $\sigma$-algebra on $\Omega$ containing $\cF$, and let $\G = (\cG_t)_{t \in [0,\infty)}$ be an arbitrary filtration on $(\Omega,\cG)$ satisfying $\cF_t \subseteq \cG_t$ for every $t \in [0,\infty)$, and denote by $\cG_{0-}$ an additional $\sigma$-algebra contained in $\cG_0$. We let $\cG_\infty \coloneqq \sigma (\cup_{t \in [0,\infty)} \cG_t)$. Fix a probability measure $\P$ on $(\Omega,\cG)$, and let $M$ be an $\R^d$-valued, right-continuous, $\G$-adapted, $(\G,\P)$--locally square-integrable martingale. Let $(\omega,s) \in \Omega \times [0,\infty)$. Assume that $\langle M \rangle = \pi \bcdot (C^{s,\omega}_{s\smallertext{+}\smallertext{\cdot}}-C_{s}(\omega))$, $\P$--a.s., where $\pi = (\pi_t)_{t \in [0,\infty)}$ is an $\S^d_\smallertext{+}$-valued, $\G$-predictable process (see \cite[Appendix A]{possamai2026aggregation}), and that the $(\G,\P)$-compensator of the jump measure $\mu^X$ is $\P$--a.s. of the form $\nu^\P(\d t,\d x) = K_{t}(\d x)\d (C^{s,\omega}_{s\smallertext{+}\smallertext{\cdot}}-C_{s}(\omega))_t$ for some kernel $K$ on $(\R^d,\cB(\R^d))$ given $(\Omega \times [0,\infty),\cP(\G))$.
	Moreover, let $M^\P_{\mu^X}[W]\coloneqq\E^\P[W\ast\mu^X_\infty]$ for nonnegative measurable functions $W$ on $\Omega\times[0,\infty)\times\R^d$, where $W\ast \mu^X = \int_0^\cdot\int_{\R^d}W_s(x)\mu^X(\d s, \d x)$. We denote by $M^\P_{\mu^X}[W|\widetilde\cP(\G)]$ the `conditional expectation' of $W$ given $\widetilde{\cP}(\G) \coloneqq \cP(\G) \otimes \cB(\R^d)$ under the measure induced by $M^\P_{\mu^X}$; see \cite[Appendix A]{possamai2026aggregation} for details.

\medskip
For $\beta \in [0,\infty)$, we recall the following spaces defined in \cite[Section 2.2]{possamai2026aggregation}; we identify two elements if the seminorm of their difference is zero:

	\begin{itemize}[leftmargin=0.6cm]
		\item $\L^{2}_{T,\beta}(\cG,\P)$: Banach space of $\cG$-measurable random variables $\zeta : \Omega \longrightarrow \R$ satisfying 
		\[
		\|\zeta\|^{2}_{\L^{\smalltext{2}}_{\smalltext{T}\smalltext{,}\smalltext{\beta}}(\cG,\P)} \coloneqq \E^\P\Big[\big|\cE(\beta A)^{1/2}_{T}\zeta\big|^2\Big] < \infty;
		\]
		\item $\cH^{2,s,\omega}_{T,\beta}(\G,\P)$: Banach space of real-valued, right-continuous, $(\G,\P)$--square-integrable martingales $L = (L_t)_{t \in [0,\infty)}$ with $L = L_{\cdot\land (T\smallertext{-}s\land T)^{s,\omega}}$ and 
		\[
		\|L\|^2_{\cH^{\smalltext{2}\smalltext{,}\smalltext{s}\smalltext{,}\smalltext{\omega}}_{\smalltext{T}\smalltext{,}\smalltext{\beta}}(\G,\P)} \coloneqq \E^\P[L^2_0] + \E^\P\bigg[\int_0^{(T\smallertext{-}s \land T)^{\smalltext{s}\smalltext{,}\smalltext{\omega}}} \cE(\beta A^{s,\omega}_{s\smallertext{+}\smallertext{\cdot}})_r \d \langle L \rangle_r\bigg] < \infty;
		\]
		\item $\cT^{2,s,\omega}_{T,\beta}(\G,\P)$: Banach space of real-valued, right-continuous, $\P$--a.s. c\`adl\`ag, $\G$-adapted processes $Y = (Y_t)_{t \in [0,\infty]}$ with $Y = Y_{\cdot \land (T\smallertext{-}s\land T)^{s,\omega}}$, $(Y_t)_{t\in[0,\infty)}$ being $\G$-optional, and
		\begin{equation*}
			\|Y\|^2_{\cT^{\smalltext{2}\smalltext{,}\smalltext{s}\smalltext{,}\smalltext{\omega}}_{\smalltext{T}\smalltext{,}\smalltext{\beta}}(\G,\P)} 
			\coloneqq \sup_{\sigma \in \sT_{\smalltext{0}\smalltext{,}\smalltext{(}\smalltext{T}\tinytext{-}\smalltext{s}\smalltext{\land}\smalltext{T}\smalltext{)}^{\tinytext{s}\tinytext{,}\tinytext{\omega}}}(\G)}\E^\P\Big[ \big| \cE(\beta A^{s,\omega}_{s\smallertext{+}\smallertext{\cdot}})^{1/2}_\sigma Y_\sigma \big|^2\Big] < \infty,
		\end{equation*}
		where $\sT_{0,(T\smallertext{-}s\land T)^{\smalltext{s}\smalltext{,}\smalltext{\omega}}}(\G)$ denotes the collection of all $\G$--stopping times $\sigma$ satisfying $0 \leq \sigma \leq (T-s\land T)^{s,\omega}$;
		\item $\cS^{2,s,\omega}_{T,\beta}(\G,\P)$: Banach space of real-valued, right-continuous, $\P$--a.s. c\`adl\`ag, $\G$-adapted processes $Y = (Y_t)_{t \in [0,\infty]}$ with $Y = Y_{\cdot \land (T\smallertext{-}s\land T)^{\smalltext{s}\smalltext{,}\smalltext{\omega}}}$, $(Y_t)_{t \in [0,\infty)}$ being $\G$-optional, and
		\begin{equation*}
			\|Y\|^2_{\cS^{\smalltext{2}\smalltext{,}\smalltext{s}\smalltext{,}\smalltext{\omega}}_{\smalltext{T}\smalltext{,}\smalltext{\beta}}(\G,\P)} 
			\coloneqq \E^\P\bigg[ \sup_{r \in [0,(T\smallertext{-}s\land T)^{\smalltext{s}\smalltext{,}\smalltext{\omega}}]} \big| \cE(\beta A^{s,\omega}_{s\smallertext{+}\smallertext{\cdot}})^{1/2}_r Y_r \big|^2\bigg] < \infty;
		\end{equation*}
		\item $\H^{2,s,\omega}_{T,\beta}(\G,\P)$: Banach space of real-valued, $\G$-adapted processes $\phi = (\phi_t)_{t \in [0,\infty]}$ with $(\phi_t)_{t \in [0,\infty)}$ $\G$-optional, $\phi = \phi_{\cdot \land (T\smallertext{-}s\land T)^{s,\omega}}$, and
		\begin{equation*}
			\|\phi\|^2_{\H^{\smalltext{2}\smalltext{,}\smalltext{s}\smalltext{,}\smalltext{\omega}}_{\smalltext{T}\smalltext{,}\smalltext{\beta}}(\G,\P)} \coloneqq \E^\P\bigg[ \int_0^{(T\smallertext{-}s\land T)^{\smalltext{s}\smalltext{,}\smalltext{\omega}}} \cE(\beta A^{s,\omega}_{s\smallertext{+}\smallertext{\cdot}})_r |\phi_r|^2 \d (C^{s,\omega}_{s\smallertext{+}\smallertext{\cdot}})_r\bigg] < \infty;
		\end{equation*}
		\item $\H^{2,s,\omega}_{T,\beta}(M;\G,\P)$: Banach space of $\R^d$-valued, $\G$-predictable $Z = (Z_t)_{t \in [0,\infty)}$ with $Z = Z \mathbf{1}_{\llbracket 0, (T\smallertext{-}s\land T)^{\smalltext{s}\smalltext{,}\smalltext{\omega}} \rrbracket}$ and
		\begin{equation*}
			\|Z\|^2_{\H^{\smalltext{2}\smalltext{,}\smalltext{s}\smalltext{,}\smalltext{\omega}}_{\smalltext{T}\smalltext{,}\smalltext{\beta}}(M;\G,\P)} \coloneqq \E^\P\bigg[\int_0^{(T\smallertext{-}s\land T)^{s,\omega}} \cE(\beta A^{s,\omega}_{s\smallertext{+}\smallertext{\cdot}})_r Z^\top_r \pi_r Z_r \d (C^{s,\omega}_{s\smallertext{+}\smallertext{\cdot}})_r\bigg] < \infty;
		\end{equation*}
		\item $\H^{2,s,\omega}_{T,\beta}(\mu^X;\G,\P)$: Banach space of real-valued, $\widetilde{\cP}(\G)$-measurable functions $U$ with $U = U \mathbf{1}_{\llbracket 0, (T\smallertext{-}s\land T)^{\smalltext{s}\smalltext{,}\smalltext{\omega}} \rrbracket}$ and
		\begin{equation*}
			\|U\|^2_{\H^{\smalltext{2}\smalltext{,}\smalltext{s}\smalltext{,}\smalltext{\omega}}_{\smalltext{T}\smalltext{,}\smalltext{\beta}}(\mu^\smalltext{X};\G,\P)} \coloneqq \E^\P\bigg[\int_0^{(T-s\land T)^{s,\omega}} \cE(\beta A^{s,\omega}_{s\smallertext{+}\smallertext{\cdot}})_r \|U_r(\cdot)\|^2_{\hat\L^\smalltext{2}_{\smalltext{\omega}\smalltext{\otimes}_\tinytext{s}\smalltext{\cdot}\smalltext{,}\smalltext{s}\smalltext{+}\smalltext{r}}(K_\smalltext{r})} \d (C^{s,\omega}_{s\smallertext{+}\smallertext{\cdot}})_r\bigg] < \infty;
		\end{equation*}
		\item $\cH^{2,s,\omega,\perp}_{0,T,\beta}(M,\mu^X;\G,\P)$: subspace consisting of all martingales $N \in \cH^{2,s,\omega}_{T,\beta}(\G,\P)$ starting at zero with $\langle N, M\rangle^{(\P)} = 0$ and $M^\P_{\mu^\smalltext{X}}[\Delta N | \widetilde\cP(\G)] = 0$$.$
	\end{itemize}
	Since for $s = 0$ the spaces do not depend on $\omega$, we will omit $(s,\omega) = (0,\omega)$ from the notation in the above spaces and norms. We do this analogously for $\beta = 0$.

\section{Well-posedness of 2BSDEs with jumps}\label{sec_wellposed_2bsde}

We introduce two second-order BSDE systems, referred to as the extrinsic and intrinsic systems, the latter being in the spirit of \cite{possamai2018stochastic,soner2012wellposedness,soner2013dual} in the continuous state process case.
The main difference between the two systems is that, in the extrinsic one, uniqueness is ensured by introducing an auxiliary family $(\cY^\P(T,\xi))_{\P\in\fP_\smalltext{0}}$ corresponding to solutions to BSDEs. In the intrinsic one uniqueness can be established solely by imposing a condition on the solution tuple, under additional conditions on the data $(T,\xi,f)$. 

\medskip
We repeat here \cite[Assumption~3.3]{possamai2026aggregation}. The Lipschitz-continuity of the generator with respect to the $(\mathrm{y},z)$-variables allows us to write
\begin{align*}
	&f(t,\tilde\omega,y,\mathrm{y},z,u(\cdot),b,a,K) - f(t,\tilde\omega,y^\prime,\mathrm{y}^\prime,z^\prime,u^\prime(\cdot),b,a,K) \nonumber\\
	& \geq f(t,\tilde\omega,y,\mathrm{y},z,u(\cdot),b,a,K) - f(t,\tilde\omega,y^\prime,\mathrm{y},z,u(\cdot),b,a,K) \nonumber\\
	&\quad + \widehat\lambda^{\mathrm{y},\mathrm{y}^\smalltext{\prime}}_t(\tilde\omega)(\mathrm{y}-\mathrm{y}^\prime)  + (\eta^{z,z^\smalltext{\prime},a}_t(\tilde\omega))^\top a (z-z^\prime) + f(t,\tilde\omega,y^\prime,\mathrm{y}^\prime,z^\prime,u(\cdot),b,a,K) - f(t,\tilde{\omega},y^\prime,\mathrm{y}^\prime,z^\prime,u^\prime(\cdot),b,a,K),
\end{align*}
for $(\tilde{\omega},t) \in \llparenthesis 0, T \rrbracket$, $(b,a,K) \in \R^d \times \S^d_\smallertext{+}\times\cL$, and $(y,y^\prime,\mathrm{y},\mathrm{y}^\prime,z,z^\prime,u(\cdot),u^\prime(\cdot)) \in (\R)^4 \times (\R^d)^2 \times  (\widehat{\L}^2_{\tilde\omega,t}\big(K)\big)^2$, where
\begin{equation}\label{eq::definition_lambda_hat_lambda}
	\widehat\lambda^{\mathrm{y},\mathrm{y}^\smalltext{\prime}}_t(\tilde\omega) \coloneqq - \sqrt{\mathrm{r}_t(\tilde\omega)}\sgn(\mathrm{y}-\mathrm{y}^\prime),\; 
	\eta^{z,z^\smalltext{\prime},a}_t(\tilde\omega) \coloneqq -\sqrt{\theta^{X}_t(\tilde\omega)} \frac{(z-z^\prime)}{| (z-z^\prime)^\top a(z-z^\prime)|^{1/2}}\1_{\{(z-z^\smalltext{\prime})^\smalltext{\top} a (z-z^\smalltext{\prime}) \neq 0\}}.
\end{equation}

\medskip
The following is our main additional assumption in this section. We will comment on its necessity further below.
\begin{assumption}\label{ass::crossing}
	For every $(\omega,s) \in \Omega \times [0,\infty)$ and every $\P \in \fP(s,\omega)$, the following hold$:$
	\begin{enumerate}
	\item[$(i)$] for every $(\cY,\cY^{\prime},\cZ,\cU) \in \big(\cS^{2,s,\omega}_{T,\hat\beta}(\F_\smallertext{+},\P)\big)^2 \times \H^{2,s,\omega}_{T,\hat\beta}(X^{c,\P};\F,\P) \times \H^{2,s,\omega}_{T,\hat\beta}(\mu^X;\F,\P)$, there exists an $\F$-progressive process $\lambda = \lambda^{s,\omega,\P}\1_{\llparenthesis 0,(T-s\land T)^{\smallertext{s}\smallertext{,}\smallertext{\omega}} \rrbracket}$ such that $\lambda\Delta C^{s,\omega}_{s\smallertext{+}\smallertext{\cdot}} > -1$ holds \textnormal{$\P$--a.s.}, $|\lambda| \leq \sqrt{r^{s,\omega}_{s\smallertext{+}\smallertext{\cdot}}}$ holds \textnormal{$\P\otimes \d C^{s,\omega}_{s\smallertext{+}\smallertext{\cdot}}$--a.e.} on $\llparenthesis 0, (T-s\land T)^{s,\omega} \rrbracket$, and
	\begin{equation*}
		f^{s,\omega,\P}\big(\cY,\cY_{\smallertext{-}},\cZ,\cU(\cdot)\big) - f^{s,\omega,\P}\big(\cY^{\prime},\cY_{\smalltext{-}},\cZ,\cU(\cdot)\big) 
		\geq \lambda (\cY - \cY^{\prime}), \; \textnormal{$\P \otimes \mathrm{d} C^{s,\omega}_{s\smallertext{+}\smallertext{\cdot}}$--a.e. on $\llparenthesis 0, (T-s\land T)^{s,\omega} \rrbracket$$;$}
	\end{equation*}
	
	\item[$(ii)$] there exists $\mathfrak{C} = \mathfrak{C}(\omega,s,\P) \in [0,\infty)$ such that
	\begin{equation*}
		\int_0^{(T\smallertext{-}s\land T)^{\smalltext{s}\smalltext{,}\smalltext{\omega}}}\max\Big\{\sqrt{\mathrm{r}^{s,\omega}_{s\smallertext{+}t}}, (\theta^{X})^{s,\omega}_{s\smallertext{+}t}, (\theta^\mu)^{s,\omega}_{s\smallertext{+}t}\Big\} \d (C^{s,\omega}_{s\smallertext{+}\smallertext{\cdot}})_t \leq \mathfrak{C}, \; \textnormal{$\P$--a.s.};
	\end{equation*}

	\item[$(iii)$] for every $(\cY,\cZ,\cU,\cU^{\prime}) \in \cS^{2,s,\omega}_{T,\hat\beta}(\F_\smallertext{+},\P) \times \H^{2,s,\omega}_{T,\hat\beta}(X^{c,\P};\F,\P) \times \big(\H^{2,s,\omega}_{T,\hat\beta}(\mu^X;\F,\P)\big)^2$, there exists $\rho \in \H^{2,s,\omega}_{T}(\mu^X;\F^\P_\smallertext{+},\P)$ such that $\Delta (\rho \ast\tilde{\mu}^{X,\P}) > -1$, \textnormal{$\P$--a.s.}, and
	\begin{equation*}
		\frac{\d\langle \rho \ast\tilde\mu^{X,\P}\rangle^{(\P)}}{\d  C^{s,\omega}_{s\smallertext{+}\smallertext{\cdot}}} \leq (\theta^\mu)^{s,\omega}_{s\smallertext{+}\smallertext{\cdot}},
	\end{equation*}
	\begin{equation*}
		f^{s,\omega,\P}\big(\cY,\cY_{\smallertext{-}},\cZ,\cU(\cdot)\big) - f^{s,\omega,\P}\big(\cY,\cY_{\smallertext{-}},\cZ,\cU^{\prime}(\cdot)\big) 
		\geq \frac{\d\langle \rho \ast\tilde\mu^{X,\P},(\cU-\cU^{\prime})\ast\tilde\mu^{X,\P}\rangle^{(\P)}}{\d C^{s,\omega}_{s\smallertext{+}\smallertext{\cdot}}},
	\end{equation*}
	both hold $\P \otimes \mathrm{d}C^{s,\omega}_{s\smallertext{+}\smallertext{\cdot}}$--{\rm a.e.} on $\llparenthesis 0, (T-s\land T)^{s,\omega} \rrbracket$.
	\end{enumerate}
	
	Moreover, for every $\P\in \fP_0$, the following holds$:$
	\begin{enumerate}
	\item[$(iv)$] there exists $\rho^\dagger \in \H^{2}_{T}(\mu^X;\F,\P)$ with $\Delta (\rho^\dagger \ast\tilde{\mu}^{X,\P}) > -1$, \textnormal{$\P$--a.s.}, such that, for every $(\cY,\cZ,\cU,\cU^{\prime}) \in \cS^{2}_{T,\hat\beta}(\F_\smallertext{+},\P) \times \H^{2}_{T,\hat\beta}(X^{c,\P};\F,\P) \times \big(\H^{2}_{T,\hat\beta}(\mu^X;\F,\P)\big)^2$,
	\begin{equation*}
		\frac{\d\langle \rho^\dagger \ast\tilde\mu^{X,\P}\rangle^{(\P)}}{\d  C} \leq \theta^\mu,
	\end{equation*}
	\begin{equation*}
		f^{\P}\big(\cY,\cY_{\smallertext{-}},\cZ,\cU(\cdot)\big) - f^{\P}\big(\cY,\cY_{\smallertext{-}},\cZ,\cU^{\prime}(\cdot)\big) 
		\geq \frac{\d\langle \rho^\dagger \ast\tilde\mu^{X,\P},(\cU-\cU^{\prime})\ast\tilde\mu^{X,\P}\rangle^{(\P)}}{\d C},
	\end{equation*}
	hold $\P \otimes \mathrm{d}C$--{\rm a.e.} on $\llparenthesis 0, T \rrbracket$.
	\end{enumerate}
\end{assumption}

\subsection{The extrinsic 2BSDE}

For $\P \in \fP_0$, we denote by $\cY^{\P}(T,\xi)$ the first component of the solution $(\cY,\cZ,\cU,\cN)$ to the BSDE
\begin{align*}
	\cY_t = \xi + \int_t^T f^\P_r\big(& \cY_r,\cY_{r\smallertext{-}},\cZ_r,\cU_r(\cdot)\big)\d C_r - \bigg(\int_t^T \cZ_r \d X^{c,\P}_r\bigg)^{(\P)} \nonumber\\
	& - \bigg(\int_t^T\int_{\R^\smalltext{d}} \cU_r(x)\tilde\mu^{X,\P}(\d r, \d x)\bigg)^{(\P)} - \int_t^T\d \cN_r, \; t \in [0,\infty], \; \text{$\P$--a.s.},
\end{align*}
in the sense of \cite[Section 3]{possamai2024reflections}, relative to $(\F_\smallertext{+},\P)$, where $X^{c,\P}$ denotes the $(\F_\smallertext{+},\P)$--continuous local martingale part of $X$; see \cite[Appendix A]{possamai2026aggregation}. 
	
\medskip
We denote by $\G = (\cG_t)_{t \in [0,\infty)}$ the filtration defined by $\cG_t \coloneqq \sigma(\cF^\ast_t \cup \sN^{\fP_\smalltext{0}})$, where $\cF^\ast_t$ denotes the universal completion of $\cF_t$, that is, $\cF^\ast_t \coloneqq \cap_\P \cF_t(\P)$ with intersection running over all probability measures $\P$ on $(\Omega,\cF_t)$ and $\cF_t(\P)$ denoting the $\P$-completion of $\cF_t$, and $\sN^{\fP_\smalltext{0}}$ denotes the collection of all subsets $A \subseteq \Omega$ that are $(\cF,\P)$--null sets for all $\P\in\fP_0$; the subset $A$ is an $(\cF,\P)$--null set if there exists $B \in \cF$ with $A \subseteq B$ and $\P[B] = 0$. We then write $\G_\smallertext{+} = (\cG_{t\smallertext{+}})_{t \in [0,\infty)}$ for $\cG_{t\smallertext{+}}\coloneqq \cap_{s \in (t,\infty)}\cG_s$ and implicitly adjoin $\cG_{0\smallertext{-}} \coloneqq \{\varnothing, \Omega\}$ and $\cG_{\infty} \coloneqq \cG_{\infty\smallertext{-}} \coloneqq \sigma(\cup_{t \in [0,\infty)} \cG_t)$ to $\G$ or $\G_\smallertext{+}$ whenever necessary. Note that $\sN^{\fP_\smalltext{0}}$ is closed under taking countable unions and subsets.

\begin{definition}
	A tuple $(Y, Z, (U^\P, N^\P, K^\P)_{\P \in \fP_\smalltext{0}})$ is said to be a solution to the \textnormal{extrinsic 2BSDE \ref{2BSDE::measurability}--\ref{2BSDE::aggregation}} if
\begin{enumerate}[{\bf (2B1)}, leftmargin=1.2cm]
	\item\label{2BSDE::measurability} $Y = (Y_t)_{t \in [0,\infty]}$ is a real-valued, $\G_\smallertext{+}$-adapted process with c\`adl\`ag paths, $Z = (Z_t)_{t \in [0,\infty)}$ is $\R^d$-valued and $\G$-predictable, and every $U^\P$ is real-valued and $\cP(\G) \otimes \cB(\R^d)$-measurable$;$
	\item\label{2BSDE::dynamics} $(Z,U^\P,N^\P) \in \H^2_T(X^{c,\P};\G,\P) \times \H^2_T(\mu^X;\G,\P) \times \cH^{2,\perp}_{0,T}(X^{c,\P},\mu^X;\G_\smallertext{+},\P)$, and $K^\P = (K^\P_t)_{t \in [0,\infty]}$ is a real-valued, $\G^\P_\smallertext{+}$-predictable process with c\`adl\`ag and non-decreasing paths starting at zero, satisfying $K^\P = K^\P_{\cdot\land T}$ and $\E^\P\big[|K^\P_T|^2\big] < \infty$, such that
	\[
	\int_0^T \big|f^\P_s\big(Y_s,Y_{s\smallertext{-}},Z_s,U^\P_s(\cdot)\big)\big|\d C_s < \infty, \; \textnormal{$\P$--a.s.},
	\]
	and
	\begin{align*}
		\displaystyle Y_t &= \xi + \int_t^T f^\P_s\big(Y_s,Y_{s\smallertext{-}},Z_s, U^\P_s(\cdot)\big) \d C_s- \bigg(\int_t^T Z_s \d X^{c,\P}_s\bigg)^{(\P)} - \bigg(\int_t^T\int_{\R^\smalltext{d}} U^\P_s(x)\tilde\mu^{X,\P}(\d s,\d x)\bigg)^{(\P)}- \int_t^T \d N^\P_s \\
		&\quad   + K^\P_T - K^\P_t, \; t \in [0,\infty], \; \textnormal{$\P$--a.s.},
	\end{align*}
	holds for all $\P \in \fP_0$$;$ and
	\item\label{2BSDE::aggregation} $Y_t = \underset{\bar{\P} \in \fP_\smalltext{0}(\cG_{\smallertext{t}\smalltext{+}},\P)}{{\esssup}^\P} \cY^{\bar{\P}}_t(T,\xi)$, \textnormal{$\P$--a.s.}, $t \in [0,\infty], \; \P \in \fP_0$, where $\fP_0(\cG_{t\smallertext{+}},\P) = \big\{\overline{\P}\in\fP_0 : \textnormal{$\overline{\P} = \P$ on $\cG_{t\smallertext{+}}$}\big\}$.
\end{enumerate}
\end{definition}

It does not matter whether we formulate \ref{2BSDE::aggregation} in terms of $\G_\smallertext{+}$ or $\F_\smallertext{+}$, as shown by the following result.

\begin{lemma}
	Let $t \in [0,\infty)$. Then $\cG_{t\smallertext{+}} = \cF^\ast_{t\smallertext{+}} \lor \sN^{\fP_\smalltext{0}}$ and therefore $\fP_0(\cG_{t\smallertext{+}},\P) = \fP_0(\cF_{t\smallertext{+}},\P)$ for any $\P\in\fP_0$.
\end{lemma}

\begin{proof}
	We have
	\[
		\cF^\ast_{t\smallertext{(}\smallertext{+}\smallertext{)}}\lor\sN^{\fP_\smalltext{0}} = \{B \subseteq \Omega: \textnormal{$B\setminus N = C\setminus N$ for some $(C,N)\in\cF^\ast_{t\smallertext{(}\smallertext{+}\smallertext{)}}\times\sN^{\fP_\smalltext{0}}$}\}, \; t \in [0,\infty).
	\]
	(This holds even when replacing $\cF_{t\smallertext{(}\smallertext{+}\smallertext{)}}$ with any $\sigma$-algebra and $\sN^{\fP_\smalltext{0}}$ with a $\sigma$-ideal, that is, a collection of sets containing the empty set that is closed under taking countable unions and subsets.)
	The inclusion $(\supseteq)$ is immediate. For the converse inclusion $(\subseteq)$, let $B \in \cG_{t\smallertext{+}}$ and choose $(C_n,N_n)\in\cF^\ast_{t+1/n} \times \sN^{\fP_\smalltext{0}}$ such that $B\setminus N_n = C_n \setminus N_n$ for $n \in \N^\star$. Then it is straightforward to check that $C \coloneqq \cup_{k \in \N^\star} \cap_{m \geq k} C_m \in \cF^\ast_{t\smallertext{+}}$ and for $N \coloneqq \cup_{n \in \N^\star}N_n \in \sN^{\fP_\smalltext{0}}$, we have $B \setminus N = C_m \setminus N$ for every $m \in \N^\star$, and therefore
	\[
		B \setminus N = C \setminus N.
	\]
	Thus $B \in \cF^\ast_{t\smallertext{+}} \lor \sN^{\fP_\smalltext{0}}$. The second assertion is immediate.
\end{proof}
	
	\begin{remark}
		$(i)$ In the above system, the difference $S^\P \coloneqq N^\P - K^\P$ satisfies
		\begin{equation*}
			S^\P_t = Y_t - Y_0 + \int_0^t f^\P_s\big(Y_s,Y_{s\smallertext{-}},Z_s, U^\P_s(\cdot)\big) \d C_s  - \bigg(\int_0^t Z_s \d X^{c,\P}_s\bigg)^{(\P)} - \bigg(\int_0^t\int_{\R^\smalltext{d}} U^\P_s(x)\tilde\mu^{X,\P}(\d s,\d x)\bigg)^{(\P)}, \; t \in [0,\infty], \; \textnormal{$\P$--a.s.}
		\end{equation*}
		Hence, we could drop $(N^\P,K^\P)_{\P \in \fP_0}$ from the tuple $(Y,Z,(U^\P,N^\P,K^\P)_{\P \in \fP_\smalltext{0}})$, and then merely  say that for each $\P \in \fP_0$, the process $S^\P$, defined as above, has to be a $(\G^\P_\smallertext{+},\P)$--super-martingale such that $[S^\P,X^{c,\P}]^{(\P)} = 0$, \textnormal{$\P$--a.s.}, and its Doob--Meyer decomposition $S^\P = N^\P - K^\P$ meets the requirement $M^\P_{\mu^\smalltext{X}}[\Delta N^\P |\widetilde\cP(\G)] = 0$$;$ compare with \textnormal{\cite[Definition 4.14]{nutz2012superhedging}} and \textnormal{\cite[Definition 3.11]{lin2020second}}.
		
		\medskip
		$(ii)$ As described later in \textnormal{\Cref{rem::aggregation}}, in case $X$ is $\P$--quasi--left-continuous and satisfies the $(\F_\smallertext{+},\P)$--martingale representation property for every $\P \in \fP_0$, then each $(\widehat{U}^\P\ast\tilde{\mu}^{X,\P})^{(\P)}$ is the unique $(\F_\smallertext{+},\P)$--purely discontinuous square-integrable martingale starting at zero whose jumps are given by $\frac{\Delta[\hat\cY^\smallertext{+},V\ast\mu^X]}{V(\Delta X)} \1_{\{\Delta X \neq 0\}}$.
	\end{remark}

	Let
	\[
		\sL^2_{T,\beta}(\fP_0) \coloneqq \Big\{\big(Z,(U^\P,N^\P)_{\P \in \fP_0}\big) : (Z,U^\P,N^\P) \in \H^2_{T,\beta}(X^{c,\P};\G,\P) \times \H^2_{T,\beta}(\mu^X;\G,\P) \times \cH^{2,\perp}_{T,\beta}(X^{c,\P},\mu^X;\G_\smallertext{+},\P), \; \P \in \fP_0 \Big\}
	\]
	and
	\begin{equation}\label{eq::constant_phi}
		\phi^{2,\hat{\beta}}_{\xi,f} 
		\coloneqq \sup_{\P \in \fP_\smalltext{0}} \E^\P\Bigg[\sup_{s \in \D_\tinytext{+}} \underset{\bar{\P} \in \fP_\smalltext{0}(\cF_{\smalltext{s}},\P)}{{\esssup}^\P}\E^{\bar{\P}}\bigg[\cE(\hat\beta A)_T|\xi|^2 + \int_s^T \cE(\hat\beta A)_r \frac{|f^{\bar\P}_r(0,0,0,\mathbf{0})|^2}{\alpha^2_r} \d C_r\bigg| \cF_{s}\bigg] \Bigg].
	\end{equation}

	\begin{theorem}\label{thm::2BSDE_wellposed}
		Suppose that {\rm Assumptions \ref{ass::probabilities2}}, {\rm\ref{ass::generator2}}, and {\rm\ref{ass::crossing}} hold, and that $\phi^{2,\hat{\beta}}_{\xi,f} < \infty$, where $\phi^{2,\hat{\beta}}_{\xi,f}$ is given by \eqref{eq::constant_phi}. Let $\beta^\star$ be the unique number in $(0,\hat{\beta})$ at which $M^\Phi_1(\beta^\star) = 1$ $($see \eqref{eq::contraction_reflected_bsde}$)$, and suppose that there exists $\beta \in (\beta^\star,\hat\beta)$ such that $\P[\cE(\beta A)_{T\smallertext{-}} < \infty] = 1$ holds for every $\P \in \fP_0$.\footnote{The second condition holds, for example, when $\P[\cE(\hat\beta A)_{T\smallertext{-}} < \infty] = 1$ for every $\P \in \fP_0$.}
		\begin{enumerate}
			\item[$(i)$] 
			There exists a solution $(\widehat{Y}, \widehat{Z}, (\widehat{U}^\P, \widehat{N}^\P, \widehat{K}^\P)_{\P \in \fP_\smalltext{0}})$ to the \textnormal{extrinsic 2BSDE \ref{2BSDE::measurability}--\ref{2BSDE::aggregation}} with $(\widehat{Z},(\widehat{U}^\P,\widehat{N}^\P\big)_{\P \in \fP_\smalltext{0}}\big) \in 
			\sL^2_{T,\beta}(\fP_0)$.
			\item[$(ii)$] Any other solution $(Y,Z,(U^\P,N^\P,K^\P)_{\P \in \fP_\smalltext{0}})$ to the \textnormal{extrinsic 2BSDE \ref{2BSDE::measurability}--\ref{2BSDE::aggregation}} with $(Z,(U^\P,N^\P)_{\P \in \fP_\smalltext{0}}) \in 
			\sL^2_{T,\beta}(\fP_0)$ is such that $\sup_{\P \in \fP_\smalltext{0}} \|Z - \widehat{Z}\|_{\H^\smalltext{2}_\smalltext{T}(X^{\smalltext{c}\smalltext{,}\smalltext{\P}};\G,\P)} = 0$, $\sup_{\P \in \fP_\smalltext{0}} \|U^\P - \widehat{U}^\P\|_{\H^\smalltext{2}_\smalltext{T}(\mu^\smalltext{X};\G,\P)} = 0$, and $(Y,N^\P,K^\P) = (\widehat{Y},\widehat{N}^\P,\widehat{K}^\P)$ outside a $\P$--null set for every $\P \in \fP_0$.
		\end{enumerate}
	\end{theorem}

\begin{remark}
	In \textnormal{\Cref{sec_intrinsic_2BSDE}}, we replace \ref{2BSDE::aggregation} with a condition which is similar to
	\begin{enumerate}[{\bf (2B3')}, leftmargin=1.3cm]
		\item\label{2BSDE::minimality_condition} $\underset{\bar\P \in \fP_\smalltext{0}(\cG_{t\tinytext{+}},\P)}{{\essinf}^\P} \E^{\bar\P}\big[K^{\bar\P}_T - K^{\bar\P}_t \big| \cG_{t\smallertext{+}}\big] = 0$, $\P$--{\rm a.s.}, $t \in [0,\infty]$, $\P \in \fP_0$.
	\end{enumerate}
	Condition \ref{2BSDE::minimality_condition} is in line with the seminal work of \textnormal{\citeauthor*{soner2012wellposedness} \cite{soner2012wellposedness}} on \textnormal{2BSDEs} $($see also \textnormal{\citeauthor*{nutz2012superhedging} \cite{nutz2012superhedging}}$)$. In the setting of \textnormal{\cite{soner2012wellposedness}}, the extrinsic and intrinsic formulations are equivalent. However, if $X$ has jumps, then we impose further conditions to obtain equivalence$;$ compare with \textnormal{\citeauthor*{kazi2015second} \cite{kazi2015second,kazi2015second2}}.
	
	\medskip
	Along similar lines of thought, it was already noted by \textnormal{\citeauthor*{popier2019second} \cite{popier2019second}} and \textnormal{\citeauthor*{o2024wellposedness} \cite{o2024wellposedness}} that when the generator is not necessarily Lipschitz-continuous, the minimality condition \ref{2BSDE::minimality_condition} needs to be modified. 
	Similarly, if one aims to study second-order \textnormal{reflected BSDEs}, then \textnormal{\ref{2BSDE::minimality_condition}} does not ensure uniqueness$;$ see \textnormal{\citeauthor*{matoussi2021corrigendum} \cite{matoussi2013second, matoussi2021corrigendum}} and also \textnormal{\citeauthor*{matoussi2014second} \cite{matoussi2014second}}. 

	\medskip
	Nonetheless, the main advantage of using the extrinsic condition \ref{2BSDE::aggregation} is that it provides a universal approach to defining the \textnormal{2BSDE} system, independent of the form of the generator, the driving martingales, or whether one seeks to aggregate \textnormal{BSDEs}, \textnormal{reflected BSDEs}, or other related \textnormal{BSDE} systems.
\end{remark}
	
	As in classical BSDE theory, we obtain norm estimates and a comparison principle for our extrinsic 2BSDE.
	
\begin{proposition}\label{prop::bound_and_comparison}
	Suppose that {\rm Assumptions \ref{ass::probabilities2}}, {\rm\ref{ass::generator2}}, and {\rm\ref{ass::crossing}} hold, and that $\phi^{2,\hat{\beta}}_{\xi,f} < \infty$, where $\phi^{2,\hat{\beta}}_{\xi,f}$ is given by \eqref{eq::constant_phi}. Let $\beta^\star$ be the unique number in $(0,\hat{\beta})$ at which $M^\Phi_1(\beta^\star) = 1$ $($see \eqref{eq::contraction_reflected_bsde}$)$, and suppose that there exists $\beta \in (\beta^\star,\hat\beta)$ such that $\P[\cE(\beta A)_{T\smallertext{-}} < \infty] = 1$ holds for every $\P \in \fP_0$. Let $(\widehat{Y},\widehat{Z},(\widehat{U}^\P,\widehat{N}^\P,\widehat{K}^\P\big)_{\P \in \fP_\smalltext{0}})$ be the solution to the \textnormal{extrinsic 2BSDE \ref{2BSDE::measurability}--\ref{2BSDE::aggregation}} from \textnormal{\Cref{thm::2BSDE_wellposed}}.
	\begin{enumerate}
		\item[$(i)$] \textnormal{(Norm bound):} There exists a constant $\mathfrak{C} \in (0,\infty)$ that depends only on $\hat\beta$, on $\beta$, and on $\Phi$, such that
		\begin{equation*}
			\sup_{\P \in \fP_\smalltext{0}}
			\Bigg\{\big\|\widehat Y\big\|^2_{\cS^\smalltext{2}_{\smalltext{T}\smalltext{,}\smalltext{\hat\beta}}(\G_\tinytext{+},\P)} 
			+ \big\|\widehat{Z}\big\|^2_{\H^\smalltext{2}_{\smalltext{T}\smalltext{,}\smalltext{\beta}}(X^{c,\P};\G,\P)}
			+ \big\|\widehat{U}^\P\big\|^2_{\H^\smalltext{2}_{\smalltext{T}\smalltext{,}\smalltext{\beta}}(\mu^X;\G,\P)} 
			+ \big\|\widehat{N}^\P\big\|^2_{\cH^{\smalltext{2}}_{\smalltext{T}\smalltext{,}\smalltext{\beta}}(\G_\tinytext{+},\P)}
			+ \E^\P\bigg[\int_0^T\cE(\beta A)_r \d[\widehat{K}^\P]_r\bigg] \Bigg\}
			\leq \mathfrak{C}\phi^{2,\hat{\beta}}_{\xi,f}.
		\end{equation*}
		
		\item[$(ii)$] \textnormal{(Comparison):} Let $f^\prime$ and $\xi^\prime$ be maps satisfying the same assumptions as $f$ and $\xi$, respectively, and assume that $f^\prime$ has the same Lipschitz coefficients as $f$. Suppose that $\phi^{2,\hat{\beta}}_{\xi^\smalltext{\prime},f^\smalltext{\prime}}<\infty$. Let $\widehat Y^\prime$ denote the first component of the solution to \ref{2BSDE::measurability}--\ref{2BSDE::aggregation} with generator $f^\prime$ and terminal condition $\xi^\prime$. Suppose that $\xi^\prime \leq \xi$, \textnormal{$\fP_0$}--quasi-surely, and that
		\begin{equation*}
			f^{\prime,\P}_r\big(\cY^{\prime,\P}_r,\cY^{\prime,\P}_{r\smallertext{-}},\cZ^{\prime,\P}_r,\cU^{\prime,\P}_r(\cdot)\big) 
			\leq f^\P_r\big(\cY^{\prime,\P}_r,\cY^{\prime,\P}_{r\smallertext{-}},\cZ^{\prime,\P}_r,\cU^{\prime,\P}_r(\cdot)\big), \; \textnormal{$\P\otimes\d C_r$--a.e.}, \; \P \in \fP_0,
		\end{equation*}
		where $(\cY^{\prime,\P},\cZ^{\prime,\P},\cU^{\prime,\P},\cN^{\prime,\P})$ denotes the solution to the corresponding $\P$--\textnormal{BSDE} with generator $f^{\prime,\P}$ and terminal condition $\xi^\prime$. Then, $\widehat Y^{\prime} \leq \widehat Y$, \textnormal{$\fP_0$}--quasi-surely.
	\end{enumerate}
\end{proposition}

\subsection{The intrinsic 2BSDE}\label{sec_intrinsic_2BSDE}

The following definition describes the intrinsic 2BSDE.

\begin{definition}
	A tuple $(Y, Z, (U^\P, N^\P, K^\P)_{\P \in \fP_\smalltext{0}})$ is said to be a solution to the \textnormal{intrinsic 2BSDE \ref{2BSDE::measurability_int}--\ref{2BSDE::minimality_condition_2}} if
\begin{enumerate}[{\bf (2B1$^\ast$)}, leftmargin=1.3cm]
	\item\label{2BSDE::measurability_int} $Y = (Y_t)_{t \in [0,\infty]}$ is a real-valued, $\G_\smallertext{+}$-adapted process with c\`adl\`ag paths, $Z = (Z_t)_{t \in [0,\infty)}$ is $\R^d$-valued and $\G$-predictable, and every $U^\P$ is real-valued and $\cP(\G) \otimes \cB(\R^d)$-measurable$;$
	\item\label{2BSDE::dynamics_int} $(Z,U^\P,N^\P) \in \H^2_T(X^{c,\P};\G,\P) \times \H^2_T(\mu^X;\G,\P) \times \cH^{2,\perp}_{0,T}(X^{c,\P},\mu^X;\G_\smallertext{+},\P)$, and $K^\P = (K^\P_t)_{t \in [0,\infty]}$ is a real-valued, $\G^\P_\smallertext{+}$-predictable process with c\`adl\`ag and non-decreasing paths starting at zero, satisfying $K^\P = K^\P_{\cdot\land T}$ and $\E^\P\big[|K^\P_T|^2\big] < \infty$, such that
	\[
	\int_0^T \big|f^\P_s\big(Y_s,Y_{s\smallertext{-}},Z_s,U^\P_s(\cdot)\big)\big|\d C_s < \infty, \; \textnormal{$\P$--a.s.},
	\]
	and
	\begin{align*}
		\displaystyle Y_t &= \xi + \int_t^T f^\P_s\big(Y_s,Y_{s\smallertext{-}},Z_s, U^\P_s(\cdot)\big) \d C_s- \bigg(\int_t^T Z_s \d X^{c,\P}_s\bigg)^{(\P)} - \bigg(\int_t^T\int_{\R^\smalltext{d}} U^\P_s(x)\tilde\mu^{X,\P}(\d s,\d x)\bigg)^{(\P)}- \int_t^T \d N^\P_s \\
		&\quad   + K^\P_T - K^\P_t, \; t \in [0,\infty], \; \textnormal{$\P$--a.s.},
	\end{align*}
	holds for all $\P \in \fP_0$$;$ and
	\item\label{2BSDE::minimality_condition_2}
		$\displaystyle\underset{\bar{\P} \in \fP_\smalltext{0}(\cG_{\smalltext{t}\tinytext{+}},\P)}{{\essinf}^{\P}} \E^{\bar{\P}}\bigg[\int_t^T \cE\Big(\int_t^\cdot \lambda^{\bar{\P}}_s\d C_s\Big)_{r\smalltext{-}} \d K^{\bar{\P}}_r \bigg| \cG_{t\smallertext{+}}\bigg] = 0, \; \textnormal{$\P$--a.s.}, \; t \in [0,\infty], \; \P \in \fP_0$, where
		\[
			\lambda^{\bar{\P}} \coloneqq \frac{f^{\bar{\P}}(Y,Y_{\smallertext{-}},Z,U^{\bar{\P}}(\cdot)) - f^{\bar{\P}}(\cY^{\bar{\P}},Y_{\smallertext{-}},Z,U^{\bar{\P}}(\cdot))}{Y-\cY^{\bar{\P}}} \1_{\{Y \neq \cY^\smalltext{\bar{\P}}\}}, \; \overline{\P} \in \fP_0.
		\]
\end{enumerate}
\end{definition}

We prove well-posedness of the intrinsic 2BSDE \ref{2BSDE::measurability_int}--\ref{2BSDE::minimality_condition_2} under the following additional conditions.
	
	\begin{assumption}\label{ass::intrinsic_2bsde}
		 \begin{enumerate} 
		\item[$(0)$] Let $\beta^\star$ be the unique number in $(0,\hat{\beta})$ at which $M^\Phi_1(\beta^\star) = 1$ $($see \eqref{eq::contraction_reflected_bsde}$)$. There exists $\beta \in (\beta^\star,\hat\beta)$ such that $\P[\cE(\beta A)_{T\smallertext{-}} < \infty] = 1$ holds for every $\P \in \fP_0$.
		 \item[$(i)$]\label{ass_jump_linearisation_bound} There exists $(\delta,\Theta) \in (0,\infty)^2$ such that for every $\P\in\fP_0$ and every $(\cY,\cZ,U,\cU) \in  \cS^{2}_{T,\hat\beta}(\F_\smallertext{+},\P) \times \H^{2}_{T,\hat{\beta}}(X^{c,\P};\F,\P) \times \H^{2}_{T,\beta}(\mu^X;\F,\P) \times \H^{2}_{T,\hat{\beta}}(\mu^X;\F,\P)$, there exist $(\rho_{1},\rho_{2}) \in \big(\H^{2}_{T}(\mu^X;\F^\P_\smallertext{+},\P)\big)^2$ satisfying $\Theta \geq \Delta (\rho_{i} \ast\tilde\mu^{X,\P}) > -1+\delta$, \textnormal{$\P$--a.s.} for $i \in \{1,2\}$, and
		 \begin{equation*}
		 	\frac{\d\langle \rho_{i} \ast\tilde\mu^{X,\P}\rangle^{(\P)}}{\d C} \leq \theta^\mu, \; i \in \{1,2\},
		\end{equation*}
			\begin{equation*}
		 	\frac{\d\langle \rho_2 \ast\tilde\mu^{X,\P},(U-\cU)\ast\tilde\mu^{X,\P}\rangle^{(\P)}}{\d C}
		 	\geq 
		 	f^{\P}\big(\cY,\cY_{\smallertext{-}},\cZ,U(\cdot)\big) - f^{\P}\big(\cY,\cY_{\smallertext{-}},\cZ,\cU(\cdot)\big) 
		 	\geq \frac{\d\langle \rho_1 \ast\tilde\mu^{X,\P},(U-\cU)\ast\tilde\mu^{X,\P}\rangle^{(\P)}}{\d C},
		 \end{equation*}
		 both hold $\P \otimes \mathrm{d}C$--{\rm a.e.} on $\llparenthesis 0, T \rrbracket$.
		 
		 \item[$(ii)$]\label{ass_F_plus_pasting} $\fP_0$ is stable under $\G_\smallertext{+}$-pastings at deterministic times, that is, for all $\P\in\fP_0$, $t \in [0,\infty)$, $B \in \cG_{t\smallertext{+}}$, and $(\P_1,\P_2)\in \big(\fP_0(\cG_{t\smallertext{+}},\P)\big)^2$, then $\overline{\P}\in\fP_0$, where
		 \[
		 	\overline{\P}[D] \coloneqq \E^\P\big[\P_1[D|\cG_{t\smallertext{+}}]\1_B + \P_2[D|\cG_{t\smallertext{+}}]\1_{B^c}\big], \; D \in \cF.
		 \]
		 $($This is equivalent to $\fP_0$ being stable under $\F_\smallertext{+}$-pastings at deterministic times.$)$

		\item[$(iii)$] There exists $\mathfrak{C} \in (0,\infty)$ such that
		 \[
		 	\int_0^T \max\{\sqrt{\mathrm{r}_s},\theta^X_s,\theta^\mu_s\}\d C_s \leq \mathfrak{C}, \; \textnormal{$\fP_0$--q.s.}
		 \]
		 
		 \item[$(iv)$] For every $\P \in \fP_0$, the compensator $\nu^\P$ of $\mu^X$ satisfies $\nu^{\P}(\{t\} \times \R^d) = 0$, for $t \in [0,\infty)$, outside a $\P$--null set. In other words, $X$ is quasi--left-continuous relative to every $\P \in \fP_0$.
		 \end{enumerate}
	\end{assumption}

The following is the main result of this section.

\begin{theorem}\label{thm::2BSDE_wellposed_2}
Suppose that {\rm Assumptions \ref{ass::probabilities2}}, {\rm\ref{ass::generator2}}, {\rm\ref{ass::crossing}}, and {\rm\ref{ass::intrinsic_2bsde}} hold, that $\phi^{2,\hat{\beta}}_{\xi,f} < \infty$, where $\phi^{2,\hat{\beta}}_{\xi,f}$ is given by \eqref{eq::constant_phi}, and that $T < \infty$, \textnormal{$\fP_0$--quasi surely}. Let $\big(\widehat{Y},\widehat{Z},(\widehat{U}^\P,\widehat{N}^\P,\widehat{K}^\P\big)_{\P \in \fP_\smalltext{0}}\big)$ denote the solution to the \textnormal{extrinsic 2BSDE \ref{2BSDE::measurability}--\ref{2BSDE::aggregation}} from \textnormal{\Cref{thm::2BSDE_wellposed}}.

\medskip
	$(i)$ $\big(\widehat{Y},\widehat{Z},(\widehat{U}^\P,\widehat{N}^\P,\widehat{K}^\P\big)_{\P \in \fP_\smalltext{0}}\big)$ is a solution to the \textnormal{intrinsic 2BSDE \ref{2BSDE::measurability_int}--\ref{2BSDE::minimality_condition_2}} with $(\widehat{Z},(\widehat{U}^\P,\widehat{N}^\P\big)_{\P \in \fP_\smalltext{0}}\big) \in \sL^2_{T,\beta}(\fP_0)$.

\medskip
	$(ii)$ Any other solution $(Y,Z,(U^\P,N^\P,K^\P)_{\P \in \fP_\smalltext{0}})$ to the \textnormal{intrinsic 2BSDE \ref{2BSDE::measurability_int}--\ref{2BSDE::minimality_condition_2}} with $(Z,(U^\P,N^\P)_{\P \in \fP_\smalltext{0}})\in \sL^2_{T,\beta}(\fP_0)$ and $\sup_{\P \in \fP_\smalltext{0}}\|Y\|_{\cS^\smalltext{2}_{\smalltext{T}\smalltext{,}\smalltext{\hat\beta}}(\G_\tinytext{+},\P)} < \infty$ is such that $\sup_{\P \in \fP_\smalltext{0}} \|Z - \widehat{Z}\|_{\H^\smalltext{2}_\smalltext{T}(X^{\smalltext{c}\smalltext{,}\smalltext{\P}};\G,\P)} = 0$, $\sup_{\P \in \fP_\smalltext{0}} \|U^\P - \widehat{U}^\P\|_{\H^\smalltext{2}_\smalltext{T}(\mu^\smalltext{X};\G,\P)} = 0$, and $(Y,N^\P,K^\P) = (\widehat{Y},\widehat{N}^\P,\widehat{K}^\P)$ outside a $\P$--null set for every $\P \in \fP_0$.
\end{theorem}

\begin{remark}
	The condition \textnormal{\ref{2BSDE::minimality_condition_2}} already appeared in a similar form in \textnormal{\cite{popier2019second}}, where \textnormal{2BSDEs} were studied under a monotonicity condition in the $y$-variable of the generator. In our setting, if the integral $\int_0^T \sqrt{r_s} \d C_s$ is \textnormal{$\fP_0$--q.s.} bounded, then, since $|\lambda^{\bar{\P}}_s| \leq \sqrt{r_s}$, the condition \textnormal{\ref{2BSDE::minimality_condition_2}} becomes equivalent to \textnormal{\ref{2BSDE::minimality_condition}}$;$ one just has to use the arguments that lead to \textnormal{\eqref{eq::boundedness_ev}}.
\end{remark}
	
	\begin{remark}
		\textnormal{\Cref{ass::intrinsic_2bsde}.$(ii)$} may be surprising at first sight. However, the intrinsic characterisation, in general, relies crucially on it. We comment on its necessity in \textnormal{\Cref{rem_increasing_process_bound}}.
	\end{remark}

\subsection{An invariance principle}

The construction of the process $\widehat\cY^\smallertext{+}_t(T,\xi)$ in \cite[Theorem~3.5]{possamai2026aggregation} implies that the following invariance principle holds
\begin{equation}\label{eq_invariance_Y}
	\widehat\cY^\smallertext{+}_{t+s}(T,\xi)(\omega\otimes_t\cdot) = \underset{\bar{\P} \in \fP_{\smalltext{t}\smalltext{,}\smalltext{\omega}}(\cF_{s\smalltext{+}},\P^{\smalltext{t}\smalltext{,}\smalltext{\omega}})}{{\esssup}^{\P^{\smalltext{t}\smalltext{,}\smalltext{\omega}}}} \cY^{t,\omega,\bar{\P}}_s(T-t \land T,\xi^{t,\omega}), \; \textnormal{$\P^{t,\omega}$--a.s.,} \; \textnormal{$\P$--a.e. $\omega \in \Omega$}, \; s \in [0,\infty], \; \P\in\fP_0,
\end{equation}
where $\fP_{t,\omega}(\cF_{s\smallertext{+}},\P^{t,\omega}) = \big\{\overline{\P}\in \fP(t,\omega) : \overline{\P} = \P^{t,\omega} \; \text{\rm on} \; \cF_{s\smallertext{+}}\big\}$.
This means, that $\widehat\cY^\smallertext{+}_{t+\cdot}(T,\xi)(\omega\otimes_t\cdot)$ can be seen as the first component of the solution
\[
	\big(\widehat\cY^{t,\omega,\smallertext{+}}(T,\xi),\widehat{Z}^{t,\omega}(T,\xi),(\widehat{U}^{t,\omega,\P^\smalltext{\prime}}(T,\xi),\widehat{N}^{t,\omega,\P^\smalltext{\prime}}(T,\xi),\widehat{K}^{t,\omega,\P^\smalltext{\prime}}(T,\xi)\big)_{\P^\smalltext{\prime} \in \fP(t,\omega)}\big)
\]
 to the 2BSDE whose initial set of measures is not $\fP_0$ but $\fP(t,\omega)$, and whose corresponding BSDEs have terminal time $T - t \land T$, terminal random variable $\xi^{t,\omega}$, generator $f_{t+\cdot}(\omega\otimes_t\cdot,\ldots)$, and integrator $C_{t+\cdot}(\omega\otimes_t\cdot)-C_t(\omega)$. Here, we abuse notation: $\widehat{Z}^{t,\omega}(T,\xi)$ should not be understood as $\widehat{Z}(T,\xi)(\omega\otimes_t\cdot)$ (yet). It follows from \eqref{eq_invariance_Y}, from uniqueness of the solution to the 2BSDE (or rather, uniqueness of the decomposition of $\widehat{\cY}^{t,\omega,\smallertext{+}}$), and from arguments similar to those used in the proof of \cite[Lemma~4.4]{possamai2026aggregation} that
 \begin{gather}
 	\widehat\cY^\smallertext{+}_{t+s}(T,\xi)(\omega\otimes_t\cdot) = \widehat\cY^{t,\omega,\smallertext{+}}_s(T,\xi), \; s \in [0,\infty], \nonumber\\[0.25em]
 	\big(\widehat{Z}_{t+s}(T,\xi)(\omega\otimes_t\cdot) - \widehat{Z}^{t,\omega}_s(T,\xi)\big)^\top \mathsf{a}^{t,\omega}\big(\widehat{Z}_{t+s}(T,\xi)(\omega\otimes_t\cdot) - \widehat{Z}^{t,\omega}_s(T,\xi)\big) = 0, \; \textnormal{$\d(C_{t+\cdot}(\omega\otimes_t\cdot)-C_t(\omega))$--a.e. $s \in [0,\infty)$}, \nonumber\\[0.25em]
 	\widehat{U}^\P_{t+s}(T,\xi)(\omega\otimes_t\cdot)= \widehat{U}^{t,\omega,\P^{\smalltext{t}\smalltext{,}\smalltext{\omega}}}_{s}(T,\xi) \; \textnormal{in $\widehat{\L}^2_{\omega\otimes_\smalltext{t}\tilde\omega,t+s}(\mathsf{K}^{t,\omega,\P^{\smalltext{t}\smalltext{,}\smalltext{\omega}}}_{\tilde\omega,s})$,  $\d(C_{t+\cdot}(\omega\otimes_t\cdot)-C_t(\omega))$--a.e. $s \in  [0,\infty)$,} \nonumber\\[0.25em]
 	\widehat{N}^{\P}_{t+s}(T,\xi)(\omega\otimes_t\cdot) = \widehat{N}^{t,\omega,\P^{\smalltext{t}\smalltext{,}\smalltext{\omega}}}_s(T,\xi), s \in [0,\infty], \nonumber\\[0.25em]
 	\widehat{K}^{\P}_{t+s}(T,\xi)(\omega\otimes_t\cdot) = \widehat{K}^{t,\omega,\P^{\smalltext{t}\smalltext{,}\smalltext{\omega}}}_s(T,\xi), \; s \in [0,\infty], \nonumber
 \end{gather}
 holds $\P^{t,\omega}$--a.s., for $\P$--a.e. $\omega \in \Omega$, for all $\P\in\fP_0$.
 
 \medskip
 A similar identity to the above appeared in \cite[Lemma A.2.1]{hernandez2023me}. The difference stems from the fact that we are shifting paths rather than space, although, on an intuitive level, they point to a similar property of invariance of solutions to 2BSDEs.
 
\section{Proofs}\label{sec_proofs}

This section is dedicated to the well-posedness of our 2BSDEs. We denote by $\widehat{\cY}^\smallertext{+}(T,\xi) = (\widehat{\cY}^\smallertext{+}_t(T,\xi))_{t \in [0,\infty]}$ the process constructed in \cite[Theorem 3.5]{possamai2026aggregation}. In both the extrinsic and intrinsic systems, this process will correspond to the first component of the solution.

\subsection{Well-posedness of the extrinsic system}\label{sec::decomposition_of_regularised_value_function}

The condition satisfied by $\widehat\cY^\smallertext{+}(T,\xi)$ in \cite[Theorem~3.5.$(iii)$]{possamai2026aggregation} is the (strong) non-linear super-martingale property in the context of non-linear expectations induced by solutions to BSDEs. As with classical super-martingales with respect to (linear) expectation, there exist non-linear super-martingale decompositions, see \citeauthor*{bouchard2016general} \cite{bouchard2016general}, \citeauthor*{chen2015g} \cite{chen2015g}, \citeauthor*{grigorova2020optimal} \cite{grigorova2020optimal}, \citeauthor*{lin2003nonlinear} \cite{lin2003nonlinear}, \citeauthor*{peng1999monotonic} \cite{peng1999monotonic}, \citeauthor*{peng2007expectations} \cite{peng2007expectations}, \citeauthor*{ren2022nonlinear} \cite{ren2022nonlinear}, \citeauthor*{royer2006backward} \cite{royer2006backward}, \citeauthor*{shi2014nonlinear} \cite{shi2014nonlinear}; these are referred to as \emph{non-linear Doob--Meyer decompositions}. The decomposition of $\widehat\cY^\smallertext{+}(T,\xi)$ will then yield the remaining terms in the solution to the 2BSDE.

\medskip
We derive the decomposition using reflected BSDEs, as suggested by Nicole El Karoui in \citeauthor*{soner2013dual} \cite[Remark 4.9]{soner2013dual}. For each $\P \in \fP_0$, we consider a reflected BSDE with (lower) obstacle $\widehat\cY^\smallertext{+}(T,\xi)$ and then show, using the non-linear super-martingale property, that the first component of the solution to the reflected BSDE cannot be strictly above $\widehat\cY^\smallertext{+}(T,\xi)$. This leads to the desired decomposition.

\begin{proposition}\label{prop::decomposition_yplus}
Suppose that {\rm Assumptions \ref{ass::probabilities2}}, {\rm\ref{ass::generator2}}, and {\rm\ref{ass::crossing}} hold, and that $\phi^{2,\hat{\beta}}_{\xi,f} < \infty$, where $\phi^{2,\hat{\beta}}_{\xi,f}$ is given by \eqref{eq::constant_phi}. Let $\beta^\star$ be the unique number in $(0,\hat{\beta})$ at which $M^\Phi_1(\beta^\star) = 1$ $($see \eqref{eq::contraction_reflected_bsde}$)$, and suppose that there exists $\beta \in (\beta^\star,\hat\beta)$ such that, for every $\P \in \fP_0$,  $\P[\cE(\beta A)_{T\smallertext{-}} < \infty] = 1$ holds.

\medskip
$(i)$ There exists a tuple $(\widehat{Z},(\widehat{U}^\P,\widehat{N}^\P,\widehat{K}^\P)_{\P \in \fP_\smalltext{0}}) = (\widehat{Z}(T,\xi),(\widehat{U}^\P(T,\xi),\widehat{N}^\P(T,\xi),\widehat{K}^\P(T,\xi))_{\P \in \fP_\smalltext{0}})$ such that, for every $\P\in\fP_0$, $(\widehat{Z},\widehat{U}^\P,\widehat{N}^\P)$ belongs to $\H^2_{T,\beta}(X^{c,\P};\G,\P) \times \H^2_{T,\beta}(\mu^X;\G,\P) \times \cH^{2,\perp}_{0,T,\beta}(X^{c,\P},\mu^X;\G_\smallertext{+},\P)$, and $\widehat{K}^\P = (\widehat{K}^\P_t)_{t \in [0,\infty]}$ is a c\`adl\`ag, non-decreasing, $\G^\P_\smallertext{+}$-predictable process starting at zero, satisfying $\widehat{K}^\P = \widehat{K}^\P_{\cdot \land T}$ and $\E^\P\big[|\widehat{K}^\P_T|^2\big] < \infty$, such that
		\begin{equation*}
			\widehat\cY^\smallertext{+}_t = \xi + \int_t^T f^\P_r\big(\widehat\cY^\smallertext{+}_r,\widehat\cY^\smallertext{+}_{r\smallertext{-}},\widehat{Z}_r, \widehat{U}^\P_r(\cdot)\big) \d C_r - \bigg(\int_t^T \widehat{Z}_r \d X^{c,\P}_r\bigg)^{(\P)} - \bigg(\int_t^T\int_{\R^\smalltext{d}} \widehat{U}^\P_r(x)\tilde\mu^{X,\P}(\d r,\d x)\bigg)^{(\P)} - \int_t^T \d \widehat{N}^\P_r + \widehat{K}^\P_T - \widehat{K}^\P_t,
		\end{equation*}
		holds for every $t \in [0,\infty]$, \textnormal{$\P$--a.s.}
		
		\medskip
		$(ii)$ Any other tuple $(Z,(U^\P,N^\P,K^\P)_{\P \in \fP_\smalltext{0}})$ meeting the conditions of $(i)$ satisfies $\sup_{\P\in\fP_\smalltext{0}}\|Z-\widehat{Z}\|_{\H^\smalltext{2}_\smalltext{T}(X^{\smalltext{c}\smalltext{,}\smalltext{\P}};\G,\P)} = 0$,
		$\sup_{\P \in \fP_\smalltext{0}}\|U^\P - \widehat{U}^\P\|_{\H^\smalltext{2}_\smalltext{T}(\mu^\smalltext{X};\G,\P)} = 0$, and $(N^\P,K^\P) = (\widehat{N}^\P,\widehat{K}^\P)$ outside a $\P$--null set for every $\P \in \fP_0$.
\end{proposition}

\begin{remark}\label{rem::aggregation}
	In case $X$ is $\P$--quasi-left-continuous, and the $\P$--martingale representation property in the sense of \textnormal{\cite[Definition~III.4.22]{jacod2003limit}} holds relative to $X$ for every $\P \in \fP_0$, we would have $\Delta[\widehat\cY^\smallertext{+},V\ast\mu^X] = \widehat{U}^\P(\Delta X)\1_{\{\Delta X \neq 0\}} V(\Delta X)$ up to $\P$-evanescence, where $V > 0$ is a $\widetilde\cP$-measurable function satisfying $0 \leq V \ast \mu^X \leq 1$ identically$;$ see \textnormal{\cite[Lemma 6.5]{neufeld2014measurability}}. However, this only describes the value of $\widehat{U}^\P(x)$ at $x = \Delta X$ whenever $X$ jumps, and does not provide an aggregated and predictable function $\widehat{U}$ of the integrand $\widehat{U}^\P$ defined on $\Omega \times [0,\infty) \times \R^d$$;$ compare with the discussion after \textnormal{Remark 4.8} in \textnormal{\cite{kazi2015second2}} or the proof of \textnormal{\cite[Lemma 2.11]{denis2024second}}. What this does provide us is an aggregated version of the jumps of the stochastic integral with respect to the compensated random measure---the jumps uniquely determine this purely discontinuous martingale up to a null set. More precisely, we have
	\[
	\Delta (\widehat{U}^\P\ast\tilde{\mu}^{X,\P})^{(\P)}= \widehat{U}^\P(\Delta X) \1_{\{\Delta X \neq 0\}} = \frac{\Delta[\widehat\cY^\smallertext{+},V\ast\mu^X]}{V(\Delta X)} \1_{\{\Delta X \neq 0\}},\; \text{\rm up to $\P$-evanescence},
	\]
	and thus $(\widehat{U}^\P\ast\tilde{\mu}^{X,\P})^{(\P)}$ is the, up to $\P$-indistinguishability, unique purely discontinuous local martingale whose jumps are given by 
	\[
	J\coloneqq \frac{\Delta[\hat\cY^\smallertext{+},V\ast\mu^X]}{V(\Delta X)} \1_{\{\Delta X \neq 0\}}.
	\]
	This does therefore not justify the claim made in \textnormal{\cite[page 13]{kazi2015second2}} that the family $\{\widehat{U}^\P : \P \in \fP_0\}$ can be aggregated into a predictable integrand $\widehat{U}$. Nonetheless, we have $M^\P_{\mu^\smalltext{X}}\big[J\big|\widetilde{\cP}(\G)\big] = \widehat{U}^\P$ in $\H^2_T(\mu^X;\G,\P)$$;$ this follows from \textnormal{\cite[Theorem III.4.20]{jacod2003limit}} and the proof of \textnormal{\cite[Lemma~III.4.24]{jacod2003limit}}.
\end{remark}
	
\begin{proof}[Proof of \Cref{prop::decomposition_yplus}]
	We fix $\P \in \fP_0$ and note that $M^\Phi_1(\beta) < M^\Phi_1(\beta^\star) = 1$. Recall that $\widehat\cY^\smallertext{+}$ is $\G_\smallertext{+}$-optional on $[0,\infty)$ and satisfies $\widehat\cY^\smallertext{+} = \widehat\cY^\smallertext{+}_{\cdot \land T}$ by \cite[Theorem 3.5.$(i)$]{possamai2026aggregation}. The integrability from \cite[Theorem 3.5.$(ii)$]{possamai2026aggregation} and the right-continuity of $\widehat\cY^\smallertext{+}$, together with \cite[Remark 2.10.$(vii)$, Lemma 3.3, Theorem 3.4, and Remark 3.5]{possamai2024reflections}, ensure that there exists a collection $(Y^\P,Z^\P,U^\P,N^\P,K^\P)$ such that $(Z^\P,U^\P,N^\P) \in \H^2_{T,\beta}(X^{c,\P};\G,\P) \times \H^2_{T,\beta}(\mu^X;\G,\P) \times \cH^{2,\perp}_{0,T,\beta}(X^{c,\P},\mu^X;\G_\smallertext{+},\P)$, $Y^\P = (Y^\P_t)_{t \in [0,\infty]}$ is real-valued, $\P$--a.s. c\`adl\`ag, $\G_\smallertext{+}$-optional and $(Y^\P,\alpha Y^\P,\alpha Y^\P_\smallertext{-}) \in \cT^2_T(\G_\smallertext{+},\P)\times\H^2_{T,\beta}(\G_\smallertext{+},\P) \times \H^2_{T,\beta}(\G_\smallertext{+},\P)$, $K^\P = (K^\P_t)_{t \in [0,\infty]}$ is real-valued, $\P$--a.s. c\`adl\`ag, non-decreasing and starting at zero, $\G$-predictable, satisfying $K^\P = K^\P_{\cdot \land T}$ and $\E^\P\big[|K^\P_T|^2\big] <\infty$, and for $t\in[0,\infty]$
	\[\begin{cases}
			\displaystyle Y^\P_t = \xi + \int_t^T f^\P_r\big(Y^\P_r,Y^\P_{r\smallertext{-}},Z^\P_r, U^\P_r(\cdot)\big) \d C_r - \int_t^T Z^\P_r \d X^{c,\P}_r - \int_t^T\int_{\R^\smalltext{d}} U^\P_r(x)\tilde\mu^{X,\P}(\d r,\d x) - \int_t^T \d N^\P_r + \int_t^T\mathrm{d}K^\P_r, \\[0.5em]
		\displaystyle Y^\P = Y^\P_{\cdot \land T} \geq \widehat\cY^\smallertext{+}_{\cdot \land T} = \widehat\cY^\smallertext{+}, \\[0.5em]
		\displaystyle \int_{(0,T)}(Y^\P_{r\smallertext{-}} - \widehat\cY^\smallertext{+}_{r\smallertext{-}}) \d K^{\P}_r + (Y^\P_{T\smallertext{-}} - \widehat\cY^\smallertext{+}_{T\smallertext{-}})\Delta K^\P_T = 0, 
	\end{cases}\]
	holds outside a $\P$--null set.	The goal is now to show that $Y^\P_{t \land T} = \widehat\cY^\smallertext{+}_{t \land T}$, $t \in [0,\infty]$, $\P$--a.s., or, equivalently, that $\P[Y^\P_{\tau \land T} > \widehat\cY^\smallertext{+}_{\tau \land T}] = 0$ for every $\G_\smallertext{+}$--stopping time $\tau$. For the sake of reaching a contradiction, suppose there exists a $\G_\smallertext{+}$--stopping time $\tau$ for which $\P[Y^\P_{\tau \land T} > \widehat\cY^+_{\tau \land T}] > 0$. Since
	\begin{equation*}
		0 < \P\big[Y^\P_\tau > \widehat\cY^\smallertext{+}_\tau\big] 
		= \P\big[Y^\P_{\tau \land T} > \widehat\cY^\smallertext{+}_{\tau \land T}\big] 
		= \P\big[Y^\P_{\tau \land T} > \widehat\cY^\smallertext{+}_{\tau \land T}, \tau \land T < \infty\big] 
		= \sum_{\ell = 1}^\infty \P\big[Y^\P_{\tau \land T} > \widehat\cY^\smallertext{+}_{\tau \land T}, \ell - 1\leq \tau \land T < \ell\big],
	\end{equation*}
	we can suppose, without loss of generality, that $\tau = \tau \land T$ is bounded.
	By \cite[Proposition 2.3.4]{weizsaecker1990stochastic}, the random time
	\begin{equation*}
		\tau_\varepsilon \coloneqq \inf\Big\{t \in [\tau \land T,\infty)  : Y^\P_t < \widehat\cY^\smallertext{+}_t + \varepsilon\cE(\beta A)^{\smallertext{-}1/2}_{t \land T}\Big\} \land T,
	\end{equation*}
	is a $\G_\smallertext{+}$--stopping time. Fix $\omega \in \{Y^\P_{\tau} > \widehat\cY^\smallertext{+}_{\tau}\} \cap \{\tau <  \tau_\varepsilon\} \cap \{\cE(\beta A)_{T\smallertext{-}} < \infty\}$. Then $Y^\P_t(\omega) \geq \widehat\cY^\smallertext{+}_t(\omega) + \varepsilon\cE(\beta A)^{-1/2}_t(\omega)$ for all $t \in [\tau(\omega),\tau_\varepsilon(\omega))$ and therefore $Y^\P_{t\smallertext{-}}(\omega) \geq \widehat\cY^\smallertext{+}_{t\smallertext{-}}(\omega) + \varepsilon\cE(\beta A)^{-1/2}_{t\smallertext{-}} > \widehat\cY^\smallertext{+}_{t\smallertext{-}}(\omega)$ for all $t \in (\tau(\omega),\tau_{\varepsilon}(\omega)]$. The Skorokhod condition
	\begin{equation*}
		\int_{(0,T)}(Y^\P_{r\smallertext{-}} - \widehat\cY^\smallertext{+}_{r\smallertext{-}}) \d K^{\P}_r + (Y^\P_{T\smallertext{-}} - \widehat\cY^\smallertext{+}_{T\smallertext{-}})\Delta K^\P_T = 0, \; \text{$\P$--a.s.},
	\end{equation*}
	then yields $K^\P_{\tau_\smalltext{\varepsilon}} = K^\P_{\tau}$, $\P$--a.s. on $\{Y^\P_\tau > \widehat\cY^\smallertext{+}_\tau\} \cap \{\tau < \tau_\varepsilon\} \cap \{\cE(\beta A)_{T\smallertext{-}} < \infty\}$.
	On the other hand, for $\omega \in \{\tau = \tau_\varepsilon\}$, we immediately have $K^\P_{\tau_\smalltext{\varepsilon}} = K^\P_{\tau}$, $\P$--a.s. on $\{Y^\P_\tau > \widehat\cY^\smallertext{+}_\tau\} \cap \{\tau = \tau_\varepsilon\} \cap \{\cE(\beta A)_{T\smallertext{-}} < \infty\}$. Therefore $K^\P_{\tau_\smalltext{\varepsilon}} = K^\P_{\tau}$, $\P$--a.s. on $\{Y^\P_\tau > \widehat\cY^\smallertext{+}_\tau\} \cap \{\cE(\beta A)_{T\smallertext{-}} < \infty\}$.
	Furthermore, we have $\tau_\varepsilon = \tau$ and thus $K^\P_{\tau_\smalltext{\varepsilon}} = K^\P_{\tau}$, $\P$--a.s. on $\{Y^\P_\tau \leq \widehat\cY^\smallertext{+}_\tau\}$. Therefore, to summarise, we must have $K^\P_{\tau_\smalltext{\varepsilon}} = K^\P_{\tau}$, $\P$--a.s. on $\{\cE(\beta A)_{T\smallertext{-}} < \infty\}$. Since $\P[\cE(\beta A)_{T\smallertext{-}} < \infty] = 1$, we find
	\begin{align*}
		Y^\P_t &= Y^\P_{\tau_\smalltext{\varepsilon}} + \int_t^{\tau_\smalltext{\varepsilon}} f^\P_r\big(Y^\P_r,Y^\P_{r\smallertext{-}},Z^\P_r, U^\P_r(\cdot)\big) \d C_r - \bigg(\int_t^{\tau_\smalltext{\varepsilon}} Z^\P_r \d X^{c,\P}_r\bigg)^{(\P)} - \bigg(\int_t^{\tau_\smalltext{\varepsilon}}\int_{\R^\smalltext{d}} U^\P_r(x)\tilde\mu^{X,\P}(\d r,\d x)\bigg)^{(\P)}\\
		&\quad  - \int_t^{\tau_\smalltext{\varepsilon}} \d N^\P_r, \; t \in [\tau,\tau_\varepsilon], \; \text{$\P$--a.s.}
	\end{align*}
	Note that $\cY^\P_t(\tau, Y^\P_{\tau})\1_{\{t < \tau\}} + Y^\P_{t \land \tau_\smalltext{\varepsilon}} \1_{\{t \geq \tau\}}$, $t \in [0,\infty]$, is the solution $\cY^\P(\tau_\varepsilon, Y^\P_{\tau_\smalltext{\varepsilon}})$ of the BSDE with generator $f^\P$ and terminal condition $Y^\P_{\tau_\smalltext{\varepsilon}}$ at time $\tau_\varepsilon$. Since $Y^\P_{\tau_\smalltext{\varepsilon}} \leq \widehat\cY^\smallertext{+}_{\tau_\smalltext{\varepsilon}} + \varepsilon\cE(\beta A)^{-1/2}_{\tau_\smalltext{\varepsilon}}$, $\P$--a.s., it follows from the comparison principle for BSDEs in \cite[Proposition~D.5]{possamai2026aggregation} that
	\begin{equation*}
		Y^\P_{\tau} = \cY^\P_{\tau}(\tau_\varepsilon ,Y^\P_{\tau_\smalltext{\varepsilon}}) 
		\leq \cY^\P_{\tau}\big(\tau_\varepsilon,\widehat\cY^\smallertext{+}_{\tau_\smalltext{\varepsilon}} + \varepsilon\cE(\beta A)^{-1/2}_{\tau_\smalltext{\varepsilon}}\big), \; \text{$\P$--a.s.},
	\end{equation*}
	and by the stability result in \cite[Proposition~D.2]{possamai2026aggregation}, that there exists a constant $\mathfrak{C}$ only depending on $\beta$ and on $\Phi$ such that
	\begin{equation*}
		\E^\P[\cY^\P_{\tau}(\tau_\varepsilon,\widehat\cY^\smallertext{+}_{\tau_\smalltext{\varepsilon}} + \varepsilon\cE(\beta A)^{-1/2}_{\tau_\smalltext{\varepsilon}})] - \E^\P[\cY^\P_{\tau}(\tau_\varepsilon,\widehat\cY^\smallertext{+}_{\tau_\smalltext{\varepsilon}})] 
		\leq \big| \E^\P[\cY^\P_{\tau}(\tau_\varepsilon,\widehat\cY^\smallertext{+}_{\tau_\smalltext{\varepsilon}} + \varepsilon\cE(\beta A)^{-1/2}_{\tau_\smalltext{\varepsilon}})] - \E^\P[\cY^\P_{\tau}(\tau_\varepsilon,\widehat\cY^\smallertext{+}_{\tau_\smalltext{\varepsilon}})] \big| \leq \mathfrak{C}^{1/2}\varepsilon.
	\end{equation*}
	This yields
	\begin{align*}
		\E^\P[Y^\P_{\tau}] 
		&\leq \E^\P\big[\cY^\P_{\tau}(\tau_\varepsilon,\widehat\cY^\smallertext{+}_{\tau_\smalltext{\varepsilon}} + \varepsilon\cE(\beta A)^{-1/2}_{\tau_\smalltext{\varepsilon}})\big] 
		\leq \varepsilon\mathfrak{C}^{1/2} + \E^\P[\cY^\P_{\tau}(\tau_\varepsilon,\widehat\cY^\smallertext{+}_{\tau_\smalltext{\varepsilon}})] 
		\leq \varepsilon\mathfrak{C}^{1/2} + \E^\P\big[ \widehat\cY^\smallertext{+}_{\tau}\big],
	\end{align*}
	where we used \cite[Theorem~3.5.$(iii)$]{possamai2026aggregation} in the final inequality, which is precisely the (strong) non-linear super-martingale property of $\widehat{\cY}^\smallertext{+}$. Given that $\varepsilon \in (0,\infty)$ was arbitrary, we must have $\E^\P[Y^\P_{\tau}] \leq \E^\P[\widehat\cY^+_{\tau}]$. It then follows from $Y^\P_\tau \geq \widehat\cY^+_{\tau}$, $\P$--a.s., that $Y^\P_{\tau} = \widehat\cY^+_{\tau}$, $\P$--a.s., which contradicts $\P[Y^\P_\tau > \widehat\cY^\smallertext{+}_\tau] > 0$. 
	
	\medskip
	We now construct $Z$, which replaces the family $(Z^\P)_{\P \in \fP_\tinytext{0}}$. To simplify the notation, we omit references to the filtration $\G_\smallertext{+}$ in what follows, provided no confusion arises. We follow the proof of \cite[Theorem 2.4]{nutz2015robust}, and denote by $\mathsf{C}^{(\hat\cY^\tinytext{+},X)}$ the $\G$-predictable, $\S^{d+1}$-valued process, which agrees, for every $\P\in\fP_0$, $\P$--a.s., with the second characteristic of the joint $\P$--semi-martingale $(\widehat\cY^\smallertext{+},X)$; see \cite[Proposition 6.6.$(i)$]{neufeld2014measurability}.\footnote{As noted in the proof of \cite[Theorem 2.4]{nutz2015robust}, the argument establishing the existence of this aggregator for the second characteristic in \cite[Proposition 6.6.$(i)$]{neufeld2014measurability} does not rely on any separability assumption of the filtration.} Let $\mathsf{C}^{X}$ denote the sub-matrix and $\mathsf{C}^{\widehat\cY^\tinytext{+},X}$ denote the row vector in $\mathsf{C}^{(\widehat\cY^\tinytext{+},X)}$ corresponding to the predictable quadratic variation of $X$ and predictable quadratic co-variation of $\widehat\cY^\smallertext{+}$ with $X$, respectively. Define the $\G$-predictable processes (with convention $a / 0 \coloneqq 0$ for $a \in \R$)
	\begin{equation*}
		\mathsf{c}^{\widehat\cY^\tinytext{+},X}_t \coloneqq c^{\widehat\cY^\tinytext{+},X}_t \1_{\{c^{\hat\cY^\tinytext{+},X}_t \in \R^\smalltext{d}\}}, \; \text{where} \; c^{\widehat\cY^\tinytext{+},X}_t \coloneqq \limsup_{n \rightarrow \infty} \frac{\mathsf{C}^{\widehat\cY^\tinytext{+},X}_t - \mathsf{C}^{\widehat\cY^\tinytext{+},X}_{(t-1/n)\lor 0}}{C_t - C_{(t-1/n)\lor 0}}, \; t \in [0,\infty),
	\end{equation*}
	and
	\begin{equation*}
		\mathsf{c}^{X}_t \coloneqq c^{X}_t \1_{\{c^{\smalltext{X}}_t \in \S^\smalltext{d}_\tinytext{+}\}}, \; \text{where} \; c^{X}_t \coloneqq \limsup_{n \rightarrow \infty} \frac{\mathsf{C}^{X}_t - \mathsf{C}^{X}_{(t-1/n)\lor 0}}{C_t - C_{(t-1/n)\lor 0}}, \; t \in [0,\infty),
	\end{equation*}
	component-wise, which $\fP_0$--q.s. satisfy $\mathsf{C}^{\widehat\cY^\tinytext{+},X} = \mathsf{c}^{\widehat\cY^\tinytext{+},X} \bcdot C$ and  $\mathsf{C}^{X} = \mathsf{c}^{X} \bcdot C$.
		Then, for any $\P\in\fP_0$,
	\begin{equation*}
		\langle (\widehat\cY^\smallertext{+})^{c,\P}, (X^{c,\P})^j \rangle^{(\P)} =  \langle \widehat\cY^\smallertext{+}, (X^{c,\P})^j \rangle^{(\P)} =  \langle Z^\P\bcdot X^{c,\P}, (X^{c,\P})^j \rangle^{(\P)} = \bigg(\sum_{i = 1}^d (Z^{\P})^i (\mathsf{c}^X)^{i,j}\bigg) \bcdot C, \; j \in \{1,\ldots,d\}, \; \text{$\P$--a.s.},
	\end{equation*}
	which expressed more compactly is $\mathsf{C}^{\widehat\cY^\tinytext{+},X} = \big((Z^\P)^\top \mathsf{c}^X\big) \bcdot C$, $\P$--a.s., and then
	\begin{equation*}
		\mathsf{c}^{\widehat\cY^\tinytext{+},X} = (Z^\P)^\top \mathsf{c}^X, \; \text{$\P \otimes \d C$--a.e.}, \; \P\in\fP_0.
	\end{equation*}
	Define the $\R^d$-valued and $\G$-predictable process
	\begin{equation*}
		\widehat{Z} \coloneqq \big(\mathsf{c}^{\widehat\cY^\tinytext{+},X} \big( \mathsf{c}^{X} \big)^{\oplus}\big)^\top.
	\end{equation*}
	This process satisfies 
	\[
		(\widehat{Z} - Z^\P)^\top \mathsf{c}^X (\widehat{Z} - Z^\P) = 0, \; \textnormal{$\P\otimes \d C$--a.e.,} \; \P \in \fP_0.
	\]
	Thus $\|\widehat{Z} - Z^\P\|_{\H^\smalltext{2}_{\smalltext{T}\smalltext{,}\smalltext{\beta}}(X^{\smalltext{c}\smalltext{,}\smalltext{\P}};\G,\P)} = 0$ and $(\widehat{Z} \bcdot X^{c,\P})^{(\P)} = (Z^\P\bcdot X^{c,\P})^{(\P)}$, $\P$--a.s., for every $\P \in \fP_0$. We also observe that the previous identity and the Lipschitz-continuity property of the generator permit replacing $Z^\P$ with $\widehat{Z}$ there. We then set $\widehat{U}^\P \coloneqq U^\P$, $\widehat{N}^\P\coloneqq N^\P$, and $\widehat{K}^\P \coloneqq K^\P$. We obtained a decomposition satisfying the conditions stated in $(i)$.
	
	\medskip
	Lastly, for $(ii)$, note that any other decomposition satisfying the conditions in $(i)$ is a solution to the well-posed reflected BSDE with obstacle $\widehat\cY^\smallertext{+}$ in the sense of \cite[Section 3.1]{possamai2024reflections}. Thus by well-posedness, we immediately get the stated uniqueness. This concludes the proof.
\end{proof}

We now turn to the existence and uniqueness of the solution.

\begin{proof}[Proof of \Cref{thm::2BSDE_wellposed}]
	The stated existence follows from \cite[Theorem~3.5]{possamai2026aggregation} and \Cref{prop::decomposition_yplus}.$(i)$. Indeed, let $\widehat{Y} \coloneqq \widehat{\cY}^\smallertext{+}$, and then denote by $(\widehat{Z},(\widehat{U}^\P,\widehat{N}^\P,\widehat{K}^\P)_{\P\in\fP_\smalltext{0}})$ the terms in the decomposition obtained in \Cref{prop::decomposition_yplus}.
	
	\medskip
	By \ref{2BSDE::aggregation} and \cite[Theorem~3.5.$(ii)$]{possamai2026aggregation}, any $Y$ from assertion $(ii)$ necessarily satisfies $Y_t = \widehat{\cY}^\smallertext{+}_t = \widehat{Y}_t$, $\P$--a.s., $t \in [0,\infty]$, for every $\P\in\fP_0$. Both $Y$ and $\widehat{\cY}^\smallertext{+}$ have c\`adl\`ag paths and thus $Y = \widehat{Y}$ outside a $\P$--null set, for every $\P\in\fP_0$. Hence, uniqueness of the solution follows from \Cref{prop::decomposition_yplus}.$(ii)$, which concludes the proof.
\end{proof}

\begin{proof}[Proof of \Cref{prop::bound_and_comparison}]
	Assertion $(ii)$ follows immediately from \ref{2BSDE::aggregation} and the comparison principle \cite[Proposition~D.5]{possamai2026aggregation}.
	
	\medskip
	We turn to $(i)$. Since the arguments are similar to those in the proofs of \cite[Proposition 5.1, Lemma 5.2, and Lemma 5.3]{possamai2024reflections}, we will be brief. We also make the following preliminary observation.
	
	\begin{lemma}\label{eq::FV_exponential}
	For any real-valued, right-continuous and non-decreasing process $V = (V_t)_{t \in [0,\infty)}$, and any $\gamma \in (0,\infty)$,
	\[
	\int_t^T\cE(\gamma A)_r \d V_r = \cE(\gamma A)_{t\land T}(V_T - V_{t\land T}) + \gamma \int_t^T\cE(\gamma A)_{u\smallertext{-}}\int_{u\smallertext{-}}^T\d V_r \d A_u, \; t \in [0,\infty).
	\]
	\end{lemma}
	
	\begin{proof}
	An application of Tonelli's theorem yields
	\begin{align*}
		\int_t^T\cE(\gamma A)_r \d V_r 
		&= \int_t^T\bigg(1+\int_0^r \cE(\gamma A)_{u\smallertext{-}}\gamma \d A_u\bigg)\d V_r \\
		&= (V_T - V_{t\land T}) + \gamma \int_0^\infty\int_0^\infty \1_{\{t < r \leq T\}}\1_{\{0 < u \leq r\}} \cE(\gamma A)_{u\smallertext{-}} \d A_u \d V_r \\
		&= (V_T - V_{t\land T}) + \gamma \int_0^\infty\int_0^\infty \1_{\{u \leq r \leq T\}}\1_{\{t < u \leq T\}} \cE(\gamma A)_{u\smallertext{-}} \d V_r \d A_u \\
		&\quad + \gamma \int_0^\infty\int_0^\infty \1_{\{t < r \leq T\}}\1_{\{0 < u \leq t\}} \cE(\gamma A)_{u\smallertext{-}} \d V_r \d A_u \\
		&= (V_T - V_{t\land T}) + \gamma \int_t^T \cE(\gamma A)_{u\smallertext{-}} \int_{u\smallertext{-}}^T \d V_r \d A_u + \gamma \int_0^t \cE(\gamma A)_{u\smallertext{-}} \int_t^T \d V_r \d A_u \\
		&= (V_T - V_{t\land T}) + \gamma \int_t^T \cE(\gamma A)_{u\smallertext{-}} \int_{u\smallertext{-}}^T \d V_r \d A_u + (V_T-V_{t\land T}) \int_0^{t\land T} \cE(\gamma A)_{u\smallertext{-}} \d (\gamma A)_u \\
		&= (V_T - V_{t\land T}) + \gamma \int_t^T \cE(\gamma A)_{u\smallertext{-}} \int_{u\smallertext{-}}^T \d V_r \d A_u + (V_T-V_{t\land T}) (\cE(\gamma A)_{t\land T} - 1) \\
		&= \cE(\gamma A)_{t\land T}(V_T - V_{t\land T}) + \gamma \int_t^T\cE(\gamma A)_{u\smallertext{-}}\int_{u\smallertext{-}}^T\d V_r \d A_u,
	\end{align*}
	which is exactly the desired equality.
\end{proof}
	
	We make another observation.
	
\begin{lemma}\label{lem::conditional_bound_2bsde}
	Suppose that \textnormal{\Cref{ass::generator2}} holds, and let $\beta \in (\beta^\star, \hat\beta)$. 
	Let $(Y,Z,(U^\P,N^\P,K^\P)_{\P\in\fP_0})$ satisfy \textnormal{\ref{2BSDE::measurability}} and \textnormal{\ref{2BSDE::dynamics}}, $\sup_{\P\in\fP_0} \|Y\|_{\cS^2_{T,\hat{\beta}}(\G_\smalltext{+},\P)} < \infty$, and $(Z,(U^\P,N^\P)_{\P\in\fP_0}) \in \sL^2_{T,\beta}(\fP_0)$. 
	There exists a constant $\mathfrak{C}(\beta,\hat\beta,\Phi) \in (0,\infty)$ depending only on $\beta$, on $\hat{\beta}$, and on $\Phi$ such that
	\begin{align*}
		&\E^\P\bigg[\int_t^T\cE(\beta A)_{r}\big(|\alpha_rY_r|^2 + |\alpha_rY_{r\smallertext{-}}|^2 +  (Z_r)^\top\mathsf{a}_rZ_r + \|U^\P_r(\cdot)\|^2_{\hat{\L}^\smalltext{2}_{\smalltext{\cdot}\smalltext{,}\smalltext{r}}(\mathsf{K}^\P_{\smalltext{\cdot}\smalltext{,}\smalltext{r}})} \big)\d C_r \nonumber\\
		&\quad + \int_t^T \cE(\beta A)_r \d \langle N^\P\rangle^{(\P)}_r + \int_t^T\cE(\beta A)_r \d[ K^\P]_r \bigg| \cF_{t\smallertext{+}} \bigg] \nonumber\\
		& \leq \mathfrak{C}(\beta,\hat\beta,\Phi) \Bigg(\E^\P\bigg[ \cE(\beta A)_T|\xi|^2
		+ \sup_{r \in [t,T]} |\cE(\hat\beta A)^{1/2}_r Y_r|^2 + \int_t^T\cE(\beta A)_{r}\frac{|f^{\P}_r(0,0,0,\mathbf{0})|^2}{\alpha^2_r}\d C_r\bigg|\cF_{t\smallertext{+}}\bigg] \Bigg), \; \textnormal{$\P$--a.s.},
	\end{align*}
	holds for all $t \in [0,\infty]$ and all $\P\in \fP_0$.
\end{lemma}
	
\begin{proof}
We fix $\P \in \fP_0$ and write
\[
f^\P \coloneqq f^\P\big(Y,Y_{\smallertext{-}},Z,U^\P(\cdot)\big),
\;  
f^{\P,0} \coloneqq f^\P(0,0,0,\mathbf{0}), 
\; \textnormal{and} \;
\eta^\P \coloneqq (Z \bcdot X^{c,\P})^{(\P)} + (U^\P\ast\tilde\mu^{X,\P})^{(\P)} + N^\P,
\]
for simplicity. Moreover, let $(\varepsilon,\kappa) \in (0,\infty)^2$ with $0 < 1-4\kappa \leq 1$. With an application of It\^o's formula to $|Y|^2$ (see \cite[Equation (5.64)]{possamai2024reflections}), we obtain
\begin{align*}
	&|Y_t|^2 + \int_{(t,\infty)} \d [\eta^\P]_r + \int_{(t,\infty)}\d[ K^\P]_r - 2 \int_{(t,\infty)} \d [\eta^\P, K^\P]_r \\
	& \leq (1+4\kappa)|\xi|^2 + \frac{1}{\varepsilon}\sup_{r \in (t,\infty)} |Y_r|^2 + (\varepsilon+4\kappa) \bigg(\int_{(t,\infty)} |f^\P_r|\d C_r\bigg)^2 - 2 \int_{(t,\infty)} Y_{r\smallertext{-}}\d\eta^\P_r + \sum_{r \in (t,\infty)} (f^\P_r \Delta C_r)^2 \\
	&\quad + \frac{2}{\kappa} \sup_{r \in [t,\infty)}|Y_r \1_{\{r < T\}}|^2 + 4\kappa \bigg(|Y_t|^2 + \bigg(\int_{(t,\infty)}\d\eta^\P_r\bigg)^2\bigg), \; t \in [0,\infty], \; \textnormal{$\P$--a.s.}
\end{align*}
We rearrange the preceding terms, take conditional expectations with respect to $\cF_{t\smallertext{+}}$,  Y (see \cite[Corollary~5.6.4 and Proposition~6.6.5]{weizsaecker1990stochastic}) 
\[
\E^\P\bigg[\int_{(t,\infty)}\d\langle\eta^\P\rangle_r\bigg|\cF_{t\smallertext{+}}\bigg] = \E^\P\bigg[\bigg(\int_{(t,\infty)}\d\eta^\P_r\bigg)^2\bigg|\cF_{t\smallertext{+}}\bigg] = \E^\P\bigg[\int_{(t,\infty)}\d[\eta^\P]_r\bigg|\cF_{t\smallertext{+}}\bigg],
\;
\sum_{r \in (t,\infty)} (f^\P_r \Delta C_r)^2 \leq  \bigg(\int_{(t,\infty)} |f^\P_r| \d C_r\bigg)^2,
\]
to obtain
\begin{align*}
	&(1-4\kappa)\Bigg(|Y_t|^2 + \E^\P\bigg[\int_{(t,\infty)} \d \langle\eta^\P\rangle_r\bigg| \cF_{t\smallertext{+}}\bigg] + \E^\P\bigg[\int_{(t,\infty)}\d[ K^\P]_r\bigg|\cF_{t\smallertext{+}}\bigg]\Bigg)  \\
	& \leq (1+4\kappa)\E^\P\big[|\xi|^2\big|\cF_{t\smallertext{+}}\big] + \frac{1}{\varepsilon}\E^\P\bigg[\sup_{r \in (t,\infty)} |Y_r|^2\bigg|\cF_{t\smallertext{+}}\bigg] \\
	&\quad + (1+\varepsilon+4\kappa) \E^\P\bigg[\bigg(\int_{(t,\infty)} |f^\P_r|\d C_r\bigg)^2\bigg|\cF_{t\smallertext{+}}\bigg] + \frac{2}{\kappa} \E^\P\bigg[\sup_{r \in [t,\infty)}|Y_r \1_{\{r < T\}}|^2\bigg|\cF_{t\smallertext{+}}\bigg], \; \textnormal{$\P$--a.s.}, \; t \in [0,\infty].
\end{align*}
The terms $\int_{0}^\cdot Y_{r\smallertext{-}}\d\eta^\P_r$ and $[\eta^\P, K^\P]$ are $(\F_\smallertext{+},\P)$--uniformly integrable martingales and hence their increments over $(t,\infty)$ have zero conditional expectation given $\cF_{t\smallertext{+}}$. Indeed, for the first term, this follows from the Burkholder--Davis--Gundy-inequality for $p = 1$ and the $\P$--square-integrability of $[\eta^\P]^{1/2}_\infty$ and $\sup_{t \in [0,T]}|Y_t|$. For the second term, \cite[Proposition~I.4.49]{jacod2003limit} implies that $[\eta^\P, K^\P]$ is an $(\F_\smallertext{+},\P)$--local martingale and the Kunita--Watanabe-inequality (see \cite[Theorem~VII.53]{dellacherie1982probabilities}) yields
\[
	\sup_{s \in [0,\infty)} |[\eta^\P, K^\P]_s| \leq \int_0^\infty |\d [\eta^\P, K^\P]_s | \leq [\eta^\P]^{1/2}_{\infty\smallertext{-}} [ K^\P]^{1/2}_{\infty\smallertext{-}} \leq [\eta^\P]^{1/2}_{\infty\smallertext{-}}  K^\P_T,
\]
where the right-hand side is $\P$-integrable. Here, $\int_0^\infty |\d [\eta^\P, K^\P]_s |$ denotes the total variation of $[\eta^\P, K^\P]$ on $(0,\infty)$.
By \cite[Lemma~C.5]{possamai2024reflections}, we obtain
\begin{align}\label{eq_conditional_estimate_increasing_process}
	&(1-4\kappa)\Bigg(|Y_{t\smallertext{-}}|^2 + \E^\P\bigg[\int_{[t,\infty)} \d \langle\eta^\P\rangle_r\bigg| \cF_{t\smallertext{-}}\bigg] + \E^\P\bigg[\int_{[t,\infty)}\d[ K^\P]_r\bigg|\cF_{t\smallertext{-}}\bigg]\Bigg)  \nonumber\\
	& \leq (1+4\kappa)\E^\P\big[|\xi|^2\big|\cF_{t\smallertext{-}}\big] + \frac{1}{\varepsilon}\E^\P\bigg[\lim_{s \uparrow\uparrow t}\sup_{r \in (s,\infty)} |Y_r|^2\bigg|\cF_{t\smallertext{-}}\bigg] \nonumber\\
	&\quad + (1+\varepsilon+4\kappa) \E^\P\bigg[\bigg(\int_{[t,\infty)} |f^\P_r|\d C_r\bigg)^2\bigg|\cF_{t\smallertext{-}}\bigg] + \frac{2}{\kappa} \E^\P\bigg[\lim_{s \uparrow\uparrow t}\sup_{r \in [s,\infty)}|Y_r \1_{\{r < T\}}|^2\bigg|\cF_{t\smallertext{-}}\bigg], \; \textnormal{$\P$--a.s.}, \; t \in (0,\infty].
\end{align}

\medskip
We now make the following observation. Suppose that $B = (B_t)_{t \in [0,\infty)}$ is a real-valued, right-continuous, $\P$--a.s. non-decreasing, $\F$-predictable process starting at zero, and let $G = (G_t)_{t \in [0,\infty)}$ be $[0,\infty]$-valued and $\cF\otimes\cB([0,\infty))$-measurable. We denote the $(\F,\P)$--predictable projection of $G$ by ${}^{p}G$; see \cite[Chapter~VI.2 and Appendix~1]{dellacherie1982probabilities} for background. Now fix $t \in [0,\infty]$ and an arbitrary $D \in \cF_{t\smallertext{+}}$, and consider $H = (H_s)_{s \in [0,\infty)}$ given by $H_s(\omega) \coloneqq \1_{D}(\omega)\1_{(t,\infty)}(s)$. Then $H$ is $\F$-predictable as it is left-continuous and $\F_{\smallertext{-}}$-adapted. Since the $(\F,\P)$-predictable projection of $H G$ is $H {}^{p}G$ (see \cite[Remark~VI.44.(e)]{dellacherie1982probabilities}), we have
\[
	\E^\P\bigg[\1_D \int_t^\infty G_s \d B_s\bigg] = \E^\P\bigg[\int_0^\infty H G_s \d B_s\bigg] = \E^\P\bigg[ \int_0^\infty H {}^{p}G_s \d B_s\bigg] = \E^\P\bigg[\1_D \int_t^\infty {}^{p}G_s \d B_s\bigg],
\]
where the equality in the middle follows from \cite[Theorem~VI.57]{dellacherie1982probabilities}. Since $D \in \cF_{t\smallertext{+}}$ was arbitrary, this yields
\begin{equation}\label{eq_conditional_predictable_projection_t}
	\E^\P\bigg[\int_t^\infty G_s \d B_s\bigg|\cF_{t\smallertext{+}}\bigg] = \E^\P\bigg[\int_t^\infty {}^{p}G_s \d B_s\bigg|\cF_{t\smallertext{+}}\bigg], \; \textnormal{$\P$--a.s.}, \; t \in [0,\infty].
\end{equation}

Since the predictable projection of a $\P$--a.s. left-continuous process is itself $\P$--a.s. left-continuous (see \cite[Theorem~VI.47.(2) and Remark~VI.50.(f)]{dellacherie1982probabilities}), it follows from \Cref{eq::FV_exponential}, \eqref{eq_conditional_estimate_increasing_process} and \eqref{eq_conditional_predictable_projection_t} that
\begin{align*}
	&\E^\P\bigg[\int_t^T \cE(\beta A)_r \d \langle\eta^\P\rangle_r + \int_t^T\cE(\beta A)_r \d[ K^\P]_r\bigg|\cF_{t\smallertext{+}}\bigg] \\
	&= \cE(\beta A)_{t\land T}\E^\P\big[\langle\eta^\P\rangle_T - \langle\eta^\P\rangle_{t \land T} \big|\cF_{t\smallertext{+}}\big] + \cE(\beta A)_{t\land T} \E^\P\big[[ K^\P]_T - [ K^\P]_{t \land T}\big|\cF_{t\smallertext{+}}\big] \\
	&\quad + \beta \E^\P\bigg[\int_t^T \cE(\beta A)_{u\smallertext{-}} \int_{u\smallertext{-}}^T \mathrm{d}\big(\langle\eta^\P\rangle + [ K^\P]\big)_r \mathrm{d} A_u\bigg|\cF_{t\smallertext{+}}\bigg] \\
	&\leq \cE(\beta A)_{t\land T}\Bigg(\frac{1+4\kappa}{1-4\kappa} \E^\P\big[|\xi|^2\big|\cF_{t\smallertext{+}}\big] + \frac{\frac1\varepsilon + \frac2\kappa}{1-4\kappa} \E^\P\bigg[\sup_{r \in [t,T]}|Y_r|^2\bigg|\cF_{t\smallertext{+}}\bigg] + \frac{1+\varepsilon+4\kappa}{1-4\kappa} \E^\P\bigg[\bigg(\int_t^T |f^\P_r|\d C_r\bigg)^2\bigg|\cF_{t\smallertext{+}}\bigg]\Bigg) \\
	&\quad + \frac{(1+4\kappa)}{(1-4\kappa)} \beta \E^\P\bigg[\int_t^T\cE(\beta A)_{u\smallertext{-}}|\xi|^2 \d A_u\bigg|\cF_{t\smallertext{+}}\bigg] + \frac{(1/\varepsilon+2/\kappa)}{(1-4\kappa)} \beta \E^\P\bigg[\int_t^T \cE(\beta A)_{u\smallertext{-}}\lim_{s\uparrow\uparrow u}\sup_{r \in [s,\infty]}|Y_r|^2 \d A_u \bigg|\cF_{t\smallertext{+}}\bigg] \\
	&\quad + \frac{(1+\varepsilon+4\kappa)}{(1-4\kappa)} \beta \E^\P\bigg[\int_t^T\cE(\beta A)_{u\smallertext{-}} \bigg(\int_{u\smallertext{-}}^T |f^\P_r|\d C_r\bigg)^2 \d A_u\bigg|\cF_{t\smallertext{+}}\bigg] - \beta \E^\P\bigg[\int_t^T \cE(\beta A)_{r\smallertext{-}}|Y_{r\smallertext{-}}|^2 \d A_r\bigg|\cF_{t\smallertext{+}}\bigg], \; \textnormal{$\P$--a.s.}
\end{align*}
By following the arguments that lead to \cite[Equations 5.17 and 5.19]{possamai2024reflections}, we obtain, for an arbitrary $\gamma \in (0,\beta)$
\begin{equation*}
	\int_t^T\cE(\beta A)_{u\smallertext{-}} \bigg(\int_{u\smallertext{-}}^T |f^\P_r|\d C_r\bigg)^2 \d A_u 
	\leq \frac{(1+\gamma\Phi)}{\gamma(\beta-\gamma)} \int_t^T\cE(\beta A)_r \frac{|f^\P_r|^2}{\alpha^2_r}\d C_r,
\end{equation*}
and by following the arguments that lead to \cite[Equations 5.21 and 5.26]{possamai2024reflections}, we obtain
\[
	\cE(\beta A)_{t\land T}\bigg(\int_t^T|f^\P_r|\d C_r\bigg)^2 \leq \frac{1}{\beta}\int_t^T\cE(\beta A)_r \frac{|f^\P_r|^2}{\alpha^2_r}\d C_r.
\]
Furthermore
\begin{align*}
	\int_t^T \cE(\beta A)_{u\smallertext{-}}\lim_{s\uparrow\uparrow u}\sup_{r \in [s,\infty]}|Y_r|^2 \d A_u
	&= \int_t^T \frac{\cE(\beta A)_{u\smallertext{-}}}{\cE(\hat\beta A)_{u\smallertext{-}}}\lim_{s\uparrow\uparrow u}\Big\{\cE(\hat\beta A)_{s}\sup_{r \in [s,T]}|Y_r|^2\Big\} \d A_u \\
	&\leq \int_t^T \frac{\cE(\beta A)_{u\smallertext{-}}}{\cE(\hat\beta A)_{u\smallertext{-}}}\lim_{s\uparrow\uparrow u}\sup_{r \in [s,T]}|\cE(\hat\beta A)^{1/2}_{r}Y_r|^2 \d A_u \\
	&\leq \sup_{r \in (t,T]}|\cE(\hat\beta A)^{1/2}_{r}Y_r|^2\int_t^T \frac{\cE(\beta A)_{u\smallertext{-}}}{\cE(\hat\beta A)_{u\smallertext{-}}} \d A_u \leq \frac{(1+\hat\beta\Phi)}{(\hat\beta-\beta)} \sup_{r \in (t,T]}|\cE(\hat\beta A)^{1/2}_{r}Y_r|^2,
\end{align*}
where the last inequality follows from the proof of \cite[Lemma 3.3]{possamai2024reflections}. We now substitute these bounds, rearrange the terms, and find
\begin{align*}
	&\min\bigg\{1,\frac{\beta}{(1+\beta \Phi)}\bigg\} \E^\P\bigg[\int_t^T \cE(\beta A)_{r}|Y_{r\smallertext{-}}|^2 \d A_r + \int_t^T \cE(\beta A)_r \d \langle\eta^\P\rangle_r + \int_t^T\cE(\beta A)_r \d[ K^\P]_r \bigg| \cF_{t\smallertext{+}} \bigg] \\
	&\leq \E^\P\bigg[\beta \int_t^T \cE(\beta A)_{r\smallertext{-}}|Y_{r\smallertext{-}}|^2 \d A_r + \int_t^T \cE(\beta A)_r \d \langle\eta^\P\rangle_r + \int_t^T\cE(\beta A)_r \d[ K^\P]_r \bigg| \cF_{t\smallertext{+}} \bigg] \\
	& \leq \frac{(1+4\kappa)}{(1-4\kappa)} \E^\P\big[\cE(\beta A)_T|\xi|^2\big|\cF_{t\smallertext{+}}\big]
	+ \frac{(1/\varepsilon + 2/\kappa)}{(1-4\kappa)} \bigg(1+\beta \frac{(1+\hat\beta\Phi)}{(\hat\beta-\beta)}\bigg)\E^\P\bigg[\sup_{r \in [t,T]} |\cE(\hat\beta A)^{1/2}_r Y_r|^2\bigg|\cF_{t\smallertext{+}}\bigg] \\
	&\quad + \frac{(1+\varepsilon + 4\kappa)}{(1-4\kappa)} \bigg(\frac{1}{\beta} + \beta\mathfrak{g}^\Phi(\beta)\bigg)\E^\P\bigg[ \int_t^T\cE(\beta A)_r \frac{|f^\P_r|^2}{\alpha^2_r} \d C_r \bigg|\cF_{t\smallertext{+}}\bigg], \; \textnormal{$\P$--a.s.}
\end{align*}
Thus, there exists a constant $\mathfrak{C}(\varepsilon,\kappa,\beta,\hat\beta,\Phi) \in (0,\infty)$ depending only on $\varepsilon$, $\kappa$, $\beta$, $\hat\beta$ and $\Phi$ such that
\begin{align*}
	&\E^\P\bigg[\int_t^T \cE(\beta A)_{r}|Y_{r\smallertext{-}}|^2 \d A_r + \int_t^T \cE(\beta A)_r \d \langle\eta^\P\rangle_r + \int_t^T\cE(\beta A)_r \d[ K^\P]_r \bigg| \cF_{t\smallertext{+}} \bigg] \\
	& \leq \mathfrak{C}(\varepsilon,\kappa,\beta,\hat\beta,\Phi) \Bigg( \E^\P\big[\cE(\beta A)_T|\xi|^2\big|\cF_{t\smallertext{+}}\big]
	+ \E^\P\bigg[\sup_{r \in [t,T]} |\cE(\hat\beta A)^{1/2}_r Y_r|^2\bigg|\cF_{t\smallertext{+}}\bigg] \Bigg) \\
	&\quad + \frac{(1+\varepsilon + 4\kappa)}{(1-4\kappa)} \max\bigg\{1,\frac{(1+\beta \Phi)}{\beta}\bigg\} \bigg(\frac{1}{\beta} + \beta\mathfrak{g}^\Phi(\beta)\bigg)\E^\P\bigg[ \int_t^T\cE(\beta A)_r \frac{|f^\P_r|^2}{\alpha^2_r} \d C_r \bigg|\cF_{t\smallertext{+}}\bigg] \\
	& \leq \mathfrak{C}(\varepsilon,\kappa,\beta,\hat\beta,\Phi) \Bigg( \E^\P\big[\cE(\beta A)_T|\xi|^2\big|\cF_{t\smallertext{+}}\big]
	+ \E^\P\bigg[\sup_{r \in [t,T]} |\cE(\hat\beta A)^{1/2}_r Y_r|^2\bigg|\cF_{t\smallertext{+}}\bigg] \Bigg) \\
	&\quad + \frac{(1+\varepsilon + 4\kappa)}{(1-4\kappa)} M^\Phi_1(\beta)\E^\P\bigg[ \int_t^T\cE(\beta A)_r \frac{|f^\P_r|^2}{\alpha^2_r} \d C_r \bigg|\cF_{t\smallertext{+}}\bigg], \; \textnormal{$\P$--a.s.}
\end{align*}

We use the inequality $(a+b)^2 \leq (1+\varpi)a^2 + (1+1/\varpi)b^2$ for an arbitrary $\varpi \in (0,\infty)$, and the Lipschitz-continuity property of $f^\P$ to obtain
\begin{align*}
	\int_t^T\cE(\beta A)_{r}\frac{|f^\P_r|^2}{\alpha^2_r}\d C_r
	&= \int_t^T\cE(\beta A)_{r}\frac{|(f^\P_r-f^{\P,0}_r) + f^{\P,0}_r|^2}{\alpha^2_r}\d C_r \\
	&\leq (1+\varpi) \int_t^T\cE(\beta A)_{r}\frac{|f^\P_r-f^{\P,0}_r|^2}{\alpha^2_r}\d C_r + \bigg(1+\frac{1}{\varpi}\bigg) \int_t^T\cE(\beta A)_{r}\frac{|f^{\P,0}_r|^2}{\alpha^2_r}\d C_r \\
	&\leq (1+\varpi) \Bigg( \int_t^T\cE(\beta A)_{r}|Y_r|^2\d A_r + \int_t^T\cE(\beta A)_{r}|Y_{r\smallertext{-}}|^2\d A_r + \int_t^T \cE(\beta A)_r\d \langle\eta^\P\rangle_r \Bigg) \\
	&\quad + \bigg(1+\frac{1}{\varpi}\bigg) \int_t^T\cE(\beta A)_{r}\frac{|f^{\P,0}_r|^2}{\alpha^2_r}\d C_r.
\end{align*}
With a new constant $\mathfrak{C}(\varepsilon,\kappa,\beta,\hat\beta,\Phi,\varpi)$ now depending additionally on $\varpi$, 
we obtain
\begin{align*}
	&\E^\P\bigg[\int_t^T \cE(\beta A)_{r}|Y_{r-}|^2 \d A_r + \int_t^T \cE(\beta A)_r \d \langle\eta^\P\rangle_r + \int_t^T\cE(\beta A)_r \d[ K^\P]_r \bigg| \cF_{t\smallertext{+}} \bigg] \\
	& \leq \mathfrak{C}(\varepsilon,\kappa,\beta,\hat\beta,\Phi,\varpi) \Bigg(\E^\P\bigg[ \cE(\beta A)_T|\xi|^2
	+ \sup_{r \in [t,T]} |\cE(\hat\beta A)^{1/2}_r Y_r|^2 + \int_t^T\cE(\beta A)_{r}|Y_r|^2\d A_r + \int_t^T\cE(\beta A)_{r}\frac{|f^{\P,0}_r|^2}{\alpha^2_r}\d C_r\bigg|\cF_{t\smallertext{+}}\bigg] \Bigg) \\
	&\quad + \frac{(1+\varepsilon + 4\kappa)}{(1-4\kappa)} (1+\varpi) M^\Phi_1(\beta)\Bigg(\E^\P\bigg[ \int_t^T\cE(\beta A)_{r}|Y_{r\smallertext{-}}|^2\d A_r + \int_t^T \cE(\beta A)_r\d \langle\eta^\P\rangle_r  \bigg|\cF_{t\smallertext{+}}\bigg] \Bigg), \; \textnormal{$\P$--a.s.},
\end{align*}
We now choose $\varepsilon$, $\kappa$ and $\varpi$ small enough to ensure
\[
	\frac{(1+\varepsilon + 4\kappa)}{(1-4\kappa)}(1+\varpi) M^\Phi_1(\beta) < 1,
\]
which is possible since $M^\Phi_1(\beta) < M^\Phi_1(\beta^\star) = 1$. We thus obtain, by changing the constant $\mathfrak{C}(\varepsilon,\kappa,\beta,\hat\beta,\Phi,\varpi)$ again, that
\begin{align*}
	&\E^\P\bigg[\int_t^T \cE(\beta A)_{r}|Y_{r-}|^2 \d A_r + \int_t^T \cE(\beta A)_r \d \langle\eta^\P\rangle_r + \int_t^T\cE(\beta A)_r \d[ K^\P]_r \bigg| \cF_{t\smallertext{+}} \bigg] \\
	& \leq \mathfrak{C}(\varepsilon,\kappa,\beta,\hat\beta,\Phi,\varpi) \Bigg(\E^\P\bigg[ \cE(\beta A)_T|\xi|^2
	+ \sup_{r \in [t,T]} |\cE(\hat\beta A)^{1/2}_r Y_r|^2 + \int_t^T\cE(\beta A)_{r}|Y_r|^2\d A_r + \int_t^T\cE(\beta A)_{r}\frac{|f^{\P,0}_r|^2}{\alpha^2_r}\d C_r\bigg|\cF_{t\smallertext{+}}\bigg] \Bigg),
\end{align*}
outside a $\P$--null set. Lastly, since $\hat\beta > \beta$, we have
\begin{align*}
	\int_t^T\cE(\beta A)_{r}|Y_r|^2\d A_r 
	&\leq \sup_{r \in (t,T]}|\cE(\hat\beta A)^{1/2}_rY_r|^2 \int_t^T \frac{\cE(\beta A)_r}{\cE(\hat\beta A)_r} \d A_r \\
	&= \sup_{r \in (t,T]}|\cE(\hat\beta A)^{1/2}_rY_r|^2 \int_t^T \frac{(1+\beta \Delta A_r)\cE(\beta A)_{r\smallertext{-}}}{(1+\hat\beta \Delta A_r)\cE(\hat\beta A)_{r\smallertext{-}}} \d A_r \\
	& \leq\sup_{r \in (t,T]}|\cE(\hat\beta A)^{1/2}_rY_r|^2 \int_t^T \frac{\cE(\beta A)_{r\smallertext{-}}}{\cE(\hat\beta A)_{r\smallertext{-}}} \d A_r \leq \frac{(1+\hat\beta\Phi)}{(\hat\beta-\beta)}\sup_{r \in (t,T]}|\cE(\hat\beta A)^{1/2}_{r}Y_r|^2.
\end{align*}
This now yields, after changing the constant $\mathfrak{C}(\varepsilon,\kappa,\beta,\hat\beta,\Phi,\varpi)$ again if necessary,
\begin{align*}
	&\E^\P\bigg[\int_t^T\cE(\beta A)_{r}|Y_r|^2\d A_r + \int_t^T \cE(\beta A)_{r}|Y_{r-}|^2 \d A_r + \int_t^T \cE(\beta A)_r \d \langle\eta^\P\rangle_r + \int_t^T\cE(\beta A)_r \d[ K^\P]_r \bigg| \cF_{t\smallertext{+}} \bigg] \nonumber\\
	& \leq \mathfrak{C}(\varepsilon,\kappa,\beta,\hat\beta,\Phi,\varpi) \Bigg(\E^\P\bigg[ \cE(\beta A)_T|\xi|^2
	+ \sup_{r \in [t,T]} |\cE(\hat\beta A)^{1/2}_r Y_r|^2 + \int_t^T\cE(\beta A)_{r}\frac{|f^{\P,0}_r|^2}{\alpha^2_r}\d C_r\bigg|\cF_{t\smallertext{+}}\bigg] \Bigg), \; \textnormal{$\P$--a.s.}
\end{align*}
This concludes the proof.
\end{proof}
	
	The desired bound now follows by choosing $t = 0$ in \Cref{lem::conditional_bound_2bsde}, taking expectations, and using \cite[Theorem~3.5.$(ii)$]{possamai2026aggregation} together with the bound
	\begin{equation*}
		\|\xi\|^2_{\L^\smalltext{2}_{\smalltext{T}\smalltext{,}\smalltext{\beta}}(\P)} + \bigg\|\frac{f^{\P,0}}{\alpha}\bigg\|^2_{\H^\smalltext{2}_{\smalltext{T}\smalltext{,}\smalltext{\beta}}(\P)} \leq \phi^{2,\hat{\beta}}_{\xi, f}.
	\end{equation*}
	This completes the proof of \Cref{prop::bound_and_comparison}.
	\end{proof}	
	
\subsection{Well-posedness of the intrinsic system}

The proof of \Cref{thm::2BSDE_wellposed_2} relies on several lemmata presented in this section. We start with the one giving a bound for the inverse of a stochastic exponential, which may be of independent interest.

\begin{lemma}\label{eq::inverse_stochastic_exponential}
	Let $(\Omega,\cF,\F=(\cF_u)_{u \in [0,\infty)},\P)$ be a filtered probability space satisfying the usual conditions and supporting a right-continuous, locally square-integrable martingale $M = (M_u)_{u \in [0,\infty)}$ satisfying $\Delta M > -1 + \delta$ outside a $\P$--null set, for some $\delta \in (0,\infty)$. Let $K \coloneqq (1+\Delta M)^{-1}$, denote its $(\F,\P)$--predictable projection by ${}^{p}\!K$, and suppose that there exists $\mathfrak{C} \in (0,\infty)$ such that
	\begin{equation}\label{eq_log_D_bound}
		\log(D) \coloneqq \langle M^c \rangle + \Bigg(\sum_{s \in (0,\cdot]} \bigg(\log({}^{p}\!K_s) + \frac{K_s}{\displaystyle {}^{p}\!K_s} - 1 + \frac{\Delta M_s}{\displaystyle {}^{p}\!K_s}\bigg)\Bigg)^p \leq \mathfrak{C}, \; \textnormal{$\P$--a.s.},
	\end{equation}
	where $(\ldots)^p$ denotes the $(\F,\P)$--dual predictable projection.
	Then
	\begin{equation*}
		\E^\P\bigg[\sup_{u \in [t, \infty)} \cE(M)_t\cE(M)^{\smallertext{-}1}_u \bigg| \cF_t\bigg] \leq  4\mathrm{e}^{\mathfrak{C}}, \; \text{$\P$--{\rm a.s.}}, \; t \in [0,\infty).
	\end{equation*}
\end{lemma}
		
\begin{proof}
	Since $\Delta M \geq -1 + \delta$, the processes $\cE(M)$ and $\cE(M)^{\smallertext{-}1}$ are both strictly positive outside a $\P$--null set. In particular, $\cE(M)$ is thus an $(\F,\P)$--super-martingale. By \cite[Th\'eor\`eme 6.36]{jacod1979calcul}, we can write
	\[
		\cE(M)^{-1} = \cE(\tilde{N}) D \leq \cE(\tilde{N}) \mathrm{e}^\mathfrak{C}, \; \textnormal{$\P$--a.s.},
	\]
	for an $(\F,\P)$--local martingale $\tilde{N}$ starting at zero, where $D$ is the (non-decreasing) process starting at $1$ defined in the statement of the lemma. 
	The process $\cE(\tilde{N})$, being a non-negative $(\F,\P)$--local martingale starting at $1$, is therefore an $(\F,\P)$--super-martingale.
	This implies, in particular, that $\cE(M)^{\smallertext{-}1}_s$ is $\P$-integrable for every $s \in [0,\infty)$. It follows that $\cE(M)^{\smallertext{-}1/2}$ is an $(\F,\P)$--sub-martingale since $x \longmapsto x^{\smallertext{-}1/2}$ is a strictly decreasing convex function defined on $(0,\infty)$, and therefore
	\[
		\cE(M)^{\smallertext{-}1/2}_s \leq \big(\E^\P[\cE(M)_{t}|\cF_s]\big)^{\smallertext{-}1/2} \leq \E^\P[\cE(M)^{\smallertext{-}1/2}_{t}|\cF_s], \; \textnormal{$\P$--a.s.}, \; 0 \leq s \leq t < \infty.
	\]
	Here, we used Jensen's inequality in the form \cite[Theorem 15.3, page 116]{bauer1996probability}. It then follows from a conditional version of Doob's $\L^2$-inequality for sub-martingales (see \cite[page 444]{doob1984classical}) that
	\begin{align*}
		\E^\P\bigg[\sup_{u \in [t, t^\smalltext{\prime}]} \big\{\cE(M)_t \cE(M)^{\smallertext{-}1}_u\big\} \bigg| \cF_t\bigg] 
		&= \cE(M)_t \E^\P\bigg[\sup_{u \in [t, t^\smalltext{\prime}]} \cE(M)^{\smallertext{-}1}_u \bigg| \cF_t\bigg] \\
		&\leq 4  \cE(M)_t \E^\P\big[\cE(M)^{\smallertext{-}1}_{t^\smalltext{\prime}}\big|\cF_t\big] \\
		&\leq 4\mathrm{e}^\mathfrak{C} \cE(M)_t \E^\P\big[\cE(\tilde{N})_{t^\smalltext{\prime}}\big|\cF_t\big] \\
		&\leq 4\mathrm{e}^\mathfrak{C} \cE(M)_t \cE(\tilde{N})_t
		= 4\mathrm{e}^\mathfrak{C} \frac{1}{D_t}
		\leq 4\mathrm{e}^\mathfrak{C} \frac{1}{D_0} = 4\mathrm{e}^\mathfrak{C}, \; \textnormal{$\P$--a.s.}, \; 0 \leq t < t^\prime < \infty.
	\end{align*}
	By monotone convergence, we obtain
	\[
		\E^\P
		\bigg[\sup_{u \in [t, \infty)} \big\{\cE(M)_t \cE(M)^{\smallertext{-}1}_u\big\} \bigg| \cF_t\bigg] 
		\leq 4\mathrm{e}^\mathfrak{C}, \; \text{$\P$--a.s.}, \; t \in [0,\infty),
	\] 
	which concludes the proof.
\end{proof}

\begin{remark}\label{rem::inverse_exponential}
	Suppose that $M$ is $(\F,\P)$--quasi--left-continuous. Then $M$ does not jump at $\F$--predictable stopping times, and therefore ${}^{p}\!K \equiv 1$. A sufficient condition for \textnormal{\eqref{eq_log_D_bound}} to hold is that $\langle M \rangle$ is $\P$--essentially bounded, say by $\mathfrak{C}^\prime \in (0,\infty)$. Indeed, it follows from \textnormal{\cite[Lemma A.4]{kazi2015second}} that there exists a constant $\mathfrak{C}(\delta) \in (0,\infty)$ only depending on $\delta$ such that
	\[
		\log(D) \leq \langle M^c \rangle + \mathfrak{C}(\delta)\Bigg(\sum_{s \in (0,\cdot]}(\Delta M_s)^2 \Bigg)^p 
		\leq \max\{1,\mathfrak{C}(\delta)\}\langle M \rangle \leq\max\{1,\mathfrak{C}(\delta)\} \mathfrak{C}^\prime
	\]
	outside a $\P$--null set.		
\end{remark}

\begin{lemma}\label{lem_characteristics_under_F_plus_pasting}
	Let $\P\in\fP_0$, $t \in [0,\infty)$, $B \in \cF_{t\smallertext{+}}$, and $(\P_1,\P_2)\in\big(\fP_0(\cF_{t\smallertext{+}},\P)\big)^2$. Let $\overline{\P}\in\fP_0$ be given by $\overline{\P}[D] \coloneqq \E^\P\big[\P_1[D|\cF_{t\smallertext{+}}]\1_B + \P_2[D|\cF_{t\smallertext{+}}]\1_{B^c}\big]$. Then
	\[
		\mathsf{b}^{\bar{\P}} = \mathsf{b}^{\P_\smalltext{1}}\1_{B} + \mathsf{b}^{\P_\smalltext{2}}\1_{B^c} \; \textnormal{and} \; \mathsf{K}^{\bar{\P}}(\d x) = \mathsf{K}^{\P_\smalltext{1}}(\d x)\1_{B} + \mathsf{K}^{\P_\smalltext{2}}(\d x)\1_{B^c}, \; \textnormal{$\overline{\P}\otimes\d C$--a.e. on $\llparenthesis t,\infty\rrparenthesis$.}
	\]
\end{lemma}

\begin{proof}
	We denote by $h$ the truncation function to define the semi-martingale characteristics of $X$. Set $$X^h \coloneqq X - \sum_{s \in (0,\cdot]} (\Delta X_s - h(\Delta X_s)).$$
	Note that $\P_1 = \P_2 = \P$ on $\cF_{t\smallertext{+}}$ and that $\overline{\P}$ coincides with $\P_1$ on $B$ and with $\P_2$ on $B^c$ since $\overline{\P}[D] = \P_1[D \cap B] + \P_2[D \cap B^c]$. Moreover, $\1_{B \times (t,\infty)}$ and $\1_{B^c\times(t,\infty)}$ are both $\F$-predictable.

	\medskip
	We first consider the compensator. Let $$\bar{\nu} (\omega;\d s, \d x) \coloneqq \big(\1_{B\times(t,\infty)}(\omega,s) \mathsf{K}^{\P_\smalltext{1}}_{\omega,s}(\d x) + \1_{B^c\times(t,\infty)}(\omega,s) \mathsf{K}^{\P_\smalltext{2}}_{\omega,s}(\d x)\big) \d C_s(\omega).$$ Then $\overline{\nu}$ is an $\F$-predictable random measure (see \cite[Definition II.1.6]{jacod2003limit}). For an arbitrary $\cP(\F)\otimes\cB(\R^d)$-measurable and nonnegative function $W$, we set $W^\prime_s(\omega;x)\coloneqq \1_{(t,\infty)}(s) W_s(\omega;x)$. Then
	\[
		\E^{\bar{\P}}\big[W^\prime\ast\mu^X_\infty\big] = \E^{\P_\smalltext{1}}\big[\1_B W^\prime\ast\mu^X_\infty\big] + \E^{\P_\smalltext{2}}\big[\1_{B^\smalltext{c}} W^\prime\ast\mu^X_\infty\big] = \E^{\P_\smalltext{1}}\big[\1_B W^\prime\ast\nu^{\P_\smalltext{1}}_\infty\big] + \E^{\P_\smalltext{2}}\big[\1_{B^\smalltext{c}} W^\prime\ast\nu^{\P_\smalltext{2}}_\infty\big] = \E^{\bar{\P}}\big[W\ast\bar{\nu}_\infty\big].
	\]
	Hence 
	\[
		\bar{\nu}(\d s, \d x) = \1_{(t,\infty)}(s)\nu^{\overline{\P}}(\d s, \d x), \; \textnormal{$\overline{\P}$--a.s.,}
	\]
	by uniqueness of the $\overline{\P}$--predictable compensator; see \cite[Theorem~II.1.8]{jacod2003limit}. Since there exists a strictly positive, $\cP(\F)\otimes\cB(\R^d)$-measurable function $V$ such that $V \ast \mu^X_\infty \leq 1$, we have that $V \ast \bar{\nu}_\infty = (V\1_{\llparenthesis t, \infty \rrparenthesis \times \R^\smalltext{d}})\ast\nu^{\bar{\P}}_\infty < \infty$, $\overline{\P}$--a.s., and therefore
	\[
		\1_B(\omega)\mathsf{K}^{\P_\smalltext{1}}_{\omega,s}(\d x) + \1_{B^\smalltext{c}}(\omega)\mathsf{K}^{\P_\smalltext{2}}_{\omega,s}(\d x) = \mathsf{K}^{\bar{\P}}_{\omega,s}(\d x), \; \textnormal{for $\overline{\P}\otimes \d C$--a.e. $(\omega,s) \in \llparenthesis t, \infty \rrparenthesis$}
	\]
	by \cite[Lemma~A.8]{possamai2026aggregation}.

	\medskip
	We turn to the first characteristic. By definition, $M^{\P_\smalltext{i}} \coloneqq X^h - \int_0^\cdot \mathsf{b}^{\P_\smalltext{i}}_r \d C_r$ is an $(\F_\smallertext{+},\P_i)$--local martingale for $i \in \{1,2\}$. Define $\overline{M} = (\overline{M}_s)_{s \in [0,\infty)}$ by
	\[
		\overline{M}_s \coloneqq \1_B (M^{\P_\smalltext{1}}_{s} - M^{\P_\smalltext{1}}_{s \land t}) + \1_{B^\smalltext{c}}(M^{\P_\smalltext{2}}_{s} - M^{\P_\smalltext{2}}_{s \land t}).
	\]
	Then
	\[
		\overline{M}_s = X^h_s - X^h_{s \land t} - \int_t^s (\1_B \mathsf{b}^{\P_\smalltext{1}}_r + \1_{B^\smalltext{c}}\mathsf{b}^{\P_\smalltext{2}}_r) \d C_r
	\]
	Since $M^{\bar{\P}} = (M^{\bar{\P}}_s)_{s \in [0,\infty)}$ given by
	\[
		M^{\bar{\P}}_s \coloneqq X^h_s - X^h_{s \land t} - \int_t^s \mathsf{b}^{\bar{\P}}_r \d C_r
	\]
	is an $(\F_\smallertext{+},\overline{\P})$--local martingale, it suffices to show that $\overline{M}$ is so too since then $\overline{M} - M^{\bar{\P}}$ will be an $\F$-predictable, $(\F_\smallertext{+},\overline{\P})$--local martingale starting at zero with paths that are $\overline{\P}$--a.s. of locally finite variation, and therefore, $\overline{\P}$--a.s., $\overline{M} = M^{\bar{\P}}$. Let $(\tau^n_i)_{n \in \N}$ be an $(\F_\smallertext{+},\P_i)$--localising sequence such that $M^{\P_\smalltext{i}}_{\cdot\land \tau^\smalltext{n}_\smalltext{i}}$ is an $(\F_\smallertext{+},\P_i)$--uniformly integrable martingale for $i \in \{1,2\}$. Let $\bar{\tau}^n = (\tau^n_1 \1_B + \tau^n_2 \1_{B^\smalltext{c}})\lor t$. Then $(\bar{\tau}^n)_{n \in \N}$ is an $(\F_\smallertext{+},\overline{\P})$--localising sequence since each $\bar{\tau}^n$ is an $\F_\smallertext{+}$--stopping time and
	\[
		\overline{\P}[\bar{\tau}^n \uparrow \infty] 
		= \P_1[\{\tau^n_1 \lor t \uparrow \infty\} \cap B] + \P_2[\{\tau^n_2 \lor t \uparrow \infty\} \cap B^c] =  1.
	\]
	For $s \in [0,\infty)$, we have
	\[
		\overline{M}_{s \land \bar{\tau}^\smalltext{n}} = \overline{M}_{s\land\tau^\smalltext{n}_\smalltext{1}}\1_B + \overline{M}_{s\land\tau^\smalltext{n}_\smalltext{2}}\1_{B^\smalltext{c}},
	\]
	and therefore
	\begin{align*}
		&\E^{\bar{\P}}[|\overline{M}_{s\land \bar{\tau}^\smalltext{n}}\1_{\{|\overline{M}_{s\land \bar{\tau}^\smalltext{n}}| \geq N\}}|] \\
		&= \E^{\P_\smalltext{1}}[\1_B|M^{\P_\smalltext{1}}_{s\land \tau^\smalltext{n}_\smalltext{1}}-M^{\P_\smalltext{1}}_{s\land t\land \tau^\smalltext{n}_\smalltext{1}}|\1_{\{|M^{\P_\smalltext{1}}_{s\land \tau^\smalltext{n}_\smalltext{1}}-M^{\P_\smalltext{1}}_{s\land t\land \tau^\smalltext{n}_\smalltext{1}}| \geq N\}}] + \E^{\P_\smalltext{2}}[\1_{B^\smalltext{c}}|M^{\P_\smalltext{2}}_{s\land \tau^\smalltext{n}_\smalltext{2}}-M^{\P_\smalltext{2}}_{s\land t\land \tau^\smalltext{n}_\smalltext{2}}|\1_{\{|M^{\P_\smalltext{2}}_{s\land \tau^\smalltext{n}_\smalltext{2}}-M^{\P_\smalltext{2}}_{s\land t\land \tau^\smalltext{n}_\smalltext{2}}| \geq N\}}],
	\end{align*}
	which converges uniformly in $s \in [0,\infty)$ to zero as $N$ tends to infinity. This shows $\overline{\P}$--uniform integrability of $(\overline{M}_{s \land \bar{\tau}^\smalltext{n}})_{s \in [0,\infty)}$. If $t < r < s < \infty$ and $A \in \cF_{r\smallertext{+}}$, then
	\begin{align*}
		\E^{\bar{\P}}\big[\1_A (\overline{M}_{s\land \bar{\tau}^\smalltext{n}} - \overline{M}_{r\land \bar{\tau}^\smalltext{n}})\big] 
		&= \E^{\P_\smalltext{1}}\big[\1_B\1_A (M^{\P_\smalltext{1}}_{s\land \tau^\smalltext{n}_\smalltext{1}} - M^{\P_\smalltext{1}}_{r\land \tau^\smalltext{n}_\smalltext{1}})\big] + \E^{\P_\smalltext{2}}\big[\1_{B^\smalltext{c}}\1_A (M^{\P_\smalltext{2}}_{s\land \tau^\smalltext{n}_\smalltext{2}} - M^{\P_\smalltext{2}}_{r\land \tau^\smalltext{n}_\smalltext{2}})\big] = 0,
	\end{align*}
	since $B \in \cF_{t\smallertext{+}} \subseteq \cF_{r\smallertext{+}}$. For $A \in \cF_{t\smallertext{+}}$, we let $r \downarrow\downarrow t$ and use $\overline{\P}$--uniform integrability together with \cite[Theorem~4, p.~225]{shiryaev2016probability} to obtain
	\[
		\E^{\bar{\P}}\big[\1_A (\overline{M}_{s\land \bar{\tau}^\smalltext{n}} - \overline{M}_{t\land \bar{\tau}^\smalltext{n}})\big] = \E^{\bar{\P}}\big[\1_A \overline{M}_{s\land \bar{\tau}^\smalltext{n}}\big] = 0 = \E^{\bar{\P}}\big[\1_A \overline{M}_{t\land \bar{\tau}^\smalltext{n}}\big].
	\]
	If $0\leq r < t < s < \infty$ or $0 \leq r < s \leq t$, and $A \in \cF_{r\smallertext{+}}$, then we have
	\[
		\E^{\bar{\P}}\big[\1_A (\overline{M}_{s\land \bar{\tau}^\smalltext{n}} - \overline{M}_{r\land \bar{\tau}^\smalltext{n}})\big] = \E^{\bar{\P}}\big[\1_A \overline{M}_{s\land \bar{\tau}^\smalltext{n}}\big] = 0.
	\]
	This now proves that $\overline{M}_{\cdot \land \bar{\tau}^\smalltext{n}}$ is an $(\F_\smallertext{+},\overline{\P})$--uniformly integrable martingale, which completes the proof.
\end{proof}

\begin{lemma}\label{lem::bound_k_2bsde}
	Suppose that the assumptions stated in \textnormal{\Cref{thm::2BSDE_wellposed_2}} are satisfied. Let $\tilde{\beta} \in (0,\infty)$ satisfy $\beta^\star < \tilde{\beta} < \beta < \hat\beta$, and let $(Y,Z,(U^\P,N^\P,K^\P)_{\P\in\fP_0})$ satisfy \textnormal{\ref{2BSDE::measurability}} and \textnormal{\ref{2BSDE::dynamics}}, $\sup_{\P\in\fP_0} \|Y\|_{\cS^2_{T,\hat{\beta}}(\G_\smalltext{+},\P)} < \infty$, and $(Z,(U^\P,N^\P)_{\P\in\fP_0}) \in \sL^2_{T,\beta}(\fP_0)$. Then
	\begin{align*}
		&\E^\P\Bigg[\underset{\bar{\P} \in \fP_\smalltext{0}(\cG_{\smalltext{t}\tinytext{+}},\P)}{{\esssup}^{\P}} \E^{\bar{\P}}\bigg[ \bigg( \int_{t}^{T} \cE(\tilde{\beta} A)^{1/2}_s \d K^{\bar{\P}}_s\bigg)^{2}\bigg|\cG_{t\smallertext{+}}\bigg] \Bigg] < \infty,
		\; t \in [0,\infty), \; \P \in \fP_0.
	\end{align*}
\end{lemma}

\begin{proof}
	We fix $\P\in\fP_0$, $t \in [0,\infty)$ and $\overline{\P}\in\fP_0(\cG_{t\smallertext{+}},\P)$ for the time being. Let
	\[
	\eta^{\bar\P} 
	\coloneqq (Z\bcdot X^{c,\bar{\P}})^{(\bar{\P})} + (U^{\bar{\P}}\ast\tilde{\mu}^{X,\bar{\P}})^{(\bar{\P})} + N^{\bar{\P}}, \; f^{\bar{\P}} \coloneqq f^{\bar{\P}}(Y,Y_{\smallertext{-}},Z,U^{\bar{\P}}(\cdot)), \; f^{\bar{\P},0} \coloneqq f^{\bar{\P}}(0,0,0,\mathbf{0}),
	\]
	for simplicity. An application of the integration by parts formula under $\overline{\P}$ yields
	\[
	\d(\cE(\tilde{\beta} A)^{1/2}  Y) =  Y_{\smallertext{-}}\d\cE(\tilde{\beta} A)^{1/2} - \cE(\tilde{\beta} A)^{1/2} f^{\bar{\P}}\d C + \cE(\tilde{\beta} A)^{1/2}\d\eta^{\bar{\P}} - \cE(\tilde{\beta} A)^{1/2}\d K^{\bar{\P}}, \; \textnormal{$\overline{\P}$--a.s.,}
	\]
	and then, $\overline{\P}$--a.s., for $0 \leq t \leq t^\prime < \infty$,
	\begin{align*}
		&\int_t^{t^\prime\land T} \cE(\tilde{\beta} A)^{1/2}_s\d K^{\bar{\P}}_s \\
		&\quad = -\int_t^{t^\prime \land T}\d(\cE(\tilde{\beta} A)^{1/2} Y)_s 
		+ \int_t^{t^\prime \land T}  Y_{s\smallertext{-}} \d\cE(\tilde{\beta} A)^{1/2}_s - \int_t^{t^\prime \land T}\cE(\tilde{\beta} A)^{1/2}_s f^{\bar{\P}}_s\d C_s + \int_t^{t^\prime \land T} \cE(\tilde{\beta} A)^{1/2}_s\d\eta^{\bar{\P}}_s.
	\end{align*}
	Next,
	\begin{align*}
		\bigg(\int_t^{t^\prime \land T}\cE(\tilde{\beta} A)^{1/2}_s f^{\bar{\P}}_s\d C_s\bigg)^2
		&\leq \bigg(\int_t^{t^\prime \land T}\frac{\cE(\tilde{\beta} A)^{1/2}_s}{\cE(\beta A)^{1/2}_s}\alpha_s \cE(\beta A)^{1/2}_s \frac{|f^{\bar{\P}}_s|}{\alpha_s}\d C_s\bigg)^2 \\
		&\leq \int_t^{t^\prime \land T}\frac{\cE(\tilde{\beta} A)_s}{\cE(\beta A)_s}\d A_s \int_t^{t^\prime \land T} \cE(\beta A)_s \frac{|f^{\bar{\P}}_s|^2}{\alpha^2_s}\d C_s \\
		&= \int_t^{t^\prime \land T}\frac{\cE(\tilde{\beta} A)_{s\smallertext{-}}(1+\tilde{\beta}\Delta A_s)}{\cE(\beta A)_{s\smallertext{-}}(1+\beta\Delta A_s)}\d A_s \int_t^{t^\prime \land T} \cE(\beta A)_s \frac{|f^{\bar{\P}}_s|^2}{\alpha^2_s}\d C_s \\
		&\leq \int_t^{t^\prime \land T}\frac{\cE(\tilde{\beta} A)_{s\smallertext{-}}}{\cE(\beta A)_{s\smallertext{-}}}\d A_s \int_t^{t^\prime \land T} \cE(\beta A)_s \frac{|f^{\bar{\P}}_s|^2}{\alpha^2_s}\d C_s \\
		&\leq \frac{(1+\beta\Phi)}{(\beta-\tilde{\beta})} \int_t^{t^\prime \land T} \cE(\beta A)_s \frac{|f^{\bar{\P}}_s|^2}{\alpha^2_s}\d C_s, \; 0 \leq t \leq t^\prime < \infty.
	\end{align*}
	Here, we used Cauchy--Schwarz's inequality to obtain the second inequality, and to bound the integral term containing only the stochastic exponentials, we applied an inequality established in the proof of \cite[Lemma 3.3]{possamai2024reflections}. It follows from the arguments that lead to \cite[Equation (5.46)]{possamai2024reflections} that there exists a constant $\mathfrak{j}(\tilde{\beta},\beta,\Phi) \in (0,\infty)$ only depending on $(\tilde{\beta},\beta,\Phi)$ such that
	\begin{align*}
		\bigg(\int_t^{t^\prime \land T}  Y_{s\smallertext{-}} \d\cE(\tilde{\beta} A)^{1/2}_s\bigg)^2
		\leq \mathfrak{j}(\tilde{\beta},\beta,\Phi) \sup_{s \in (t,t^\prime)} |\cE(\beta A)^{1/2}_s Y_s|^2.
	\end{align*}
	Moreover, we have
	\[
	\E^{\bar{\P}}\bigg[\bigg(\int_t^{t^\prime \land T} \cE(\tilde{\beta} A)^{1/2}_s\d\eta^{\bar{\P}}_s\bigg)^2\bigg|\cG_{t\smallertext{+}}\bigg] 
	= \E^{\bar{\P}}\bigg[\int_t^{t^\prime \land T} \cE(\tilde{\beta} A)_s\d\langle\eta^{\bar{\P}}\rangle^{(\bar{\P})}_s\bigg|\cG_{t\smallertext{+}}\bigg], \; \textnormal{$\overline{\P}$--a.s.}
	\]
	It follows from monotone convergence, from the Lipschitz-continuity property of the generator, and from \Cref{lem::conditional_bound_2bsde}, 
	that there exists a constant $\mathfrak{C}(\tilde{\beta},\beta,\hat{\beta},\Phi) \in (0,\infty)$ only depending on $(\tilde{\beta},\beta,\hat{\beta},\Phi)$, which may change from line to line, such that 
	\begin{align*}
		& \E^{\bar{\P}}\bigg[\bigg(\int_t^{T} \cE(\tilde{\beta} A)^{1/2}_s\d K^{\bar{\P}}_s\bigg)^2\bigg|\cG_{t\smallertext{+}}\bigg] \\
		&\leq \mathfrak{C}(\tilde{\beta},\beta,\hat\beta,\Phi) \E^{\bar{\P}}\bigg[\sup_{r \in [t,T]} |\cE(\beta A)^{1/2}_r  Y_r|^2 + \int_t^{T} \cE(\tilde{\beta} A)_s\d\langle\eta^{\bar{\P}}\rangle^{(\bar{\P})}_s + \int_t^T\cE(\beta A)_{r}\frac{|f^{\bar{\P}}_r|^2}{\alpha^2_r}\d C_r\bigg|\cG_{t\smallertext{+}}\bigg] \\
		&\leq \mathfrak{C}(\tilde{\beta},\beta,\hat\beta,\Phi) \E^{\bar{\P}}\bigg[\cE(\beta A)_T|\xi|^2 + \sup_{r \in [t,T]} |\cE(\hat{\beta} A)^{1/2}_r  Y_r|^2 +  \int_t^T\cE(\beta A)_{r}\frac{|f^{\bar{\P}}_r(0,0,0,\mathbf{0})|^2}{\alpha^2_r}\d C_r\bigg|\cG_{t\smallertext{+}}\bigg] \\
		&\leq \mathfrak{C}(\tilde{\beta},\beta,\hat\beta,\Phi) \E^{\bar{\P}}\bigg[\sup_{r \in [0,T]} |\cE(\hat{\beta} A)^{1/2}_r  Y_r|^2 +  \int_t^T\cE(\beta A)_{r}\frac{|f^{\bar{\P}}_r(0,0,0,\mathbf{0})|^2}{\alpha^2_r}\d C_r\bigg|\cG_{t\smallertext{+}}\bigg], \; \textnormal{$\overline{\P}$--a.s. and $\P$--a.s.}
	\end{align*}
	Since $\overline{\P}\in\fP_0(\cG_{t\smallertext{+}},\P)$ was arbitrary, we obtain
	\begin{align*}
		& \underset{\P^\smalltext{\prime} \in \fP_\smalltext{0}(\cG_{\smalltext{t}\tinytext{+}},\P)}{{\esssup}^{\P}} \E^{\P^\smalltext{\prime}}\bigg[\bigg(\int_t^{T} \cE(\tilde{\beta} A)^{1/2}_s\d K^{\P^\smalltext{\prime}}_s\bigg)^2\bigg|\cG_{t\smallertext{+}}\bigg] \\
		& \leq \mathfrak{C}(\tilde{\beta},\beta,\hat\beta,\Phi) \underset{\P^\smalltext{\prime} \in \fP_\smalltext{0}(\cG_{\smalltext{t}\tinytext{+}},\P)}{{\esssup}^{\P}} \E^{\P^\smalltext{\prime}}\bigg[\sup_{r \in [0,T]} |\cE(\hat{\beta} A)^{1/2}_r  Y_r|^2 +  \int_t^T\cE(\beta A)_{r}\frac{|f^{\P^\smalltext{\prime}}_r(0,0,0,\mathbf{0})|^2}{\alpha^2_r}\d C_r\bigg|\cG_{t\smallertext{+}}\bigg], \; \textnormal{$\P$--a.s.}
	\end{align*}
	The family of conditional expectations on the right-hand side, indexed by $\P^\prime \in \fP_0(\cG_{t\smallertext{+}},\P)$, is $\P$--upward directed. This follows from a standard argument using \Cref{ass::intrinsic_2bsde}.$(ii)$ and \Cref{lem_characteristics_under_F_plus_pasting}. Thus the $\P$--essential supremum can be written as a $\P$--a.s. increasing limit. This yields
	\begin{align*}
		&\sup_{\P\in\fP_0}\E^\P\Bigg[ \underset{\P^\smalltext{\prime} \in \fP_\smalltext{0}(\cG_{\smalltext{t}\tinytext{+}},\P)}{{\esssup}^{\P}} \E^{\P^\smalltext{\prime}}\bigg[\bigg(\int_t^{T} \cE(\tilde{\beta} A)^{1/2}_s\d K^{\P^\smalltext{\prime}}_s\bigg)^2\bigg|\cG_{t\smallertext{+}}\bigg]\Bigg] \\
		&\leq  \mathfrak{C}(\tilde{\beta},\beta,\hat\beta,\Phi) \sup_{\P\in\fP_0} \E^{\P}\bigg[\sup_{r \in [0,T]} |\cE(\hat{\beta} A)^{1/2}_r Y_r|^2 +  \int_t^T\cE(\beta A)_{r}\frac{|f^{\P}_r(0,0,0,\mathbf{0})|^2}{\alpha^2_r}\d C_r\bigg] < \infty.
	\end{align*}
	The finiteness of the last term follows from $\phi_{\xi,f}^{2,\hat{\beta}} < \infty$ and the assumption $\sup_{\P\in\fP_0} \|Y\|_{\cS^2_{T,\hat{\beta}}(\G_\smalltext{+},\P)} < \infty$. This concludes the proof.
\end{proof}
	
\begin{remark}\label{rem_increasing_process_bound}
	The condition in \textnormal{\Cref{ass::intrinsic_2bsde}.$(ii)$} is crucial to conclude
	\begin{equation*}
		\underset{\bar{\P} \in \fP_\smalltext{0}(\cG_{\smalltext{t}\tinytext{+}},\P)}{{\esssup}^{\P}} \E^{\bar{\P}}\bigg[ \bigg( \int_{t}^{T} \cE(\tilde{\beta} A)^{1/2}_s \d \widehat{K}^{\bar{\P}}_s\bigg)^{2}\bigg|\cG_{t\smallertext{+}}\bigg] < \infty, \; \textnormal{$\P$--a.s.}, \; t \in [0,\infty), \; \P\in\fP_0.
	\end{equation*}
	In the proofs of \textnormal{\cite[Theorem 4.1 and 4.2]{possamai2018stochastic}}, $\F_\smallertext{+}$-stability of $\fP_0$ is invoked$:$ for all $\P \in \fP_0$, $t \in [0,\infty)$, $B \in \cF_{t\smallertext{+}}$, and $(\P_1,\P_2) \in \big(\fP_0(\cF_{t\smallertext{+}},\P)\big)^2$, we have $\overline{\P}\in\fP_0$ for
	\[
	\overline{\P}[A] \coloneqq \E^\P\big[\P_1[A|\cF_{t\smallertext{+}}]\1_B + \P_2[A|\cF_{t\smallertext{+}}] \1_{B^\smalltext{c}}\big], \; A \in \cF.
	\]
	This stability is, however, not automatic in general. Nonetheless, \textnormal{\cite[Assumption~3.1]{possamai2018stochastic}} is similar to \textnormal{\Cref{ass::intrinsic_2bsde}.$(ii)$}, and thus the proofs can be modified accordingly.
\end{remark}

\begin{lemma}\label{prop::minimality}
	Suppose that the assumptions of \textnormal{\Cref{thm::2BSDE_wellposed_2}} hold.
	The family of processes $(\widehat{K}^\P)_{\P \in \fP_0}$ from \textnormal{\Cref{thm::2BSDE_wellposed}} satisfies \ref{2BSDE::minimality_condition_2}, that is,
	\begin{equation*}
		\underset{\bar{\P} \in \fP_\smalltext{0}(\cG_{\smalltext{t}\tinytext{+}},\P)}{{\essinf}^{\P}} \E^{\bar{\P}}\bigg[\int_t^T \cE\Big(\int_t^\cdot \lambda^{\bar{\P}}_s\d C_s\Big)_{r\smalltext{-}} \d\widehat{K}^{\bar{\P}}_r \bigg| \cG_{t\smallertext{+}}\bigg] = 0, \; \textnormal{$\P$--a.s.}, \; t \in [0,\infty], \; \P \in \fP_0,
	\end{equation*}
	where
	\[
		\lambda^{\bar{\P}} \coloneqq \frac{f^{\bar{\P}}(\widehat{Y},\widehat{Y}_{\smallertext{-}},\widehat{Z},\widehat{U}^{\bar{\P}}(\cdot)) - f^{\bar{\P}}(\cY^{\bar{\P}},\widehat{Y}_{\smallertext{-}},\widehat{Z},\widehat{U}^{\bar{\P}}(\cdot))}{\widehat{Y}-\cY^{\bar{\P}}} \1_{\{\widehat{Y} \neq \cY^\smalltext{\bar{\P}}\}}, \; \overline{\P} \in \fP_0.
	\]
\end{lemma}

	\begin{proof}[Proof of \Cref{prop::minimality}]
		Let $\P \in \fP_0$, and fix $t \in [0,\infty)$ and an arbitrary $\bar{\P} \in \fP_0(\cG_{t\smallertext{+}}, \P)$; for $t = \infty$, the conclusion of the lemma follows immediately. We write
		\begin{gather*}
			\delta Y \coloneqq \widehat{Y} - \cY^{\bar{\P}}, \; \delta Z \coloneqq \widehat{Z}-\cZ^{\bar{\P}}, \; \delta U \coloneqq \widehat{U}^{\bar{\P}} - \cU^{\bar{\P}}, \; \delta N \coloneqq \widehat{N}^{\bar{\P}} - \cN^{\bar{\P}}, \; K \coloneqq \widehat{K}^{\bar{\P}}, \; X^c \coloneqq X^{c,\bar{\P}}, \; \tilde{\mu}^{X} \coloneqq \tilde{\mu}^{X,\bar{\P}},\\
			\delta f \coloneqq f^{\bar{\P}}\big(\widehat{Y},\widehat{Y}_\smallertext{-}, \widehat{Z}, \widehat{U}^{\bar{\P}}(\cdot)\big) - f^{\bar{\P}}\big(\cY^{\bar{\P}},\cY^{\bar{\P}}_\smallertext{-}, \cZ^{\bar{\P}}, \cU^{\bar{\P}}(\cdot)\big),
		\end{gather*}
		and, see \eqref{eq::definition_lambda_hat_lambda},
		\begin{equation*}
			\lambda_s(\omega) \coloneqq \lambda^{\bar{\P}}_s(\omega), 
			\; \widehat\lambda_s(\omega) \coloneqq \widehat\lambda^{\widehat{Y}_{\smalltext{s}\tinytext{-}}(\omega),\cY^{\bar{\P}}_{\smalltext{s}\tinytext{-}}(\omega)}_s(\omega), 
			\; \eta_s(\omega) \coloneqq \eta^{\widehat{Z}_\smalltext{s}(\omega),\cZ^{\bar{\P}}_\smalltext{s}(\omega),\mathsf{a}_\smalltext{s}(\omega)}_s(\omega).
		\end{equation*}
		Now consider $v \coloneqq \int_0^{\cdot \land T} \gamma_s \d C_s,$ where  $\gamma \coloneqq {\widehat\lambda}/{(1-\widehat\lambda\Delta C)}.$ Here, we note that the process $\gamma$ is $\G_\smallertext{+}$-predictable and $v$ is real-valued by \Cref{ass::generator2}, since $|\widehat\lambda \Delta C| = |\widehat\lambda| \Delta C \leq \sqrt{\mathrm{r}}\Delta C \leq \Phi < 1 $ and thus
		\begin{equation*}
			|\gamma| \leq \frac{|\widehat\lambda|}{1-\Phi} \leq \frac{\sqrt{\mathrm{r}}}{1-\Phi}.
		\end{equation*}
		Moreover, $\Delta v = \gamma\Delta C = \widehat\lambda\Delta C/(1-\widehat\lambda\Delta C) > -1$, and therefore $\cE(v)$ is positive, $\G_\smallertext{+}$--predictable and of $\P$--finite variation. An application of the integration by parts formula yields
		\begin{align*}
			\d \big(\cE(v)\delta Y\big)_s &= \cE(v)_{s\smallertext{-}}\d (\delta Y)_s + \delta Y_{s\smallertext{-}} \d \cE(v)_s + \d [\cE(v),\delta Y]_s \\
			&= -\cE(v)_{s\smallertext{-}} \delta f_s \d C_s + \cE(v)_{s\smallertext{-}} \d(\delta Z \bcdot X^c)_s + \cE(v)_{s\smallertext{-}}\d (\delta U \ast\tilde{\mu}^X)_s + \cE(v)_{s\smallertext{-}}\d (\delta N)_s - \cE(v)_{s\smallertext{-}}\d K_s  \\
			&\quad + \cE(v)_{s\smallertext{-}}\delta Y_{s\smallertext{-}}\gamma_s\d C_s + \cE(v)_{s\smallertext{-}}\gamma_s\d [C,\delta Y]_s \\
			&= -\cE(v)_{s\smallertext{-}} \delta f_s \d C_s + \cE(v)_{s\smallertext{-}} \d(\delta Z \bcdot X^c)_s + \cE(v)_{s\smallertext{-}}\d (\delta U \ast\tilde{\mu}^X)_s + \cE(v)_{s\smallertext{-}}\d (\delta N)_s - \cE(v)_{s\smallertext{-}}\d K_s  \\
			&\quad + \cE(v)_{s\smallertext{-}}\delta Y_{s\smallertext{-}}\gamma_s\d C_s - \cE(v)_{s\smallertext{-}}\gamma_s\delta f_s \Delta C_s \d C_s + \cE(v)_{s\smallertext{-}}\gamma_s \d [C,\delta Z \bcdot X^c]_s + \cE(v)_{s\smallertext{-}}\gamma_s \d [C,\delta U \ast\tilde{\mu}^X]_s \\
			&\quad + \cE(v)_{s\smallertext{-}}\gamma_s \d [C,\delta N]_s - \cE(v)_{s\smallertext{-}}\gamma_s \d [C,K]_s \\
			&= - \cE(v)_{s\smallertext{-}} \big(\delta f_s(1+\gamma_s\Delta C_s) - \gamma_s\delta Y_{s\smallertext{-}}\big)\d C_s + \cE(v)_{s\smallertext{-}} \d(\delta Z \bcdot X^c)_s + \cE(v)_{s\smallertext{-}}\d (\delta U \ast\tilde{\mu}^X)_s + \cE(v)_{s\smallertext{-}}\d (\delta N)_s \nonumber\\
			&\quad - \cE(v)_{s\smallertext{-}}\d K_s - \cE(v)_{s\smallertext{-}}\gamma_s \d [C,K]_s \\
			&\quad + \cE(v)_{s\smallertext{-}}\gamma_s \d [C,\delta Z \bcdot X^c]_s + \cE(v)_{s\smallertext{-}}\gamma_s \d [C,\delta U \ast\tilde{\mu}^X]_s + \cE(v)_{s\smallertext{-}}\gamma_s \d [C,\delta N]_s, \; \textnormal{$\overline{\P}$--a.s.}
		\end{align*}
		Since $C$ is $\G$-predictable, the integral processes
		\[
			\int_0^\cdot \cE(v)_{s\smallertext{-}}\gamma_s \d [C,\delta Z \bcdot X^c]_s, \; \int_0^\cdot\cE(v)_{s\smallertext{-}}\gamma_s \d [C,\delta U \ast\tilde{\mu}^X]_s, \; \textnormal{and} \int_0^\cdot\cE(v)_{s\smallertext{-}}\gamma_s \d [C,\delta N]_s,
		\]
		are $(\G_\smallertext{+},\bar{\P})$--local martingales (see \cite[Proposition I.4.49.(c)]{jacod2003limit}). Now let $\overline{\Q}$ denote the probability measure defined through the density
		\begin{equation*}
			\frac{\d\overline{\Q}}{\d\overline{\P}} \coloneqq \cE(L)_\infty \coloneqq \cE\big(\eta \bcdot X^c + \rho_1 \ast \tilde{\mu}^X\big)_\infty = \cE\big(\eta \bcdot X^c + \rho_1 \ast \tilde{\mu}^X\big)_T,
		\end{equation*}
		where $\rho_1$ comes from \Cref{ass::intrinsic_2bsde}.$(i)$. Since $\langle L \rangle^{(\bar{\P})}$ is $\overline{\P}$--essentially bounded by  \Cref{ass::intrinsic_2bsde}.$(i)$,$(iii)$, we have $\cE(L) \in \cH^2_T(\G_{\smallertext{+}},\overline{\P})$ (see \cite[Lemma~7.4]{possamai2024reflections}).
		We now express the $(\G_\smallertext{+},\overline{\P})$--local martingales appearing in the decomposition of $\Ec(v)\delta Y$ in terms of their $(\G_\smallertext{+},\overline{\Q})$--semi-martingale decomposition. By Girsanov's theorem (see \cite[Th\'eor\`eme 7.24, Proposition 7.25 and 7.26]{jacod1979calcul}), we obtain
		\begin{align*}
			\cE(v)_{s\smallertext{-}} \d (\delta Z \bcdot X^c)_s 
			&= \cE(v)_{s\smallertext{-}}\d (\delta Z\bcdot X^c - \langle \delta Z \bcdot X^c,L \rangle^{(\bar{\P})})_s + \cE(v)_{s\smallertext{-}}\d\langle \delta Z \bcdot X^c,L \rangle^{(\bar{\P})}_s \\
			&= \cE(v)_{s\smallertext{-}}\d (\delta Z \bcdot X^c - \langle \delta Z \bcdot X^c,L \rangle^{(\bar{\P})})_s + \cE(v)_{s\smallertext{-}}\d\langle \delta Z \bcdot X^c, \eta \bcdot X^c \rangle^{(\bar{\P})}_s \\
			&= \cE(v)_{s\smallertext{-}}\d (\delta Z \bcdot X^c - \langle \delta Z\bcdot X^c,L \rangle^{(\bar{\P})})_s + \cE(v)_{s\smallertext{-}} \eta^\top_s \mathsf{a}_s\delta Z_s \d C_s,\\
			\cE(v)_{s\smallertext{-}}\d(\delta U \ast \tilde\mu^X)_s 
			&= \cE(v)_{s\smallertext{-}}\d\big(\delta U \ast \tilde\mu^X - \langle \delta U \ast\tilde\mu^X,L\rangle^{(\bar{\P})}\big)_s + \cE(v)_{s\smallertext{-}}\d \langle \delta U \ast\tilde\mu^X,L\rangle^{(\bar{\P})}_s \\
			&= \cE(v)_{s\smallertext{-}}\d\big(\delta U \ast \tilde\mu^X - \langle \delta U \ast\tilde\mu^X,L\rangle^{(\bar{\P})}\big)_s + \cE(v)_{s\smallertext{-}}\d \langle \delta U \ast\tilde\mu^X,\rho_1\ast\tilde\mu^X\rangle^{(\bar{\P})}_s,\\
			\cE(v)_{s\smallertext{-}}\d (\delta N)_s 
			&= \cE(v)_{s\smallertext{-}}\d(\delta N-\langle \delta N, L \rangle^{(\bar{\P})})_s + \cE(v)_{s\smallertext{-}}\d\langle \delta N,L\rangle^{(\bar{\P})}_s = \cE(v)_{s\smallertext{-}}\d(\delta N-\langle \delta N, L \rangle^{(\bar{\P})})_s, \; \textnormal{$\overline{\P}$--a.s.},
		\end{align*}
		where the last equality follows from the (strong) $\bar{\P}$-orthogonality of $\delta N$ with respect to $X^c$ and $\mu^X$. The first term in each of the above sums represents the $\overline{\Q}$--local-martingale part. For the remaining $(\G_\smallertext{+},\overline{\P})$--local martingale terms appearing in the decomposition of $\Ec(v)\delta Y$, we find using $\d [C,\delta Z \bcdot X^c]_s = \Delta C_s \d(\delta Z \bcdot X^c)_s$ (see \cite[Proposition I.4.49.b)]{jacod2003limit}) that
		\begin{align*}
			\cE(v)_{s\smallertext{-}}\gamma_s \d[C,\delta Z \bcdot X^c]_s 
			&= \cE(v)_{s\smallertext{-}}\gamma_s \d\big([C,\delta Z \bcdot X^c]-\big\langle [C,\delta Z \bcdot X^c],L \big\rangle^{(\bar{\P})}\big)_s + \cE(v)_{s\smallertext{-}}\gamma_s \d\big\langle [C,\delta Z \bcdot X^c],L \big\rangle^{(\bar{\P})}_s \\
			&= \cE(v)_{s\smallertext{-}}\gamma_s \d\big([C,\delta Z \bcdot X^c]-\big\langle [C,\delta Z \bcdot X^c],L \big\rangle^{(\bar{\P})}\big)_s + \cE(v)_{s\smallertext{-}}\gamma_s \Delta C_s \d\langle \delta Z \bcdot X^c,L \rangle^{(\bar{\P})}_s \\
			&= \cE(v)_{s\smallertext{-}}\gamma_s \d\big([C,\delta Z \bcdot X^c]-\big\langle [C,\delta Z \bcdot X^c],L \big\rangle^{(\bar{\P})}\big)_s + \cE(v)_{s\smallertext{-}} \gamma_s \Delta C_s \eta_s^\top \mathsf{a}_s \delta Z_s \d C_s, \; \textnormal{$\overline{\P}$--a.s.}
		\end{align*}
		Similarly, we obtain
		\begin{align*}
			\cE(v)_{s\smallertext{-}}\gamma_s \d[C,\delta U \ast\tilde\mu^X]_s 
			&= \cE(v)_{s\smallertext{-}}\gamma_s \d\big([C,\delta U \ast\tilde\mu^X]-\big\langle [C,\delta U \ast\tilde\mu^X],L \big\rangle^{(\bar{\P})}\big)_s + \cE(v)_{s\smallertext{-}}\gamma_s \Delta C_s \d\langle \delta U \ast\tilde\mu^X,\rho_1\ast\tilde\mu^X \rangle^{(\bar{\P})}_s,\\
			\cE(v)_{s\smallertext{-}}\gamma_s \d [C,\delta N]_s 
			&= \cE(v)_{s\smallertext{-}}\gamma_s \d \big([C,\delta N] - \big\langle [C,\delta N],L\big\rangle^{(\bar{\P})} \big)_s, \; \textnormal{$\overline{\P}$--a.s.}
		\end{align*}
		Again, the first term in each of the above sums represents the $\overline{\Q}$--local martingale part.
		We now substitute the $\overline{\Q}$--semi-martingale decompositions into the decomposition of $\Ec(v)\delta Y$ and obtain
		\begin{align*}
			\d \big(\cE(v)\delta Y\big)_s 
			&= - \cE(v)_{s\smallertext{-}} \big(\delta f_s(1+\gamma_s\Delta C_s) - \gamma_s\delta Y_{s\smallertext{-}}\big)\d C_s + \cE(v)_{s\smallertext{-}} \d(\delta Z \bcdot X^c)_s + \cE(v)_{s\smallertext{-}}\d (\delta U \ast\tilde\mu^X)_s + \cE(v)_{s\smallertext{-}}\d(\delta N)_s \\
			&\quad + \cE(v)_{s\smallertext{-}}\gamma_s \d [C,\delta Z \bcdot X^c]_s + \cE(v)_{s\smallertext{-}}\gamma_s \d [C,\delta U \ast\tilde\mu^X]_s + \cE(v)_{s\smallertext{-}}\gamma_s \d [C,\delta N]_s - \cE(v)_{s\smallertext{-}}\d K_s  \\
			&\quad- \cE(v)_{s\smallertext{-}}\gamma_s \d [C,K]_s \\
			&= - \cE(v)_{s\smallertext{-}} \big(\delta f_s(1+\gamma_s\Delta C_s) - \gamma_s\delta Y_{s\smallertext{-}}\big)\d C_s + \cE(v)_{s\smallertext{-}}\d (\delta Z \bcdot X^c - \langle \delta Z \bcdot X^c,L \rangle^{(\bar{\P})})_s \\
			&\quad + \cE(v)_{s\smallertext{-}} \eta^\top_s \mathsf{a}_s\delta Z_s \d C_s + \cE(v)_{s\smallertext{-}}\d\big(\delta U \ast \tilde\mu^X - \langle \delta U \ast\tilde\mu^X,L\rangle^{(\bar{\P})}\big)_s + \cE(v)_{s\smallertext{-}}\d \langle \delta U \ast\tilde\mu^X,\rho_1\ast\tilde\mu^X\rangle^{(\bar{\P})}_s \\
			&\quad + \cE(v)_{s\smallertext{-}}\d(\delta N)_s + \cE(v)_{s\smallertext{-}}\gamma_s \d\big([C,\delta Z \bcdot X^c]-\big\langle [C,\delta Z \bcdot X^c],L \big\rangle^{(\bar{\P})}\big)_s + \cE(v)_{s\smallertext{-}}\eta^\top_s \mathsf{a}_s \delta Z_s \gamma_s \Delta C_s \d C_s \\
			&\quad + \cE(v)_{s\smallertext{-}}\gamma_s \d\big([C,\delta U \ast\tilde\mu^X]-\big\langle [C,\delta U \ast\tilde\mu^X],L \big\rangle^{(\bar{\P})}\big)_s + \cE(v)_{s\smallertext{-}}\gamma_s \Delta C_s \d\langle \delta U \ast\tilde\mu^X,\rho_1\ast\tilde\mu^X \rangle^{(\bar{\P})}_s \\
			&\quad + \cE(v)_{s\smallertext{-}}\gamma_s \d [C,\delta N]_s \nonumber - \cE(v)_{s\smallertext{-}}\d K_s - \cE(v)_{s\smallertext{-}}\gamma_s \d [C,K]_s \\
			&= -  \cE(v)_{s\smallertext{-}} \big(\delta f_s(1+\gamma_s\Delta C_s) - \gamma_s\delta Y_{s\smallertext{-}}\big)\d C_s + \cE(v)_{s\smallertext{-}}\d (\delta Z \bcdot X^c - \langle \delta Z \bcdot X^c,L \rangle^{(\bar{\P})})_s \\
			&\quad + \cE(v)_{s\smallertext{-}} \eta^\top_s \mathsf{a}_s\delta Z_s (1+\gamma_s\Delta C_s) \d C_s + \cE(v)_{s\smallertext{-}}\d\big(\delta U \ast \tilde\mu^X - \langle \delta U \ast\tilde\mu^X,L\rangle^{(\bar{\P})}\big)_s \\
			&\quad + \cE(v)_{s\smallertext{-}}(1+\gamma_s \Delta C_s)\d \langle \delta U \ast\tilde\mu^X,\rho_1\ast\tilde\mu^X\rangle^{(\bar{\P})}_s  + \cE(v)_{s\smallertext{-}}\d(\delta N)_s \\
			&\quad + \cE(v)_{s\smallertext{-}}\gamma_s \d\big([C,\delta Z \bcdot X^c]-\big\langle [C,\delta Z \bcdot X^c],L \big\rangle^{(\bar{\P})}\big)_s + \cE(v)_{s\smallertext{-}}\gamma_s \d\big([C,\delta U \ast\tilde\mu^X]-\big\langle [C,\delta U \ast\tilde\mu^X],L \big\rangle^{(\bar{\P})}\big)_s \\
			&\quad + \cE(v)_{s\smallertext{-}}\gamma_s \d [C,\delta N]_s - \cE(v)_{s\smallertext{-}}\d K_s - \cE(v)_{s\smallertext{-}}\gamma_s \d [C,K]_s, \; \textnormal{$\overline{\P}$--a.s.}
		\end{align*}
		Note now that $\cE(v)_{\smallertext{-}}(1+\gamma\Delta C) = \cE(v) > 0$ and $\cE(v)_{s\smallertext{-}}\gamma_s \d [C,K]_s = \cE(v)_{s\smallertext{-}}\gamma_s \Delta C_s \d K_s$ imply
		\begin{equation*}
			\cE(v)_{s\smallertext{-}}\d K_s + \cE(v)_{s\smallertext{-}}\gamma_s \d [C,K]_s = \cE(v)_{s\smallertext{-}}\d K_s + \cE(v)_{s\smallertext{-}}\gamma_s \Delta C_s \d K_s = \cE(v)_{s\smallertext{-}}(1+\gamma_s\Delta C_s)\d K_s = \cE(v)_s\d K_s, \; \textnormal{$\overline{\P}$--a.s.},
		\end{equation*}
		and
		\begin{align*}
			&\cE(v)_{s\smallertext{-}}\delta f_s(1+\gamma_s\Delta C_s) \\
			&\quad \geq \cE(v)_{s\smallertext{-}}\bigg(\lambda_s\delta Y_s + \widehat\lambda_s\delta Y_{s\smallertext{-}} + \eta^\top_s \mathsf{a}_s \delta Z_s + \frac{\d\langle \delta U \ast\tilde\mu^X, \rho_1\ast\tilde\mu^X\rangle^{(\bar{\P})}}{\d C_s}\bigg)(1+\gamma_s\Delta C_s) \\
			&\quad = \cE(v)_{s\smallertext{-}}\bigg(\lambda_s\delta Y_s + \eta^\top_s \mathsf{a}_s \delta Z_s + \frac{\d\langle \delta U \ast\tilde\mu^X, \rho_1\ast\tilde\mu^X\rangle^{(\bar{\P})}}{\d C_s}\bigg)(1+\gamma_s\Delta C_s) + \cE(v)_{s\smallertext{-}}\gamma_s\delta Y_{s\smallertext{-}}, \; \textnormal{$\overline{\P}\otimes\d C_s$--a.e.}
		\end{align*}		
		Therefore
		\begin{align}\label{eq::inequality_exponential_y}
			 \d \big(\cE(v)\delta Y\big)_s 
			& \leq - \cE(v)_{s\smallertext{-}}\lambda_s \delta Y_s (1+\gamma_s\Delta C_s) \d C_s + \d M^{\bar{\Q}}_s - \cE(v)_s\d K_s \nonumber\\
			&= - \cE(v)_s\lambda_s \delta Y_s \d C_s + \d M^{\bar{\Q}}_s - \cE(v)_s\d K_s, \; \textnormal{$\overline{\P}$--a.s.}
		\end{align}
		Here we used $\cE(v) = \cE(v)_{\smallertext{-}}(1+\gamma\Delta C)$ in the last equality, and we denoted by $M^{\bar{\Q}}$ the sum of the $\overline{\Q}$--local martingales appearing in the decomposition of $\cE(v)\delta Y$. Consider now 
		\[
		w \coloneqq \int_0^{\cdot \land T} \lambda_s \d C_s,
		\]
		which is well-defined and real-valued by \Cref{ass::generator2} since $|\lambda_s| \leq \sqrt{r_s}$. Moreover, $|\lambda_s \Delta C_s| = |\lambda_s|\Delta C_s \leq \sqrt{r_s} \Delta C_s \leq \Phi < 1$ implies $\cE(w) > 0$. Another application of the integration by parts formula, together with \cite[Proposition I.4.49.(a)]{jacod2003limit}, yields
		\begin{align*}
			\d (\cE(w)\cE(v)\delta Y)_s 
			&= \cE(w)_{s\smallertext{-}}\d (\cE(v)\delta Y)_s + (\cE(v)\delta Y)_{s\smallertext{-}}\d \cE(w)_s + \d [\cE(w),\cE(v)\delta Y] \\
			&= \cE(w)_{s\smallertext{-}}\d (\cE(v)\delta Y)_s + (\cE(v)\delta Y)_s\d \cE(w)_s \\
			&= \cE(w)_{s\smallertext{-}}\d (\cE(v)\delta Y)_s + \cE(w)_{s\smallertext{-}}\cE(v)_s \lambda_s \delta Y_s\d C_s \\
			&\leq -\cE(w)_{s\smallertext{-}}\cE(v)_s\lambda_s\delta Y_s \d C_s + \cE(w)_{s\smallertext{-}}\d M^{\bar{\Q}}_s + \cE(w)_{s\smallertext{-}}\cE(v)_s \lambda_s \delta Y_s\d C_s - \cE(w)_{s\smallertext{-}}\cE(v)_s\d K_s \\
			&= \cE(w)_{s\smallertext{-}}\d M^{\bar{\Q}}_s - \cE(w)_{s\smallertext{-}}\cE(v)_s\d K_s, \; \textnormal{$\overline{\P}$--a.s.}
		\end{align*}
		Here we used \eqref{eq::inequality_exponential_y} to obtain the inequality.
		We find
		\begin{equation*}
			\int_{t}^{t^\smalltext{\prime}} \d (\cE(w)\cE(v)\delta Y)_s \leq \int_{t}^{t^\smalltext{\prime}}\cE(w)_{s\smallertext{-}}\d M^{\bar{\Q}}_s - \int_{t}^{t^\prime} \cE(w)_{s\smallertext{-}}\cE(v)_s \d K_s, \; 0 \leq t \leq t^\prime < \infty, \; \textnormal{$\overline{\P}$--a.s.},
		\end{equation*}
		which, together with $\delta Y_{t^\prime} \geq 0$, $\overline{\P}$--a.s., implies
		\begin{align}\label{eq::comparison_ineq_martingale}
		\nonumber	\cE(w)_{t}\cE(v)_{t}\delta Y_{t} 
			&\geq \cE(w)_{t^\smalltext{\prime}}\cE(v)_{t^\smalltext{\prime}}\delta Y_{t^\smalltext{\prime}} - \int_{t}^{t^\smalltext{\prime}}\cE(w)_{s\smallertext{-}}\d M^{\bar{\Q}}_s + \int_{t}^{t^\prime} \cE(w)_{s\smallertext{-}}\cE(v)_s \d K_s \\
			&\geq - \int_{t}^{t^\smalltext{\prime}}\cE(w)_{s\smallertext{-}}\d M^{\bar{\Q}}_s + \int_{t}^{t^\smalltext{\prime}} \cE(w)_{s\smallertext{-}}\cE(v)_s \d K_s, \; 0 \leq t \leq t^\prime < \infty, \; \textnormal{$\overline{\P}$--a.s.}
		\end{align}
		Since $0 < \cE(v) \leq \mathrm{e}^{v}$ (see \cite[Lemma 4.1]{cohen2012existence}) we have that $\cE(v)$ is $\overline{\P}$--essentially (and thus $\overline{\Q}$--essentially) bounded by $\mathrm{e}^{\mathfrak{C}/(1-\Phi)}$, where $\mathfrak{C} \in [0,\infty)$ is the constant appearing in \Cref{ass::intrinsic_2bsde}.$(iii)$. Thus,
		\begin{align*}
			\E^{\bar{\Q}}\bigg[\sup_{s \in [0,T]}|\cE(v)_s\cE(w)_s\delta Y_s|\bigg] 
			&\leq \E^{\bar{\P}}\big[|\cE(L)_T|^2\big]^{1/2} \E^{\bar{\P}}\bigg[\sup_{s \in [0,T]}|\cE(v)_s\cE(w)_s\delta Y_s|^2\bigg]^{1/2} \\
			&\leq \E^{\bar{\P}}\big[|\cE(L)_T|^2\big]^{1/2}  \mathrm{e}^{\mathfrak{C}/(1-\Phi)} \E^{\bar{\P}}\bigg[\sup_{s \in [0,T]}|\cE(w)_s\delta Y_s|^2\bigg]^{1/2} \\
			&\leq \E^{\bar{\P}}\big[|\cE(L)_T|^2\big]^{1/2}  \mathrm{e}^{\mathfrak{C}/(1-\Phi)} \E^{\bar{\P}}\bigg[\sup_{s \in [0,T]}|\cE(\hat\beta A)^{1/2}_s\delta Y_s|^2\bigg]^{1/2} < \infty,
		\end{align*}
		where the last inequality follows from the fact that $M^\Phi_1(\hat\beta) < M^\Phi_1(\beta^\star) = 1$ implies $\hat\beta > \beta^\star > 2+\Phi$, and therefore
		\begin{equation}\label{eq::inequality_w_phi}
			|\cE(w)|^2 = \cE(2w + [w]) = \cE(2w + \Delta w \bcdot w) = \cE((2 + \Delta w)\bcdot w) \leq \cE((2+\Phi)A) \leq \cE(\beta^\star A).
		\end{equation}
		Let $(\tau_n)_{n \in \N}$ be an $(\G_{\smallertext{+}},\overline{\Q})$--localising sequence such that for each $n \in \N$, the stopped process $\int_0^{\cdot \land \tau_\smalltext{n}} \cE(w)_{s\smallertext{-}} \d M^{\bar{\Q}}_s,$ is a uniformly integrable martingale relative to $(\G_\smallertext{+},\overline{\Q})$. Fix $0 \leq t \leq t^\prime < \infty$. After taking conditional expectations in \eqref{eq::comparison_ineq_martingale}, we obtain
		\begin{equation}\label{eq::comp_limit_n}
			\cE(w)_{t \land \tau_\smalltext{n}}\cE(v)_{t \land \tau_\smalltext{n}}\delta Y_{t \land \tau_\smalltext{n}} 
			\geq 
			\E^{\bar{\Q}}\bigg[\int_{t \land \tau_\smalltext{n}}^{t^\smalltext{\prime} \land \tau_\smalltext{n}} \cE(w)_{s\smallertext{-}}\cE(v)_s \d K_s\bigg|\cG_{t\smallertext{+}}\bigg], \; \textnormal{$\bar{\P}$--a.s.}
		\end{equation}
		An application of Fatou's lemma for conditional expectations (see \cite[Theorem 2.(d), page 259]{shiryaev2016probability}) yields
		\begin{align*}
			\cE(w)_{t}\cE(v)_{t}\delta Y_{t} 
			\geq \liminf_{n \rightarrow \infty}\E^{\bar{\Q}}\bigg[\int_{t \land \tau_\smalltext{n}}^{t^\smalltext{\prime} \land \tau_\smalltext{n}} \cE(w)_{s\smallertext{-}}\cE(v)_s \d K_s\bigg|\cG_{t\smallertext{+}}\bigg] 
			\geq \E^{\bar{\Q}}\bigg[\int_{t}^{t^\smalltext{\prime}} \cE(w)_{s\smallertext{-}}\cE(v)_s \d K_s\bigg|\cG_{t\smallertext{+}}\bigg], \; \textnormal{$\overline{\P}$--a.s.}
		\end{align*}
		We now use Bayes's formula for conditional expectation, the property from \cite[Equation VI.58.1]{dellacherie1982probabilities}, along with the fact that the $(\G_\smallertext{+},\overline{\P})$--optional projection of the time-constant process $\cE(L)_T$ is the martingale $\cE(L)$ itself, and obtain
		\begin{align*}
			\nonumber	\cE(w)_{t}\cE(v)_{t}\delta Y_{t} 
			&\geq \E^{\bar{\Q}}\bigg[\int_{t}^{t^\smalltext{\prime}} \cE(w)_{s\smallertext{-}}\cE(v)_s \d K_s\bigg|\cG_{t\smallertext{+}}\bigg] \nonumber\\
			&= \frac{1}{\cE(L)_t}\E^{\bar{\P}}\bigg[\int_{t}^{t^\smalltext{\prime}} \cE(L)_{T}\cE(w)_{s\smallertext{-}}\cE(v)_s \d K_s\bigg|\cG_{t\smallertext{+}}\bigg] \nonumber\\
			&= \frac{1}{\cE(L)_t}\E^{\bar{\P}}\bigg[\int_{t}^{t^\smalltext{\prime}} \cE(L)_{s}\cE(w)_{s\smallertext{-}}\cE(v)_s \d K_s\bigg|\cG_{t\smallertext{+}}\bigg], \; \text{$\overline{\P}$--a.s.}
		\end{align*}
		We then obtain
		\begin{equation*}
			\delta Y_{t} 
			\geq \E^{\bar{\P}}\bigg[\int_{t}^{t^\smalltext{\prime}} \frac{\cE(L)_s \cE(w)_{s\smallertext{-}}\cE(v)_s}{\cE(L)_t\cE(w)_t \cE(v)_t} \d K_s\bigg|\cG_{t\smallertext{+}}\bigg]
			\geq \E^{\bar{\P}}\bigg[\inf_{s \in (t, t^\smalltext{\prime}]} \bigg\{\frac{\cE(L)_s \cE(v)_s}{\cE(L)_t \cE(v)_t}\bigg\} \bigg(\int_{t}^{t^\smalltext{\prime}} \frac{\cE(w)_{s\smallertext{-}}}{\cE(w)_t} \d K_s\bigg)\bigg|\cG_{t\smallertext{+}}\bigg], \; \text{$\overline{\P}$--a.s.}
		\end{equation*}
		Since, for $0 \leq t < s < \infty$,
		\[
		\frac{\cE(v)_s}{\cE(v)_t} = \cE\bigg(\int_t^\cdot \d v_u\bigg)_s \leq \cE\bigg(\int_t^\cdot |\d v_u|\bigg)_s \leq \cE\bigg(\int_t^\cdot |\gamma_u|\d C_u\bigg)_s \leq \cE\bigg(\frac{1}{1-\Phi}\int_t^\cdot \sqrt{\mathrm{r}_u}\d C_u\bigg)_s  \leq \exp\bigg(\frac{1}{1-\Phi}\int_0^s \sqrt{\mathrm{r}_u}\d C_u\bigg)
		\]
		we obtain from \cite[Lemma 4.4]{cohen2012existence} that
		\begin{align*}
			\frac{\cE(v)_t}{\cE(v)_s}  
			= \bigg(\frac{\cE(v)_s}{\cE(v)_t} \bigg)^{-1}
			&= \bigg(\cE\bigg(\int_t^\cdot \d v_r\bigg)_s \bigg)^{-1}
			= \cE\bigg(-\int_t^\cdot \d v_r + \int_t^\cdot \frac{\Delta v_r}{1+\Delta v_r}\d v_r\bigg) \\
			&\leq \cE\bigg(\int_t^\cdot |\d v_r| + \int_t^\cdot \Big|\frac{\Delta v_r}{1+\Delta v_r}\Big||\d v_r|\bigg) \\
			&\leq \cE\bigg(\int_t^\cdot |\d v_r| + \Phi\int_t^\cdot|\d v_r|\bigg) \\
			&= \cE\bigg( (1+\Phi)\int_t^\cdot|\d v_r|\bigg)
			\leq \exp\bigg(\frac{1+\Phi}{(1-\Phi)}\int_0^s \sqrt{\mathrm{r}_u}\d C_u\bigg).
		\end{align*}
		Thus, by \Cref{ass::intrinsic_2bsde}, there exists $\mathfrak{C}^\prime \in (0,\infty)$ such that
		\begin{equation}\label{eq::boundedness_ev}
			\frac{1}{\mathfrak{C}^\prime} \leq \frac{\cE(v)_s}{\cE(v)_t} \leq \mathfrak{C}^\prime \; \textnormal{and} \; \frac{1}{\mathfrak{C}^\prime} \leq \frac{\cE(v)_t}{\cE(v)_s} \leq \mathfrak{C}^\prime, \; 0 \leq t < s < \infty, \; \textnormal{$\fP_0$--q.s.}
		\end{equation}
		This implies, by \cite[Proposition III.3.5]{jacod2003limit} and the fact that $\overline{\Q}$ and $\overline{\P}$ are equivalent, that
		\begin{gather}
			\mathfrak{C}^\prime \inf_{s \in (t, t^\smalltext{\prime}]} \bigg\{\frac{\cE(L)_s}{\cE(L)_t}\bigg\} 
			\geq \inf_{s \in (t, t^\smalltext{\prime}]} \bigg\{\frac{\cE(L)_s \cE(v)_s}{\cE(L)_t \cE(v)_t}\bigg\} 
			\geq \frac{1}{\mathfrak{C}^\prime} \inf_{s \in (t, t^\smalltext{\prime}]} \bigg\{\frac{\cE(L)_s}{\cE(L)_t}\bigg\} > 0, \; \text{$\overline{\P}$--a.s.},\label{eq_inf_exponential_ratios}\\
			\infty > \mathfrak{C}^\prime \sup_{s \in (t, t^\smalltext{\prime}]} \bigg\{\frac{\cE(L)_t}{\cE(L)_s}\bigg\} \geq \sup_{s \in (t, t^\smalltext{\prime}]} \bigg\{\frac{\cE(L)_t \cE(v)_t}{\cE(L)_s \cE(v)_s}\bigg\} \geq \frac{1}{\mathfrak{C}^\prime} \sup_{s \in (t, t^\smalltext{\prime}]} \bigg\{\frac{\cE(L)_t}{\cE(L)_s}\bigg\} > 0, \; \text{$\overline{\P}$--a.s.} \label{eq_sup_exponential_ratios}
		\end{gather}
		Note also that by \eqref{eq::inequality_w_phi},
		\begin{equation}\label{lambda_exponential_bound}
		\cE\bigg(\int_t^\cdot \lambda_r \d C_r\bigg)_s = \frac{\cE(w)_s}{\cE(w)_t} \leq \cE(\tilde{\beta} A)^{1/2}_s, \; 0 \leq t < s < \infty,
		\end{equation}
		for some fixed $\tilde{\beta}$ satisfying $2+\Phi < \beta^\star < \tilde{\beta} < \beta < \hat\beta$. Furthermore, \Cref{eq::inverse_stochastic_exponential} and \Cref{rem::inverse_exponential}, together with \Cref{ass::intrinsic_2bsde}, imply (up to changing the constant $\mathfrak{C}^\prime$ if necessary) that
		\[
		\E^{\bar{\P}}\bigg[\sup_{s \in [t, \infty)} \big\{\cE(L)_t (\cE(L)_s)^{-1}\big\}  \bigg|\cG_{t\smallertext{+}}\bigg] \leq 4 \mathrm{e}^{\mathfrak{C}^\prime}, \; \textnormal{$\overline{\P}$--a.s.}, \; t \in [0,\infty).
		\]
		This yields
		\begin{align}
			&\E^{\bar{\P}}\bigg[\int_{t}^{t^{\smalltext{\prime}}} \cE\bigg(\int_t^\cdot \lambda_r\d C_r\bigg)_{s\smallertext{-}} \d K_s\bigg|\cG_{t\smallertext{+}}\bigg] 
			= \E^{\bar{\P}}\bigg[\int_{t}^{t^{\smalltext{\prime}}} \frac{\cE(w)_{s\smallertext{-}}}{\cE(w)_t} \d K_s\bigg|\cG_{t\smallertext{+}}\bigg] \nonumber\\
			&= \E^{\bar{\P}}\bigg[\bigg(\inf_{s \in (t, t^{\smalltext{\prime}}]} \bigg\{\frac{\cE(L)_s \cE(v)_s}{\cE(L)_t \cE(v)_t}\bigg\} \bigg)^{1/3} \bigg( \int_{t}^{t^{\smalltext{\prime}}} \frac{\cE(w)_{s\smallertext{-}}}{\cE(w)_t} \d K_s\bigg)^{1/3}\bigg(\inf_{s \in (t, t^{\smalltext{\prime}}]} \bigg\{\frac{\cE(L)_s \cE(v)_s}{\cE(L)_t \cE(v)_t}\bigg\} \bigg)^{-1/3} \bigg( \int_{t}^{t^{\smalltext{\prime}}} \frac{\cE(w)_{s\smallertext{-}}}{\cE(w)_t} \d K_s\bigg)^{2/3}\bigg|\cG_{t\smallertext{+}}\bigg] \nonumber\\
			&\leq \Bigg(\E^{\bar{\P}}\bigg[\bigg(\inf_{s \in (t, t^{\smalltext{\prime}}]} \bigg\{\frac{\cE(L)_s \cE(v)_s}{\cE(L)_t \cE(v)_t}\bigg\} \bigg) \bigg( \int_{t}^{t^{\smalltext{\prime}}} \frac{\cE(w)_{s\smallertext{-}}}{\cE(w)_t} \d K_s\bigg)\bigg|\cG_{t\smallertext{+}}\bigg] \Bigg)^{1/3} \Bigg( \E^{\bar{\P}}\bigg[\bigg(\inf_{s \in (t, t^{\smalltext{\prime}}]} \bigg\{\frac{\cE(L)_s \cE(v)_s}{\cE(L)_t \cE(v)_t}\bigg\} \bigg)^{-1} \bigg|\cG_{t\smallertext{+}}\bigg]\Bigg)^{1/3} \nonumber\\
			&\quad \times \Bigg(\E^{\bar{\P}}\bigg[ \bigg( \int_{t}^{t^{\smalltext{\prime}}} \frac{\cE(w)_{s\smallertext{-}}}{\cE(w)_t} \d K_s\bigg)^{2}\bigg|\cG_{t\smallertext{+}}\bigg]\Bigg)^{1/3} \nonumber\\
			&\leq \big(\delta Y_{t}\big)^{1/3} \Bigg( \E^{\bar{\P}}\bigg[\bigg(\inf_{s \in (t, t^{\smalltext{\prime}}]} \bigg\{\frac{\cE(L)_s \cE(v)_s}{\cE(L)_t \cE(v)_t}\bigg\} \bigg)^{-1} \bigg|\cG_{t\smallertext{+}}\bigg]\Bigg)^{1/3} \Bigg(\E^{\bar{\P}}\bigg[ \bigg( \int_{t}^{t^{\smalltext{\prime}}} \frac{\cE(w)_{s\smallertext{-}}}{\cE(w)_t} \d K_s\bigg)^{2}\bigg|\cG_{t\smallertext{+}}\bigg]\Bigg)^{1/3} \nonumber\\
			&= \big(\delta Y_{t}\big)^{1/3} \Bigg( \E^{\bar{\P}}\bigg[\sup_{s \in (t, t^{\smalltext{\prime}}]} \bigg\{\frac{\cE(L)_t \cE(v)_t}{\cE(L)_s \cE(v)_s}\bigg\}  \bigg|\cG_{t\smallertext{+}}\bigg]\Bigg)^{1/3} \Bigg(\E^{\bar{\P}}\bigg[ \bigg( \int_{t}^{t^{\smalltext{\prime}}} \frac{\cE(w)_{s\smallertext{-}}}{\cE(w)_t} \d K_s\bigg)^{2}\bigg|\cG_{t\smallertext{+}}\bigg]\Bigg)^{1/3} \nonumber\\
			&\leq \big(\delta Y_{t}\big)^{1/3} \mathfrak{C}^\prime \Bigg( \E^{\bar{\P}}\bigg[\sup_{s \in (t, t^{\smalltext{\prime}}]} \big\{\cE(L)_t (\cE(L)_s)^{-1}\big\}  \bigg|\cG_{t\smallertext{+}}\bigg]\Bigg)^{1/3} \Bigg(\E^{\bar{\P}}\bigg[ \bigg( \int_{t}^{t^{\smalltext{\prime}}} \frac{\cE(w)_{s\smallertext{-}}}{\cE(w)_t} \d K_s\bigg)^{2}\bigg|\cG_{t\smallertext{+}}\bigg]\Bigg)^{1/3} \nonumber\\
			&\leq \big(\delta Y_{t}\big)^{1/3} \mathfrak{C}^\prime 4^{1/3} \mathrm{e}^{\mathfrak{C}^\prime/3} \Bigg(\E^{\bar{\P}}\bigg[ \bigg( \int_{t}^{t^{\smalltext{\prime}}} \frac{\cE(w)_{s\smallertext{-}}}{\cE(w)_t} \d K_s\bigg)^{2}\bigg|\cG_{t\smallertext{+}}\bigg]\Bigg)^{1/3} \nonumber\\
			&\leq \big(\delta Y_{t}\big)^{1/3} \mathfrak{C}^\prime 4^{1/3} \mathrm{e}^{\mathfrak{C}^\prime/3} \Bigg(\E^{\bar{\P}}\bigg[ \bigg( \int_{t}^{t^{\smalltext{\prime}}} \cE(\tilde{\beta} A)^{1/2}_s \d K_s\bigg)^{2}\bigg|\cG_{t\smallertext{+}}\bigg]\Bigg)^{1/3} \nonumber\\
			&\leq \big(\delta Y_{t}\big)^{1/3} \mathfrak{C}^\prime 4^{1/3} \mathrm{e}^{\mathfrak{C}^\prime/3} \Bigg(\E^{\bar{\P}}\bigg[ \bigg( \int_{t}^{T} \cE(\tilde{\beta} A)^{1/2}_s \d K_s\bigg)^{2}\bigg|\cG_{t\smallertext{+}}\bigg]\Bigg)^{1/3}, \; \textnormal{$\overline{\P}$--a.s.} \nonumber
		\end{align}
		Since $\overline{\P}$ was arbitrary in $\fP_0(\cG_{t+},\P)$, we obtain
		\begin{align}\label{eq::minimality_condition}
			&\E^{\bar{\P}}\bigg[\int_{t}^{t^{\smalltext{\prime}}} \cE\bigg(\int_t^\cdot \lambda_r\d C_r\bigg)_{s\smallertext{-}} \d K_s\bigg|\cG_{t\smallertext{+}}\bigg]\nonumber \\
			&\quad \leq \big(\delta Y_{t}\big)^{1/3} \mathfrak{C}^\prime 4^{1/3} \mathrm{e}^{\mathfrak{C}^\prime/3} \Bigg( \underset{\P^\smalltext{\prime} \in \fP_\smalltext{0}(\cG_{\smalltext{t}\tinytext{+}},\P)}{{\esssup}^{\P}} \E^{\P^\smalltext{\prime}}\bigg[ \bigg( \int_{t}^{T} \cE(\tilde{\beta} A)^{1/2}_s \d \widehat{K}^{\P^\smalltext{\prime}}_s\bigg)^{2}\bigg|\cG_{t\smallertext{+}}\bigg]\Bigg)^{1/3}, \; \textnormal{$\P$--a.s.}
		\end{align}
		The term after the inequality in \eqref{eq::minimality_condition} is $\P$--a.s. finite by \Cref{lem::bound_k_2bsde}. The result then follows by letting $t^\prime$ tend to infinity using conditional monotone convergence, taking the $\P$--essential infimum over $\overline{\P} \in \fP_0(\cG_{t+},\P)$, and applying \ref{2BSDE::aggregation}. This concludes the proof.
	\end{proof}
	
	\begin{proof}[Proof of \Cref{thm::2BSDE_wellposed_2}]
		Assertion $(i)$ follows from \Cref{thm::2BSDE_wellposed} and \Cref{prop::minimality}. We turn to $(ii)$. Fix $\P \in \fP_0$, $t \in [0,\infty)$, and then $\overline{\P} \in \fP_0(\cG_{t+},\P)$. We use notation analogous to that in the proof of \Cref{prop::minimality} but write $\overline{\Q}^1$ instead of $\overline{\Q}$ to emphasize the role of $\rho^1$ in the change-of-measure argument. By following the arguments leading to \eqref{eq::comp_limit_n}, but without using $\delta Y_{t^\smalltext{\prime}} \geq 0$ there, we obtain 
		\begin{equation}\label{eq::comp_limit_n_intrinsic}
			\cE(w)_{t \land \tau_\smalltext{n}}\cE(v)_{t \land \tau_\smalltext{n}}\delta Y_{t \land \tau_\smalltext{n}} 
			\geq 
			\E^{\bar{\Q}^\smalltext{1}}\bigg[\cE(w)_{t^\smalltext{\prime} \land \tau_\smalltext{n}}\cE(v)_{t^\smalltext{\prime} \land \tau_\smalltext{n}}\delta Y_{t^\smalltext{\prime} \land \tau_\smalltext{n}} + \int_{t \land \tau_\smalltext{n}}^{t^\smalltext{\prime} \land \tau_\smalltext{n}} \cE(w)_{s\smallertext{-}}\cE(v)_s \d K_s\bigg|\cG_{t\smallertext{+}}\bigg], \; \textnormal{$\overline{\P}$--a.s.}
		\end{equation}
		Note that because $T < \infty$, $\fP_0$--q.s., we have
		\[
			\lim_{t^\prime \uparrow\uparrow \infty}\lim_{n\rightarrow\infty} \big\{ \cE(w)_{t^\smalltext{\prime} \land \tau_\smalltext{n}}\cE(v)_{t^\smalltext{\prime} \land \tau_\smalltext{n}}\delta Y_{t^\smalltext{\prime} \land \tau_\smalltext{n}} \big\} = \cE(w)_{T}\cE(v)_{T} (Y_T - \cY^{\bar{\P}}_T) = 0, \; \textnormal{$\overline{\P}$--a.s.}
		\]
		since, $\overline{\P}$--a.s., $Y_T = Y^{\bar{\P}}_T = \xi$. Since
		\[
			\E^{\bar{\Q}^\smalltext{1}}\bigg[\sup_{r \in [0,T]}\big|\cE(w)_{r}\cE(v)_{r}\delta Y_{r}\big|\bigg] < \infty,
		\]
		we apply Fatou's lemma (see \cite[Theorem~2.(d), p.~259]{shiryaev2016probability}) twice in \eqref{eq::comp_limit_n_intrinsic} and obtain
		\[
			\cE(w)_{t \land T}\cE(v)_{t \land T}\delta Y_{t \land T}
			\geq 
			\E^{\bar{\Q}^\smalltext{1}}\bigg[\int_{t}^{T} \cE(w)_{s\smallertext{-}}\cE(v)_s \d K_s\bigg|\cG_{t\smallertext{+}}\bigg] \geq 0, \; \textnormal{$\overline{\P}$--a.s. and $\P$--a.s.}
		\]
		Therefore $Y_t \geq \cY^{\bar{\P}}_t$, $\P$--a.s., since $\cE(w)\cE(v) > 0$, which yields
		\[
		Y_t \geq \underset{\P^\smalltext{\prime} \in \fP_\smalltext{0}(\cG_{\smallertext{t}\smalltext{+}},\P)}{{\esssup}^\P} \cY^{\P^\smalltext{\prime}}_t, \; \textnormal{$\P$--a.s.}
		\]
		For the converse inequality, we repeat the arguments leading to \eqref{eq::comp_limit_n} but with $\rho_2$ instead of $\rho_1$. We denote the resulting probability measure then by $\d \overline{\Q}^2 = \cE(L^2)_T \d\overline{\P}$. The inequality in \eqref{eq::comparison_ineq_martingale} is then reversed and leads to
		\begin{equation}\label{comp_limit_n_converse}
			\cE(w)_{t \land \tau_\smalltext{n}}\cE(v)_{t \land \tau_\smalltext{n}}\delta Y_{t \land \tau_\smalltext{n}} 
			\leq 
			\E^{\bar{\Q}^\smalltext{2}}\bigg[\cE(w)_{t^\smalltext{\prime} \land \tau_\smalltext{n}}\cE(v)_{t^\smalltext{\prime} \land \tau_\smalltext{n}}\delta Y_{t^\smalltext{\prime} \land \tau_\smalltext{n}} + \int_{t \land \tau_\smalltext{n}}^{t^\smalltext{\prime} \land \tau_\smalltext{n}} \cE(w)_{s\smallertext{-}}\cE(v)_s \d K_s\bigg|\cG_{t\smallertext{+}}\bigg], \; \textnormal{$\overline{\P}$--a.s.}
		\end{equation}
		We have
		\[
			0 \leq \int_{t \land \tau_\smalltext{n}}^{t^\smalltext{\prime} \land \tau_\smalltext{n}} \cE(w)_{s\smallertext{-}}\cE(v)_s \d K_s \leq \int_t^T \cE(w)_{s\smallertext{-}}\cE(v)_s \d K_s
		\]
		and by \eqref{eq::boundedness_ev} and \eqref{lambda_exponential_bound} there exists $\mathfrak{C}^\prime \in (0,\infty)$ such that
		\begin{align*}
			\E^{\bar{\Q}^\smalltext{2}}\Bigg[\int_t^T\cE(w)_{s\smallertext{-}}\cE(v)_s \d K_s\bigg] 
			&\leq \E^{\bar{\P}}\bigg[\bigg(\frac{\d\overline{\Q}^2}{\d\overline{\P}}\bigg)^2\Bigg]^{1/2} \E^{\bar{\P}}\Bigg[\bigg( \int_t^T \cE(w)_{s\smallertext{-}}\cE(v)_s \d K_s\bigg)^2\Bigg]^{1/2} \\
			&\leq \mathfrak{C}^\prime \E^{\bar{\P}}\bigg[\bigg(\frac{\d\overline{\Q}^2}{\d\overline{\P}}\bigg)^2\Bigg]^{1/2} \E^{\bar{\P}}\Bigg[\bigg( \int_t^T \cE(\tilde{\beta} A)^{1/2}_s \d K_s\bigg)^2\Bigg]^{1/2}  < \infty.
		\end{align*}
		The boundedness of the last terms follows from the density being $\overline{\P}$--square-integrable, and for the term involving $K$ this follows from \Cref{lem::bound_k_2bsde} for $\beta^\star < \tilde{\beta} < \beta$. Similarly,
		\[
			\E^{\bar{\Q}^\smalltext{2}}\bigg[\sup_{r \in [0,T]}\big|\cE(w)_{r}\cE(v)_{r}\delta Y_{r}\big|\bigg] \leq \mathfrak{C}^\prime \E^{\bar{\P}}\bigg[\bigg(\frac{\d\overline{\Q}^2}{\d\overline{\P}}\bigg)^2\Bigg]^{1/2} \E^{\bar{\P}} \bigg[\sup_{r \in [0,T]}\big|\cE(\tilde{\beta} A )^{1/2}_r\delta Y_{r}\big|^2\bigg]^{1/2} < \infty.
		\]
		By applying conditional dominated convergence in \eqref{comp_limit_n_converse}, letting first $n\rightarrow\infty$ and then $t^\prime \uparrow\uparrow \infty$, we obtain
		\begin{align*}
			\cE(w)_{t \land T}\cE(v)_{t \land T}\delta Y_{t \land T} 
			&\leq 
			\E^{\bar{\Q}^\smalltext{2}}\bigg[\cE(w)_{T}\cE(v)_{T}\delta Y_{T} + \int_{t}^{T} \cE(w)_{s\smallertext{-}}\cE(v)_s \d K_s\bigg|\cG_{t\smallertext{+}}\bigg], \\
			&= \E^{\bar{\Q}^\smalltext{2}}\bigg[\int_{t}^{T} \cE(w)_{s\smallertext{-}}\cE(v)_s \d K_s\bigg|\cG_{t\smallertext{+}}\bigg], \; \textnormal{$\overline{\P}$--a.s.}
		\end{align*}
		Rearranging the terms, using Bayes' formula and \eqref{eq::boundedness_ev} again yields
		\begin{align*}
			Y_t - \cY^{\bar{\P}}_t
			\leq \frac{1}{\cE(L^2)_t}\E^{\bar{\P}}\bigg[\cE(L^2)_T\int_{t}^{T} \frac{\cE(w)_{s\smallertext{-}}\cE(v)_s}{\cE(w)_t \cE(v)_t} \d K^{\bar{\P}}_s\bigg|\cG_{t\smallertext{+}}\bigg]
			\leq \mathfrak{C}^\prime \E^{\bar{\P}} \bigg[\frac{\cE(L^2)_T}{\cE(L^2)_t} \int_{t}^{T} \frac{\cE(w)_{s\smallertext{-}}}{\cE(w)_t} \d K^{\bar{\P}}_s \bigg|\cG_{t\smallertext{+}}\bigg], \; \text{$\overline{\P}$--a.s.}
		\end{align*}
		Then, \eqref{eq_sup_exponential_ratios}, H\"{o}lder's inequality applied first with conjugate exponents $(3,3/2)$ and then with $(2,2)$, \Cref{eq::inverse_stochastic_exponential}, \Cref{rem::inverse_exponential}, and \eqref{eq::inequality_w_phi} for $\beta^\star < \tilde{\beta} < \beta < \hat{\beta}$ yield
		\begin{align*}
			Y_t - \cY^{\bar{\P}}_t
			&\leq \mathfrak{C}^\prime \,\E^{\bar{\P}}\bigg[ \bigg(\frac{\cE(L^2)_T}{\cE(L^2)_t}\bigg)^3 \bigg|\cG_{t\smallertext{+}}\bigg]^{1/3} \, \E^{\bar{\P}}\bigg[\bigg(\int_{t}^{T} \frac{\cE(w)_{s\smallertext{-}}}{\cE(w)_t} \d K^{\bar{\P}}_s\bigg)^{3/2}\bigg|\cG_{t\smallertext{+}}\bigg]^{2/3} \\
			&\leq \mathfrak{C}^\prime \, \E^{\bar{\P}}\bigg[\bigg(\frac{\cE(L^2)_T}{\cE(L^2)_t}\bigg)^3 \bigg|\cG_{t\smallertext{+}}\bigg]^{1/3} \, \E^{\bar{\P}}\bigg[\bigg(\int_{t}^{T} \frac{\cE(w)_{s\smallertext{-}}}{\cE(w)_t} \d K^{\bar{\P}}_s\bigg)^{2}\bigg|\cG_{t\smallertext{+}}\bigg]^{1/3} \, \E^{\bar{\P}}\bigg[\int_{t}^{T} \frac{\cE(w)_{s\smallertext{-}}}{\cE(w)_t} \d K^{\bar{\P}}_s\bigg|\cG_{t\smallertext{+}}\bigg]^{1/3} \\
			&\leq \mathfrak{C}^\prime  \, \E^{\bar{\P}}\bigg[\bigg(\frac{\cE(L^2)_T}{\cE(L^2)_t}\bigg)^3 \bigg|\cG_{t\smallertext{+}}\bigg]^{1/3} \, \E^{\bar{\P}}\bigg[\bigg(\int_{t}^{T} \cE(\tilde{\beta} A)^{1/2}_s \d K^{\bar{\P}}_s\bigg)^{2}\bigg|\cG_{t\smallertext{+}}\bigg]^{1/3} \, \E^{\bar{\P}}\bigg[\int_{t}^{T} \frac{\cE(w)_{s\smallertext{-}}}{\cE(w)_t} \d K^{\bar{\P}}_s\bigg|\cG_{t\smallertext{+}}\bigg]^{1/3}, \; \textnormal{$\overline{\P}$--a.s.},
		\end{align*}
		which also outside a $\P$--null set. 
		Since $-1+\delta <\Delta L^2 \leq \Theta$ outside some $\overline{\P}$--null set by \Cref{ass::intrinsic_2bsde}.$(i)$, we obtain with \Cref{ass::intrinsic_2bsde}.$(iii)$ together with the arguments in the proof of \cite[Th\'eor\`eme~II.3.c)]{lepingle1978sur} or with \cite[Theorem~6.36, and Remark~5 on p.~207]{jacod1979calcul} that there exists a constant $\mathfrak{C}^{\prime\prime} \in (0,\infty)$ such that
		\[
			\underset{\bar{\P} \in \fP_\smalltext{0}(\cG_{\smallertext{t}\smalltext{+}},\P)}{{\esssup}^\P}\E^{\bar{\P}}\bigg[ \bigg\{\frac{\cE(L^2)_{T}}{\cE(L^2)_t}\bigg\}^3 \bigg|\cG_{t\smallertext{+}}\bigg] \leq \mathfrak{C}^{\prime\prime}, \; \textnormal{$\P$--a.s.}
		\]
		By \Cref{lem::bound_k_2bsde}, we have
		\[
			\mathfrak{C}(\P) \coloneqq \underset{\bar{\P} \in \fP_\smalltext{0}(\cG_{\smalltext{t}\tinytext{+}},\P)}{{\esssup}^{\P}} \E^{\bar{\P}}\bigg[ \bigg( \int_{t}^{T} \cE(\tilde{\beta} A)^{1/2}_s \d K^{\bar{\P}}_s\bigg)^{2}\bigg|\cG_{t\smallertext{+}}\bigg] < \infty, \; \textnormal{$\P$--a.s. and $\overline{\P}$--a.s.}
		\]
		Thus
		\[
			Y_t - \cY^{\bar{\P}}_t \leq \mathfrak{C}^\prime (\mathfrak{C}^{\prime\prime})^{1/3} \mathfrak{C}(\P)^{1/3} \E^{\bar{\P}}\bigg[\int_{t}^{T} \frac{\cE(w)_{s\smallertext{-}}}{\cE(w)_t} \d K^{\bar{\P}}_s\bigg|\cG_{t\smallertext{+}}\bigg]^{1/3}, \; \textnormal{$\P$--a.s.}
		\]
		Since
		\[
		\frac{\cE(w)_{s\smallertext{-}}}{\cE(w)_t} = \cE\bigg(\int_t^\cdot \lambda^{\bar{\P}}_r \d C_r\bigg)_{s\smallertext{-}}, \quad 0 \leq t < s < \infty,
		\]
		it remains to take the $\P$--essential infimum on both sides over $\bar{\P} \in \fP_0(\cG_{t+},\P)$ and then use \ref{2BSDE::minimality_condition_2}. This yields
		\[
		Y_t \leq \underset{\P^\smalltext{\prime} \in \fP_\smalltext{0}(\cG_{\smallertext{t}\smalltext{+}},\P)}{{\esssup}^\P} \cY^{\P^\smalltext{\prime}}_t, \; \textnormal{$\P$--a.s.}
		\]
		Since $\P\in\fP_0$ and $t \in [0,\infty)$ were arbitrary, we see that the tuple $(Y,Z,(U^\P,N^\P,K^\P)_{\P \in \fP_0})$ is a solution to the extrinsic 2BSDE \ref{2BSDE::measurability}--\ref{2BSDE::aggregation}. Thus \Cref{thm::2BSDE_wellposed}.$(ii)$ implies the stated uniqueness. This concludes the proof.
	\end{proof}

\section{Semi-martingale 2BSDE without random measures}\label{sec_2bsde_X}
	
	It is possible to consider a completely different class of 2BSDEs, where the corresponding BSDEs are driven only by an It\^o-integrator with jumps. Specifically, we consider BSDEs driven by the canonical process $X$, under the condition that $X$ is a suitable martingale. To apply the results from \cite{possamai2024reflections}, we work with probability measures under which $X$ is a locally square-integrable martingale. While we omit most of the details, the proofs and techniques used to derive the results in \Cref{sec_preliminaries,sec_wellposed_2bsde} and \cite{possamai2026aggregation} also apply in this setting. Nonetheless, in this class of 2BSDEs, we are still unable to resolve the problem of aggregating the integrands. We discuss this in more detail in \Cref{sec_aggregation_problems}.
	
	\subsection{Martingale laws}\label{sec::martingale_laws}
	
	We restrict our attention to measures in
	\begin{equation*}
		\fH^2_\textnormal{loc} \coloneqq \big\{\P \in \fP_\textnormal{sem} : \textnormal{$X$ is an $(\F,\P)$--locally square-integrable martingale} \big\}.
	\end{equation*}
	Whenever $\P \in \fH^2_\textnormal{loc}$ and $(\mathsf{B}^\P,\mathsf{C}^\P,\nu^\P)$ denote the corresponding characteristics, we have 
	\[
		\langle X \rangle^{(\P)} = \mathsf{C}^{\P} + (x x^\top) \ast \nu^\P, \; \textnormal{$\P$--a.s.},
	\]
	where the equality is understood to hold component-wise (see \cite[Proposition II.2.29.b)]{jacod2003limit}). 

\medskip
	For semi-martingale measures, it is possible to determine whether they belong to $\fH^{2}_\textnormal{loc}$ by examining the corresponding characteristics. The following provides a precise result; it follows directly from \cite[Proposition II.2.29]{jacod2003limit} and \cite[Lemma~A.2]{possamai2026aggregation}.
	
	\begin{proposition}\label{prop::equivalence_locally_square_integrable}
		Let $\P \in \fP_\textnormal{sem}$, let $(\mathsf{B},\mathsf{C},\nu)$ be $(\F,\P)$--semi-martingale characteristics of $X$.
		Then the following conditions are equivalent
		\begin{enumerate}
			\item[$(i)$] $X$ is an $(\F,\P)$--locally square-integrable martingale relative to $(\F,\P);$
			\item[$(ii)$] the characteristics satisfy, outside a $\P$--null set
			\begin{enumerate}
				\item[$(a)$] $|x|^2\ast\nu_t < \infty,\; t\in[0,\infty);$
				\item[$(b)$] $ \mathsf{B} = -(x-h(x))\ast\nu.$
			\end{enumerate}
		\end{enumerate}
	\end{proposition}
	
	The next result concerns measurability.

	\begin{proposition}\label{prop::borel_measurability_martingale_laws}
		The set $\fH^{2}_\textnormal{loc}\subseteq \fP(\Omega)$ is Borel-measurable.
	\end{proposition}

	\begin{proof}
		Let
		\[
			G_1 \coloneqq \big\{ (\P,\omega) \in \fP_\textnormal{sem} \times \Omega:|x|^2\ast\nu^\P_t(\omega) < \infty, \, t \in \D_\smallertext{+} \big\},
		\]
		and
		\[
			G_2 \coloneqq \big\{ (\P,\omega) \in \fP_\textnormal{sem} \times \Omega : \sfB^\P_t(\omega) = -(x-h(x))\ast\nu^\P_t(\omega), \, t \in \D_\smallertext{+} \big\},
		\]
		which are both Borel-measurable subsets of $\fP_\textnormal{sem}\times\Omega$ after choosing versions of the characteristics measurable in $\P$ (see \cite[Theorem 2.5]{neufeld2014measurability}). By letting $G \coloneqq G_1 \cap G_2$, we find with \Cref{prop::equivalence_locally_square_integrable} that
		\[
			\fH^{2}_\textnormal{loc} = \big\{\P\in \fP_\textnormal{sem} : \E^\P[\1_G(\P,\cdot)] = 1\big\}.
		\]
		Here, we use that there exists a constant $C_h \in (0,\infty)$ such that, for all $x \in \R^d$, $|x-h(x)| \leq C_h |x|^2$.
		The Borel-measurability of $\fH^{2}_\textnormal{loc}$ then follows from \cite[Lemma 3.1]{neufeld2014measurability}, which concludes the proof.
	\end{proof}
	
	The next result shows that $\fH^{2}_{\textnormal{loc}}$ satisfies the invariance and stability properties specified in \Cref{ass::probabilities2}. 
	\begin{proposition}\label{prop::stab_invar_martingale_measures}
		Let $\P \in \fH^{2}_\textnormal{loc}$, and let $t \in [0,\infty)$.
		\begin{enumerate}
			\item[$(i)$] $\P^{t,\omega} \in \fH^{2}_\textnormal{loc}$ for \textnormal{$\P$--a.e.} $\omega \in \Omega$. 

			\item[$(ii)$] If $\Q$ is a stochastic kernel on $(\Omega,\cF)$ given $(\Omega,\cF_t)$ with $\Q(\omega) \in \fH^{2}_\textnormal{loc}$ for \textnormal{$\P$--a.e.} $\omega \in \Omega$, then
			\[
				\overline{\P}[A] \coloneqq \iint \big(\1_A\big)^{t,\omega}(\omega^\prime)\Q(\omega;\d\omega^\prime)\P(\d\omega), \; A \in \cF,
			\]
			belongs to $\fH^{2}_\textnormal{loc}$.
		\end{enumerate}
	\end{proposition}
	
	\begin{remark}
	The same statement and proof also hold if the deterministic time \textnormal{`$t$'} is replaced by a finite $\F$--stopping time $\tau$$;$ the concatenation $\omega\otimes_\tau\omega^\prime$, the shifted law $\P^{\tau,\omega}$, and $(\1_A)^{\tau,\omega}$ are defined analogously, see \textnormal{\cite[Section 2]{possamai2026aggregation}}.
	\end{remark}
	
	\begin{proof}
		Assertion $(i)$ follows from \Cref{prop::equivalence_locally_square_integrable} and the formula for the characteristics with respect to conditioning in \cite[Theorem 3.1]{neufeld2016nonlinear}.
		
		\medskip
		We turn to $(ii)$. By \cite[Proposition 4.1]{neufeld2016nonlinear}, we have $\overline{\P} \in \fP_\textnormal{sem}$, and we denote (a measurable version of) the characteristics relative to $\overline{\P}$ by $(\mathsf{B}^{\bar\P},\mathsf{C},\mathsf{K}^{\bar\P}(\d x)\d\mathsf{A}^{\bar\P})$; by \cite[Proposition 6.6.$(i)$]{neufeld2014measurability}, we choose the second characteristic independent of the probability measure.
		 Since $\P = \overline{\P}$ on $\cF_t$, it follows that
		\[
			\big(\mathsf{B}^{\bar\P}_{\cdot\land t},\mathsf{C}_{\cdot\land t},\1_{\llbracket 0,t\rrbracket}\mathsf{K}^{\bar\P}(\d x)\d\mathsf{A}^{\bar\P}\big) = \big(\mathsf{B}^\P_{\cdot\land t},\mathsf{C}_{\cdot\land t},\1_{\llbracket 0,t\rrbracket}\mathsf{K}^\P(\d x)\d\mathsf{A}^\P\big),
		\]
		holds up to a $\overline{\P}$--null set and up to a $\P$--null set. It thus suffices to show that $\sN$ is a $\overline{\P}$--null set, where
		\[
			\sN^c \coloneqq \bigcap_{r \in \D_\tinytext{+}}\bigg\{ \mathsf{B}^{\bar\P}_{t\smallertext{+}r}-\mathsf{B}^{\bar\P}_{t} = - \int_{t}^{t\smallertext{+}r} (x-h(x)) \mathsf{K}^{\bar\P}_s(\d x)\d\mathsf{A}^{\bar\P}_s, \; \int_{t}^{t\smallertext{+}r}\int_{\R^\smalltext{d}} |x|^2 \mathsf{K}^{\bar\P}_s(\d x)\d\mathsf{A}^{\bar\P}_s < \infty \bigg\},
		\]
		The measurability of $\sN$ follows from \cite[Theorem 2.5]{neufeld2014measurability}. Since $\bar{\P}^{t,\omega} = \Q(\omega) \in \fH^{2}_\textnormal{loc}$ for $\overline{\P}$--a.e. $\omega \in \Omega$, and then $\P$--a.e. $\omega \in\Omega$ since $\omega \longmapsto \overline{\P}^{t,\omega}$ is $\cF_t$-measurable, it follows from the formula for the conditional characteristics under $\overline{\P}^{t,\omega}$ (see \cite[Theorem 3.1]{neufeld2016nonlinear}) that
		\[
			\overline{\P}[\sN] = \iint \big(\1_\sN\big)^{t,\omega}(\omega^\prime)\bar{\P}^{t,\omega}(\d\omega^\prime)\overline{\P}(\d\omega) = 0.
		\]
		This yields $\overline{\P} \in \fH^{2}_\textnormal{loc}$, which completes the proof.
	\end{proof}

	\subsection{Data and assumptions}

	\subsubsection{The family of martingale measures}

	The additional assumption we make on the family of probability measures $(\fP(s,\omega))_{(\omega,s)\in\Omega\times[0,\infty)}$ in this section is the following; it enables us to take $X$ as the It\^o-integrator in the (2)BSDEs.
	
	\begin{assumption}\label{ass::probabilities2_dom}
		$\fP(s,\omega) \subseteq \fH^{2}_\textnormal{loc}$ for all $(\omega,s) \in \Omega \times [0,\infty)$, and \textnormal{\Cref{ass::probabilities2}} holds.
	\end{assumption}
	
	\begin{example}
		For $C = (C_t)_{t \in [0,\infty)}$ given by $C_t \coloneqq t$, the family $(\fP(s,\omega))_{(\omega,s)\in\Omega\times[0,\infty)}$ given by $\fP(s,\omega) \equiv \fP_\Theta \cap \fH^{2}_\textnormal{loc}$, where $\fP_\Theta$ was defined in \textnormal{\Cref{ex::differential_characteristics}}, satisfies \textnormal{\Cref{ass::probabilities2_dom}}. This follows from \textnormal{\Cref{prop::borel_measurability_martingale_laws}}, \textnormal{\Cref{prop::stab_invar_martingale_measures}} and \textnormal{\cite[Theorem 2.1.$(i)$]{neufeld2016nonlinear}}.

		\medskip
		As a concrete example, previously inaccessible in the literature, we consider the truncation function $h(x) = x \1_{\{|x|\leq 1\}}$ and the collection of differential characteristic triplets
		\[
		\Theta \coloneqq \{(0,\sigma^2,\lambda\mathbf{\delta}_1(\d x)) : (\sigma^2,\lambda) \in \Sigma\times\Lambda\},
		\]
		for some nonempty Borel-measurable subsets $\Sigma\subseteq [0,\infty)$ and $\Lambda \subseteq [0,\infty)$. For each triplet $(0,\sigma^2,\lambda\mathbf{\delta}_1(\d x))$ in $\Theta$, the canonical process is the sum of a Wiener process with variance $\sigma^2$ $($in the sense of \textnormal{\cite[Definition I.4.9.a)]{jacod2003limit}}$)$ and a compensated Poisson process with intensity $\lambda$; see \textnormal{\cite[Theorem II.2.34]{jacod2003limit}}. Then $\Theta \subseteq \R\times[0,\infty)\times\cL$ is Borel-measurable $($see the arguments in \textnormal{\cite[Example 2.6]{neufeld2016nonlinear}}$)$ and $\fP_\Theta \subseteq \fH^{2}_\textnormal{loc}$. From a modelling perspective, however, a subtlety must be noted. Although under each $\P \in \fP_\Theta$ the process $X$ can be viewed infinitesimally in time as the sum of a Wiener process and a compensated Poisson process, the variance rate and intensity may vary over time across the entire horizon, evolving within the sets $\Sigma$ and $\Lambda$, respectively.
\end{example}

	We write  
\begin{equation}\label{eq::semi_martingale_laws_dom}  
    \widetilde{\Omega}^C_s \coloneqq \big\{(\omega,\P) \in \Omega \times \fH^2_\textnormal{loc} :(\mathsf{B}^{\P},\mathsf{C}^{\P},\nu^{\P}) \ll (C^{s,\omega}_{s\smallertext{+}\smallertext{\cdot}}-C_{s}(\omega)), \; \text{$\P$--a.s.} \big\}, \; s \in [0,\infty).
\end{equation}  

We only sketch the proof of the following result, as it is analogous to the proof of \cite[Lemma~2.4]{possamai2026aggregation}.

	\begin{lemma}
		Let $s \in [0,\infty)$. The subset $\widetilde\Omega^C_s \subseteq \Omega \times \fP_\textnormal{sem}$ defined in \eqref{eq::semi_martingale_laws_dom} is Borel-measurable. Moreover, there exists an $\cF\otimes\cB(\fP_\textnormal{sem})\otimes\cP$-measurable map
		\begin{equation*}
			\Omega \times \fP_\textnormal{sem} \times \Omega\times [0,\infty) \ni (\omega,\P,\tilde\omega,t) \longmapsto c^{s,\omega,\P}_t(\tilde\omega) \in \S^{d}_\smallertext{+},
		\end{equation*}
		such that $c^{s,\omega,\P}$ is $\F$-predictable and
		\begin{equation}\label{eq_measurable_density_jump_quadratic_variation}
			\langle X \rangle^{(\P)} = \int_0^\cdot c^{s,\omega,\P}_t \d (C^{s,\omega}_{s\smallertext{+}\smallertext{\cdot}}-C_{s}(\omega))_t, \; \textnormal{$\P$--a.s.}, \; \textnormal{for all $(\omega,\P) \in \widetilde{\Omega}^C_s$.}
		\end{equation}
	\end{lemma}
	
	\begin{proof}
		Borel-measurability of $\widetilde{\Omega}^C_s$ follows from \Cref{prop::borel_measurability_martingale_laws} and \cite[Lemma~2.4]{possamai2026aggregation}, and we denote by
		\begin{equation*}
			\Omega \times \fP_\textnormal{sem} \times \Omega\times [0,\infty) \ni (\omega,\P,\tilde\omega,t) \longmapsto \big(\mathsf{b}^{s,\omega,\P}_t(\tilde\omega),\mathsf{a}^{s,\omega}_t(\tilde\omega),\mathsf{K}^{s,\omega,\P}_{\tilde\omega,t}\big) \in \R^d \times \S^{d}_\smallertext{+} \times \cL
		\end{equation*}
		the map constructed in \cite[Lemma~2.4]{possamai2026aggregation} for $\tau \equiv s$. Set
		\[
			c^{s,\omega,\P}_t \coloneqq \tilde{c}^{s,\omega,\P}_t \1_{\{\tilde{c}^{\smalltext{s}\smalltext{,}\smalltext{\omega}\smalltext{,}\smalltext{\P}}_t \in \S^\smalltext{d}_\smalltext{+}\}}, \; \textnormal{where} \; \tilde{c}^{s,\omega,\P}_t \coloneqq \mathsf{a}^{s,\omega}_t + \int_{\R^\smalltext{d}}(x x^\top)\mathsf{K}^{s,\omega,\P}_{t}(\d x), \; t \in [0,\infty),
		\]
		Then $(\omega,\P,\tilde{\omega},t) \longmapsto c^{s,\omega,\P}_{t}(\tilde{\omega})$ is $\cF\otimes\cB(\fP_\textnormal{sem})\otimes\cP$-measurable, and for $(\omega,\P)\in\widetilde{\Omega}^C_s$, we obtain \eqref{eq_measurable_density_jump_quadratic_variation}. This concludes the proof.
	\end{proof}

	\subsubsection{The generator}
	
	We abuse notation slightly here, and still denote by $f$ the generator in this section, which is now a function that does not depend on the integrand of the jump measure, and thus is of the form
	\begin{equation}\label{eq::generator_dom}
		f : \Omega \times [0,\infty) \times \R \times \R \times \R^d \times \R^d \times \S^d_\smallertext{+} \times \S^{d}_\smallertext{+} \times \cL \longrightarrow \R.
	\end{equation}
	We then define, as before,
	\[
		f^{s,\omega,\P}_t(\tilde\omega,y,\mathrm{y},z) \coloneqq f\big(\omega\otimes_s\tilde\omega,s+t,y,\mathrm{y},z,\mathsf{b}^{s,\omega,\P}_t(\tilde\omega),\mathsf{a}^{s,\omega}_t(\tilde\omega),c^{s,\omega,\P}_t(\tilde\omega),\mathsf{K}^{s,\omega,\P}_{\tilde\omega,t}\big),
	\]
	for any $s \in [0,\infty)$ and $(\omega,\P) \in \widetilde{\Omega}^C_s$. Recall that $(\mathsf{b}^{s,\omega,\P},\mathsf{a}^{s,\omega},\mathsf{K}^{s,\omega,\P})$ denote the (measurable) $\P$--differential characteristics of $X$ relative to $C^{s,\omega}_{s\smallertext{+}\smallertext{\cdot}}-C_s(\omega)$; see \cite[Lemma~2.4]{possamai2026aggregation}.
	
	\begin{assumption}\label{ass::generator_without_random_measure}
		The triplet $(T, \xi, f)$, where $T$ is an $\F$--stopping time, $\xi$ is an $\cF_T$-measurable real-valued random variable, and $f$ has the form \textnormal{\eqref{eq::generator_dom}}, satisfies the following conditions

\medskip
$(i)$ for every $(y,\mathrm{y},z) \in \R \times \R \times \R^d$, the function
				\[
					\Omega \times [0,\infty) \times \R^d \times \S^d_\smallertext{+} \times \S^d_\smallertext{+} \times \cL \ni (\omega,t,b,a,c,\kappa) \longmapsto f(\omega,t,y,\mathrm{y},z,b,a,c,\kappa) \in \R,
				\]
				is $\textnormal{Prog}(\F)\otimes\cB(\R^d)\otimes\cB(\S^d_\smallertext{+})\otimes\cB(\S^d_\smallertext{+})\otimes\cB(\cL)$-measurable$;$
				
				\medskip
$(ii)$ there exist $\F$-predictable, $[0,\infty)$-valued processes $(r,\mathrm{r},\theta^X) = (r_t,\mathrm{r}_t,\theta^X_t)_{t \in [0,\infty)}$ with
			\begin{align*}
				&\big|f\big(\omega,t,y,\mathrm{y},z,b,a,c,\kappa\big) - f\big(\omega,t,y^\prime,\mathrm{y}^\prime,z^\prime,b,a,c,\kappa\big)\big|^2 \\
				& \leq r_t(\omega) |y-y^\prime|^2 + \mathrm{r}_t(\omega) |\mathrm{y}-\mathrm{y}^\prime|^2 + \theta^X_t(\omega) (z-z^\prime)^\top c\,(z-z^\prime),
			\end{align*}
			for any $(\omega,t) \in \Omega \times [0,\infty)$, and $(y,y^\prime,\mathrm{y},\mathrm{y}^\prime,z,z^\prime,b,a,c,\kappa) \in \R^2 \times \R^2 \times (\R^d)^2 \times \R^d \times \S^d_\smallertext{+} \times \S^d_\smallertext{+} \times \cL;$
			
			\medskip
			$(iii)$ \textnormal{\Cref{ass::generator2}.$(iv)$--$(v)$} hold with $\alpha^2 \coloneqq\max\{\sqrt{r},\sqrt{\mathrm{r}},\theta^X\}$, and with the jump integrand argument omitted.
	\end{assumption}
	
	\begin{remark}
		First of all, we note that the $\F$-progressive $\sigma$-algebra and $\F$-optional $\sigma$-algebra on $\Omega \times [0,\infty)$ are identical$;$ see \textnormal{\cite[Theorem 97]{dellacherie1978probabilities}}.
		Secondly, it is important to note the following measurability properties to define the corresponding \textnormal{BSDEs} and the value function in this section: for fixed $s \in [0,\infty)$ and $(y,\mathrm{y},z) \in \R \times \R \times \R^d$, the map
		\[
			\Omega \times \Omega \times [0,\infty) \times \R^d \times \S^d_\smallertext{+} \times \S^d_\smallertext{+} \times \cL \ni (\bar\omega,\omega,t,b,a,c,\kappa) \longmapsto f(\bar\omega\otimes_s\omega,s+t,y,\mathrm{y},z,b,a,c,\kappa) \in \R
		\]
		is Borel-measurable by \textnormal{\Cref{ass::generator_without_random_measure}.$(i)$} and an application of \textnormal{\cite[Theorem IV.96.(d)]{dellacherie1978probabilities}}. In particular, for fixed $\bar\omega \in \Omega$,
		\begin{equation}\label{eq::borel_measurability_without_random_measure}
			\Omega \times [0,\infty) \times \R^d \times \S^d_\smallertext{+} \times \S^d_\smallertext{+} \times \cL \ni (\omega, t, b,a,c,\kappa) \longmapsto f(\bar\omega\otimes_s\omega,s+t,y,\mathrm{y},z,b,a,c,\kappa) \in \R
		\end{equation}
		is Borel-measurable. Then
		 \textnormal{\Cref{ass::generator_without_random_measure}.$(i)$} together with \textnormal{\cite[Theorem IV.97.(b)]{dellacherie1978probabilities}} implies that
		\begin{align*}
			f\big(\bar\omega\otimes_s\omega,s+t,y,\mathrm{y},z,b,a,c,\kappa\big) 
			&= f\big((\bar\omega\otimes_s\omega)_{\cdot\land(s+t)},s+t,y,\mathrm{y},z,b,a,c,\kappa\big) = f\big(\bar\omega\otimes_s\omega_{\cdot\land t},s+t,y,\mathrm{y},z,b,a,c,\kappa\big).
		\end{align*}
	Therefore, the map \textnormal{\eqref{eq::borel_measurability_without_random_measure}} is $\textnormal{Prog}(\F)\otimes\cB(\R^d)\otimes\cB(\S^d_\smallertext{+})\otimes\cB(\S^d_\smallertext{+})\otimes\cB(\cL)$-measurable $($see \textnormal{\cite[Theorem IV.97.(a)]{dellacherie1978probabilities}}$)$. The additional product-measurability in $(y, \mathrm{y}, z)$ follows from the continuity in those variables$;$ see \textnormal{\cite[Lemma 4.51]{aliprantis2006infinite}}.
	\end{remark}
	
	\subsection{Value function: measurability, regularisation and decomposition}

	For $(\omega,s) \in \Omega \times [0,\infty)$, and then $\P \in \fP(s,\omega)$, we denote by $\cY^{s,\omega,\P}((T- s \land T)^{s,\omega},\xi^{s,\omega})$ the first component of the solution $(\cY,\cZ,\cN)$ to the BSDE
	\begin{align*}
	\nonumber	\cY_t &= \xi^{s,\omega} + \int_t^{(T\smallertext{-}s\land T)^{\smalltext{s}\smalltext{,}\smalltext{\omega}}} f^{s,\omega,\P}_r(\cY_r,\cY_{r\smallertext{-}},\cZ_r)\d (C^{s,\omega}_{s\smallertext{+}\smallertext{\cdot}})_r  - \bigg(\int_t^{(T\smallertext{-}s\land T)^{\smalltext{s}\smalltext{,}\smalltext{\omega}}} \cZ_r \d X_r\bigg)^{(\P)}
		- \int_t^{(T\smallertext{-}s\land T)^{\smalltext{s}\smalltext{,}\smalltext{\omega}}}\d \cN_r, \; t \in [0,\infty], \; \textnormal{$\P$--a.s.},
	\end{align*}
	relative to $(\F_{+},\P)$, the well-posedness of which is guaranteed by the results in \cite[Section 3.2]{possamai2024reflections}, which also cover BSDEs driven by martingales with jumps.

	\begin{theorem}
		Suppose that \textnormal{\Cref{ass::probabilities2_dom}} and \textnormal{\ref{ass::generator_without_random_measure}} hold. For every $s \in [0,\infty)$, the function $\widehat{\cY}_s(T,\xi) : \Omega \longrightarrow [-\infty,\infty]$ defined by
		\[
			\widehat{\cY}_s(T,\xi)(\omega) \coloneqq \sup_{\P \in \fP(s,\omega)}\E^\P\big[\cY^{s,\omega,\P}_0((T-s\land T)^{s,\omega},\xi^{s,\omega})\big]
		\]
		is upper semi-analytic and $\cF_s$--universally measurable. Moreover, $\widehat{\cY}_s(T,\xi)(\omega) = \widehat{\cY}_{s \land T(\omega)}(T,\xi)(\omega)$ for all $\omega \in \Omega$ and
		\begin{equation*}
			\widehat\cY_s(T,\xi) = \underset{\bar{\P} \in \fP_\smalltext{0}(\cF_\smallertext{s},\P)}{{\esssup}^\P} \E^{\bar{\P}} \big[ \cY^{\bar{\P}}_s(T,\xi)\big| \cF_s\big], \; \textnormal{$\P$--a.s.}, \; \P \in \fP_0,
		\end{equation*}
		where $\fP_0(\cF_s,\P) =\big\{\overline{\P} \in \fP_0: \overline{\P} = \P \; \textnormal{on} \; \cF_s\big\}$.
	\end{theorem}
	
	The proof of the previous result is analogous to that of \cite[Theorem~3.1]{possamai2026aggregation}. More precisely, the analogues of \cite[Proposition~4.3 and Lemma~4.4]{possamai2026aggregation} follow from the same arguments, and \cite[Lemma~4.2]{possamai2026aggregation}, which allows us to measurably select the integrands in the BSDE there, is replaced by the following result.
	
	\begin{lemma}
		Let $\Omega \times \fH^{2}_\textnormal{loc} \times \Omega \times [0,\infty) \ni (\omega,\P,\tilde\omega,t) \longmapsto \cM^{\omega,\P}_t(\tilde\omega)\in\R$ be Borel-measurable. Suppose that for every $(\omega,\P) \in \Omega\times\fH^{2}_\textnormal{loc}$, the process $\cM^{\omega,\P}$ is a right-continuous, $(\F_\smallertext{+},\P)$--square-integrable martingale. There exists a $\cF\otimes\cB(\fH^2_\textnormal{loc})\otimes\cP$-measurable map
		\begin{equation}\label{eq::measurability_Z}
			\Omega\times\fH^{2}_\textnormal{loc} \times \Omega \times [0,\infty) \ni (\omega,\P,\tilde\omega,t) \longmapsto \cZ^{\omega,\P}_t(\tilde\omega) \in \R^d,
		\end{equation}
		such that $\cZ^{\omega,\P} \in \H^2(X;\F,\P)$ and
		\[
			\cN^{\omega,\P} \coloneqq \cM^{\omega,\P} - \cM^{\omega,\P}_0 - \big(\cZ^{\omega,\P}\bcdot X\big)^{(\P)},
		\]
		belongs to $\cH^{2,\perp}_0(X;\F_\smallertext{+},\P)$.
	\end{lemma}	
	
	\begin{proof}
		For fixed $(\omega,\P)$, we decompose $\cM^{\omega,\P}$ along $X$ such that for $Z^{\omega,\P} \in \H^2(X;\F,\P)$, we have that
		\[
			N^{\omega,\P} = \cM^{\omega,\P} - \cM^{\omega,\P}_0 - (Z^{\omega,\P}\bcdot X)^{(\P)}
		\]
		belongs to $\cH^{2,\perp}_0(X;\F_\smallertext{+},\P)$; see \cite[Proposition~2.6]{possamai2024reflections}. Then (see \cite[Theorem III.6.4]{jacod2003limit})
		\[
			\langle \cM^{\omega,\P},X^j\rangle^{(\P)} = \langle (Z^{\omega,\P}\bcdot X)^{(\P)},X^j\rangle^{(\P)} = \Bigg(\sum_{i = 1}^d (Z^{\omega,\P})^i \frac{\d\langle X^i,X^j\rangle^{(\P)}}{\d\textnormal{Tr}(\langle X \rangle^{(\P)})}\Bigg)\bcdot \textnormal{Tr}(\langle X\rangle^{(\P)}), \; \textnormal{$\P$--a.s.}, \; j \in \{1,\ldots,d\}
		\]
		or, compactly as column vectors,
		\[
			\frac{\d\langle \cM^{\omega,\P},X\rangle^{(\P)}}{\d\textnormal{Tr}(\langle X\rangle^{(\P)})} = 
			\frac{\d\langle X\rangle^{(\P)}}{\d\textnormal{Tr}(\langle X\rangle^{(\P)})} Z^{\omega,\P}, \; \textnormal{$\P\otimes\d\textnormal{Tr}(\langle X\rangle^{(\P)})$--a.e.}
		\]
		It is then straightforward to check that
		\[
			\cZ^{\omega,\P}
			\coloneqq \bigg( \frac{\d\langle X\rangle^{(\P)}}{\d\textnormal{Tr}(\langle X\rangle^{(\P)})} \bigg)^{\oplus} \bigg(\frac{\d\langle \cM^{\omega,\P},X\rangle^{(\P)}}{\d\textnormal{Tr}(\langle X\rangle^{(\P)})}\bigg),
		\]
		where $A^\oplus$ denotes the Moore--Penrose pseudo-inverse of a matrix $A$, satisfies
		\[
			\|\cZ^{\omega,\P}-Z^{\omega,\P}\|_{\H^\smalltext{2}(X;\F,\P)} = 0.
		\]
		Thus, it suffices to show that $\cZ^{\omega,\P}$ can be constructed measurably in $\P$. This is the remaining task in this proof.
		
		\medskip
		It follows from the arguments used in the proof of \cite[Proposition 5.1]{neufeld2014measurability} together with an extension of \cite[Lemma~B.1]{possamai2026aggregation} that there exists a $\cF\otimes\cB(\fH^2_\textnormal{loc})\otimes\cP$-measurable map
		\[
			\Omega\times\fH^{2}_\textnormal{loc} \times \Omega \times [0,\infty) \ni (\omega,\P,\tilde\omega,t) \longmapsto \langle \cM^{\omega,\P},X\rangle_t(\tilde\omega) \in \R^d,
		\]
		such that each $\langle \cM^{\omega,\P},X\rangle$ is $\P$-indistinguishable from $\langle \cM^{\omega,\P},X\rangle^{(\P)}$. 
		
		\medskip
		Next, we choose versions of the characteristics $(\mathsf{B}^\P,\mathsf{C},\nu^\P)$ of $X$ which are measurable in the probability law $\P$; see \cite[Theorem 2.5]{neufeld2014measurability}. We then define, component-wise, the process
		\[
			a^{X,\P}_t \coloneqq \tilde{a}^{X,\P}_t \1_{\S^\smalltext{d}_\smalltext{+}}(\tilde{a}^{X,\P}_t),
		\]
		where
		\[
			\tilde{a}^{X,\P}_t 
			\coloneqq 
			\limsup_{n \rightarrow \infty} \frac{(\mathsf{C}_t+ (x x^\top) \ast \nu^\P_t) - (\mathsf{C}_{(t\smallertext{-}1/n)\lor 0}+ (x x^\top) \ast \nu^\P_{(t\smallertext{-}1/n)\lor 0})}{\textnormal{Tr}(\mathsf{C}_t+ (x x^\top) \ast \nu^\P_t) - \textnormal{Tr}(\mathsf{C}_{(t\smallertext{-}1/n)\lor 0}+ (x x^\top) \ast \nu^\P_{(t\smallertext{-}1/n)\lor 0})}, t \in [0,\infty).
		\]
		Here, we use our usual conventions $0/0 = 0$ and $\infty - \infty = -\infty$.
		Then
		\[
			a^{X,\P} = \frac{\d\langle X\rangle^{(\P)}}{\d\textnormal{Tr}(\langle X\rangle^{(\P)})}, \; \textnormal{$\P\otimes\d\textnormal{Tr}(\langle X\rangle^{(\P)})$--a.e.}
		\]
		We then also define
		\[
			a^{\cM,X,\omega,\P}_t \coloneqq \tilde{a}^{\cM,X,\omega,\P}_t \1_{\R^\smalltext{d}}(\tilde{a}^{\cM,X,\omega,\P}_t),
		\]
		where
		\[
			\tilde{a}^{\cM,X,\omega,\P}_t 
			\coloneqq 
			\limsup_{n \rightarrow \infty} \frac{\langle \cM^{\omega,\P},X\rangle_t - \langle \cM^{\omega,\P},X\rangle_{(t-1/n)\lor 0}}{\textnormal{Tr}(\mathsf{C}_t+ (x x^\top) \ast \nu^\P_t) - \textnormal{Tr}(\mathsf{C}_{(t\smallertext{-}1/n)\lor 0}+ (x x^\top) \ast \nu^\P_{(t\smallertext{-}1/n)\lor 0})}, t \in [0,\infty).
		\]
		We then have
		\[
			a^{\cM,X,\omega,\P} = \frac{\d\langle \cM^{\omega,\P},X\rangle^{(\P)}}{\d\textnormal{Tr}(\langle X\rangle^{(\P)})}, \; \textnormal{$\P\otimes\d\textnormal{Tr}(\langle X\rangle^{(\P)})$--a.e.},
		\]
	 	and can thus define $\cZ^{\omega,\P}$ as
		\[
			\cZ^{\omega,\P} \coloneqq (a^{X,\P})^\oplus a^{\cM,X,\omega,\P},
		\]
		which implies that the map in \eqref{eq::measurability_Z} is Borel-measurable, and each $\cZ^{\omega,\P}$ is $\F_\smallertext{+}$-adapted and $\F^\P_\smallertext{+}$-predictable. To obtain a $\cF\otimes\cB(\fH^2_\textnormal{loc})\otimes\cP$-measurable map $(\omega,\P,\tilde{\omega},t) \longmapsto \cZ^{\omega,\P}_t(\tilde{\omega})$, we apply an extension of \cite[Lemma~B.1]{possamai2026aggregation} again. This concludes the proof.
	\end{proof}
	
	Our main additional assumption in this section is as follows; we refer to \textnormal{\cite[Remark~3.4]{possamai2026aggregation}} for the (analogous) reasoning behind it.
	
\begin{assumption}\label{ass::crossing_dom}
	For every $(\omega,s) \in \Omega \times [0,\infty)$ and every $\P \in \fP(s,\omega)$ the following hold
	
\begin{enumerate}
\item[$(i)$] for any $(\cY,\cY^{\prime},\cZ) \in \big(\cS^{2,s,\omega}_{T,\hat\beta}(\F_\smallertext{+},\P)\big)^2 \times \H^{2,s,\omega}_{T,\hat\beta}(X;\F,\P)$, there is an $\F$-progressive process $\lambda = \lambda^{s,\omega,\P}\1_{\llparenthesis 0,(T-s\land T)^{\smallertext{s}\smallertext{,}\smallertext{\omega}} \rrbracket}$ such that $\lambda\Delta C^{s,\omega}_{s\smallertext{+}\smallertext{\cdot}} > -1$ holds \textnormal{$\P$--a.s.}, $|\lambda| \leq \sqrt{r^{s,\omega}_{s\smallertext{+}\smallertext{\cdot}}}$ holds \textnormal{$\P\otimes \d C^{s,\omega}_{s\smallertext{+}\smallertext{\cdot}}$--a.e.} on $\llparenthesis 0, (T-s\land T)^{s,\omega} \rrbracket$, and
		\begin{equation*}
			f^{s,\omega,\P}(\cY,\cY_{\smallertext{-}},\cZ) - f^{s,\omega,\P}(\cY^{\prime},\cY_{\smalltext{-}},\cZ) 
			\geq \lambda (\cY - \cY^{\prime}), \; \textnormal{$\P \otimes \mathrm{d} C^{s,\omega}_{s\smallertext{+}\smallertext{\cdot}}$--a.e. on $\llparenthesis 0, (T-s\land T)^{s,\omega} \rrbracket$$;$}
		\end{equation*}
\item[$(ii)$]	for every $(\cZ,\cZ^\prime) \in \big(\H^{2,s,\omega}_{T,\hat\beta}(X;\F,\P)\big)^2$, we have $\eta^{s,\omega,\P,\cZ,\cZ^\smalltext{\prime}}\Delta X > -1$, \textnormal{$\P$--a.s.}, where
	\begin{equation*}
		\eta^{s,\omega,\P,\cZ,\cZ^\smalltext{\prime}}_t 
		\coloneqq -\sqrt{(\theta^X)^{s,\omega}_{s\smallertext{+}t}} \frac{(\cZ_t-\cZ_t^\prime)}{|(\cZ_t-\cZ^\prime_t)^\top c^{s,\omega,\P}_t(\cZ_t-\cZ_t^\prime)|^{1/2}}\mathbf{1}_{\{(\cZ_\smalltext{t}-\cZ_\smalltext{t}^\smalltext{\prime})^\smalltext{\top}c^{\smalltext{s}\smalltext{,}\smalltext{\omega}\smalltext{,}\smalltext{\P}}_\smalltext{t}(\cZ_\smalltext{t}-\cZ_\smalltext{t}^\smalltext{\prime})\neq 0\}}, \; t \in [0,\infty).
	\end{equation*}

\item[$(iii)$] there exists $\mathfrak{C} = \mathfrak{C}(\omega,s,\P) \in [0,\infty)$ such that
		\begin{equation*}
			\int_0^{(T\smallertext{-}s\land T)^{\smalltext{s}\smalltext{,}\smalltext{\omega}}}\max\Big\{\sqrt{\mathrm{r}^{s,\omega}_{s\smallertext{+}t}}, (\theta^{X})^{s,\omega}_{s\smallertext{+}t})\Big\} \d (C^{s,\omega}_{s\smallertext{+}\smallertext{\cdot}})_t \leq \mathfrak{C}, \; \textnormal{$\P$--a.s.};
		\end{equation*}
\end{enumerate}
\end{assumption}

\begin{theorem}\label{thm::down-crossing_dom}
	Suppose that {\rm\Cref{ass::probabilities2_dom}} and {\rm\ref{ass::generator_without_random_measure}} hold, and that {\rm\Cref{ass::crossing_dom}} holds for $s = 0$, and that
	\begin{equation}\label{eq::constant_phi_dom}
		\phi^{2,\hat{\beta}}_{\xi,f} 
		\coloneqq \sup_{\P \in \fP_\smalltext{0}} \E^\P\Bigg[\sup_{s \in \D_\tinytext{+}} \underset{\bar{\P} \in \fP_\smalltext{0}(\cF_{\smalltext{s}},\P)}{{\esssup}^\P}\E^{\bar{\P}}\bigg[\cE(\hat\beta A)_T|\xi|^2 + \int_s^T \cE(\hat\beta A)_r \frac{|f^{\bar\P}_r(0,0,0)|^2}{\alpha^2_r} \d C_r\bigg| \cF_{s}\bigg] \Bigg] < \infty.
	\end{equation}
	Then the following hold
	\begin{enumerate}
		\item[$(i)$] there exists a real-valued, c\`adl\`ag, $\G_\smallertext{+}$-adapted process $\widehat\cY^\smallertext{+}(T,\xi) = (\widehat\cY^\smallertext{+}_t(T,\xi))_{t \in [0,\infty]}$ that satisfies $\widehat\cY^{\smallertext{+}}(T,\xi) = \widehat\cY^{\smallertext{+}}_{\cdot \land T}(T,\xi)$ and $\widehat\cY^{\smallertext{+}}_T(T,\xi) = \xi$ identically, and
		\begin{equation*}
			\widehat\cY^\smallertext{+}_t(T,\xi) = \lim_{\D_\tinytext{+} \ni s \downarrow\downarrow t} \widehat\cY_s(T,\xi), \; t \in [0,\infty), \; \textnormal{$\fP_0$--q.s.;}
		\end{equation*}
		\item[$(ii)$] there exists a constant $\mathfrak{C} \in (0,\infty)$ depending only on $\hat\beta$ and $\Phi$ such that
		\begin{equation*}
			\sup_{\P\in\fP_\smalltext{0}}\E^\P\bigg[\sup_{s \in \D_\tinytext{+}}\big|\cE(\hat\beta A)^{1/2}_{s \land T}\widehat\cY_{s \land T}(T,\xi)\big|^2\bigg] + \sup_{\P\in\fP_\smalltext{0}}\E^\P\bigg[\sup_{s \in [0,T]}\big|\cE(\hat\beta A)^{1/2}_{s}\widehat\cY^\smallertext{+}_s(T,\xi)\big|^2\bigg] \leq \mathfrak{C}\phi^{2,\hat{\beta}}_{\xi,f} < \infty,
		\end{equation*}
		\begin{equation*}
			\widehat\cY^\smallertext{+}_t(T,\xi) = \underset{\bar{\P} \in \fP_\smalltext{0}(\cG_{t\smalltext{+}},\P)}{{\esssup}^\P} \cY^{\bar{\P}}_t(T,\xi), \; \textnormal{$\P$--a.s.}, \; t \in [0,\infty], \; \textnormal{$\P \in \fP_0$;}
		\end{equation*}
		\item[$(iii)$] if, in addition, {\rm\Cref{ass::crossing_dom}} holds, then for all $\G_\smallertext{+}$--stopping times $\sigma$ and $\tau$
		\begin{equation*}
			\widehat\cY^\smallertext{+}_{\sigma \land \tau \land T}(T,\xi) \geq \cY^\P_{\sigma \land \tau \land T}\big(\tau\land T,\widehat\cY^\smallertext{+}_{\tau\land T}(T,\xi)\big), \; \textnormal{$\P$--a.s.}, \; \text{\rm $\P \in \fP_0$.}
		\end{equation*}
	\end{enumerate}
\end{theorem}

\begin{remark}
	The same link in the context of optimisation problems holds, as in \textnormal{\cite[Corollary~3.7]{possamai2026aggregation}}, between $\widehat{\cY}$ and $\widehat{\cY}^\smallertext{+}$.
\end{remark}

	We now state the decomposition of $\widehat{\cY}^\smallertext{+}$.
	
\begin{proposition}\label{prop::decomposition_yplus_dom}
	Suppose that {\rm\Cref{ass::probabilities2_dom}}, {\rm\ref{ass::generator_without_random_measure}} and {\rm\ref{ass::crossing_dom}} hold, that $M^\Phi_1(\hat\beta) < 1$ $($see \eqref{eq::contraction_reflected_bsde}$)$ and that $\phi^{2,\hat{\beta}}_{\xi, f} < \infty$, where $\phi^{2,\hat{\beta}}_{\xi, f}$ is given by \eqref{eq::constant_phi_dom}. Let $\beta^\star$ be the unique number in $(0,\hat{\beta})$ at which $M^\Phi_1(\beta^\star) = 1$, and suppose that there exists $\beta \in (\beta^\star,\hat\beta)$ such that, for every $\P \in \fP_0$, $\P[\cE(\beta A)_{T\smallertext{-}} < \infty] = 1$ holds.
	
	\medskip
	$(i)$ There exists a tuple $(\widehat{Z}^\P,\widehat{N}^\P,\widehat{K}^\P)_{\P \in \fP_\smalltext{0}} = (\widehat{Z}^\P(T,\xi),\widehat{N}^\P(T,\xi),\widehat{K}^\P(T,\xi))_{\P \in \fP_\smalltext{0}}$ such that, for every $\P\in\fP_0$, $(\widehat{Z}^\P,\widehat{N}^\P)\in\H^2_{T,\beta}(X;\G,\P) \times \cH^{2,\perp}_{0,T,\beta}(X;\G_\smallertext{+},\P)$, and $\widehat{K}^\P = (\widehat{K}^\P_t)_{t \in [0,\infty]}$ is a c\`adl\`ag, non-decreasing, $\G^\P_\smallertext{+}$-predictable process starting at zero, satisfying $\widehat{K}^\P = \widehat{K}^\P_{\cdot \land T}$ and $\E^\P\big[|\widehat{K}^\P_T|^2\big] < \infty$, such that
	\begin{equation*}
		\widehat\cY^\smallertext{+}_t = \xi + \int_t^T f^\P_r\big(\widehat\cY^\smallertext{+}_r,\widehat\cY^\smallertext{+}_{r\smallertext{-}},\widehat{Z}^\P_r\big) \d C_r - \bigg(\int_t^T \widehat{Z}^\P_r \d X_r\bigg)^{(\P)} - \int_t^T \d \widehat{N}^\P_r + \widehat{K}^\P_T - \widehat{K}^\P_t,
	\end{equation*}
	holds for every $t \in [0,\infty]$, \textnormal{$\P$--a.s.}
	
	\medskip
	$(ii)$ Any other tuple $(Z^\P,N^\P,K^\P)_{\P \in \fP_\smalltext{0}}$ meeting the conditions of $(i)$ satisfies $\sup_{\P \in \fP_\smalltext{0}}\|Z^\P - \widehat{Z}^\P\|_{\H^\smalltext{2}_\smalltext{T}(X;\G,\P)} = 0$, and $(N^\P,K^\P) = (\widehat{N}^\P,\widehat{K}^\P)$ outside a $\P$--null set for every $\P \in \fP_0$.
\end{proposition}

	\subsection{Characterisation as the solution to a 2BSDE system without random measures}
	
	The regularised value function and its decomposition solve a 2BSDE in the spirit of \cite{soner2012wellposedness, soner2013dual}. More precisely, $\big(\widehat\cY^\smallertext{+}(T,\xi), (\widehat{Z}^\P(T,\xi),\widehat{N}^\P(T,\xi), \widehat{K}^\P(T,\xi))_{\P \in \fP_\smalltext{0}}\big)$ is the unique (in a sense clarified below) solution $(Y, (Z^\P,N^\P, K^\P)_{\P \in \fP_\smalltext{0}})$ to the following 2BSDE system.

	\begin{definition}
		A tuple $(Y,(Z^\P, N^\P, K^\P)_{\P \in \fP_\smalltext{0}})$ is said to be a solution to the \textnormal{extrinsic 2BSDE \ref{2BSDE::measurability_dom}--\ref{2BSDE::aggregation_dom}} if
	\begin{enumerate}[{\bf (2B1$^\prime$)}, leftmargin=1.3cm]
		\item\label{2BSDE::measurability_dom} $Y = (Y_t)_{t \in [0,\infty]}$ is a real-valued, c\`adl\`ag, $\G_\smallertext{+}$-adapted process$;$
		\item\label{2BSDE::dynamics_dom} $Z^\P \in \H^2_T(X;\G,\P)$, $N^\P \in \cH^{2,\perp}_{0,T}(X;\G_\smallertext{+},\P)$ and $K^\P = (K^\P_t)_{t \in [0,\infty]}$ is a real-valued, c\`adl\`ag and non-decreasing, $\G^\P_\smallertext{+}$-predictable process starting at zero, satisfying $K^\P = K^\P_{\cdot\land T}$ and $\E^\P\big[|K^\P_T|^2\big] < \infty$ such that
		\[
		\int_0^T \big|f^\P_s(Y_s,Y_{s\smallertext{-}},Z^\P_s)\big|\d C_s < \infty, \; \text{$\P$--a.s.},
		\]
		and
		\begin{align*}
			\displaystyle Y_t &= \xi + \int_t^T f^\P_s(Y_s,Y_{s\smallertext{-}},Z^\P_s) \d C_s- \bigg(\int_t^T Z^\P_s \d X_s\bigg)^{(\P)} - \int_t^T \d N^\P_s + \int_t^T \d K^\P_s, \; t \in [0,\infty], \; \textnormal{$\P$--a.s.},
		\end{align*}
		holds for all $\P \in \fP_0$$;$
		\item\label{2BSDE::aggregation_dom} $Y_t = \underset{\bar{\P} \in \fP_\smalltext{0}(\cG_{\smallertext{t}\smalltext{+}},\P)}{{\esssup}^\P} \cY^{\bar{\P}}_t(T,\xi)$, $\P$--a.s., $t \in [0,\infty], \; \P \in \fP_0$.
	\end{enumerate}
	\end{definition}
	
	As done previously, we introduce the space
	\[
	\sL^2_{T,\beta}(\fP_0) \coloneqq \Big\{(Z^\P,N^\P)_{\P \in \fP_0} : (Z^\P,N^\P) \in \H^2_{T,\beta}(X;\G,\P) \times \cH^{2,\perp}_{T,\beta}(X;\G_\smallertext{+},\P),\; \P \in \fP_0 \Big\}.
	\]
	The well-posedness of the 2BSDE system \ref{2BSDE::measurability_dom}--\ref{2BSDE::aggregation_dom} is established in the following result.
\begin{theorem}\label{thm::2BSDE_wellposed_dom}
		Suppose that {\rm\Cref{ass::probabilities2_dom}}, {\rm\ref{ass::generator_without_random_measure}} and {\rm\ref{ass::crossing_dom}} hold, and that $\phi^{2,\hat{\beta}}_{\xi, f} < \infty$, where $\phi^{2,\hat{\beta}}_{\xi, f}$ is given by \eqref{eq::constant_phi_dom}. Let $\beta^\star$ be the unique number in $(0,\hat{\beta})$ at which $M^\Phi_1(\beta^\star) = 1$ $($see \eqref{eq::contraction_reflected_bsde}$)$, and suppose that there exists $\beta \in (\beta^\star,\hat\beta)$ such that $\P[\cE(\beta A)_{T\smallertext{-}} < \infty] = 1$ holds for every $\P \in \fP_0$. 
		\begin{enumerate}
			\item[$(i)$] 
			There exists a solution $(\widehat{Y}, (\widehat{Z}^\P, \widehat{N}^\P, \widehat{K}^\P)_{\P \in \fP_\smalltext{0}})$ to the \textnormal{extrinsic 2BSDE \ref{2BSDE::measurability_dom}--\ref{2BSDE::aggregation_dom}} with $(\widehat{Z}^\P,\widehat{N}^\P)_{\P \in \fP_\smalltext{0}} \in 
			\sL^2_{T,\beta}(\fP_0)$.
			\item[$(ii)$] Any other solution $(Y,(Z^\P,N^\P,K^\P)_{\P \in \fP_\smalltext{0}})$ to the \textnormal{extrinsic 2BSDE \ref{2BSDE::measurability_dom}--\ref{2BSDE::aggregation_dom}} with $(Z^\P,N^\P)_{\P \in \fP_\smalltext{0}} \in 
			\sL^2_{T,\beta}(\fP_0)$ is such that $\sup_{\P \in \fP_\smalltext{0}} \|Z^\P - \widehat{Z}^\P\|_{\H^\smalltext{2}_\smalltext{T}(X;\G,\P)} = 0$, and $(Y,N^\P,K^\P) = (\widehat{Y},\widehat{N}^\P,\widehat{K}^\P)$ outside a $\P$--null set for every $\P \in \fP_0$.
		\end{enumerate}
	\end{theorem}
	
	Lastly, we close this section with some norm bounds of the solution to the 2BSDE system by the terminal condition and generator evaluated at zero.
	
	\begin{proposition}\label{prop::bound_and_comparison2}
		Suppose that {\rm\Cref{ass::probabilities2_dom}}, {\rm\ref{ass::generator_without_random_measure}} and {\rm\ref{ass::crossing_dom}} hold, and that $\phi^{2,\hat{\beta}}_{\xi, f} < \infty$, where $\phi^{2,\hat{\beta}}_{\xi, f}$ is given by \eqref{eq::constant_phi_dom}. Let $\beta^\star$ be the unique number in $(0,\hat{\beta})$ at which $M^\Phi_1(\beta^\star) = 1$ $($see \eqref{eq::contraction_reflected_bsde}$)$, and suppose that there exists $\beta \in (\beta^\star,\hat\beta)$ such that $\P[\cE(\beta A)_{T\smallertext{-}} < \infty] = 1$ holds for every $\P \in \fP_0$. Let $(\widehat{Y},(\widehat{Z}^\P,\widehat{N}^\P,\widehat{K}^\P\big)_{\P \in \fP_\smalltext{0}})$ be the solution to the \textnormal{extrinsic 2BSDE \ref{2BSDE::measurability_dom}--\ref{2BSDE::aggregation_dom}} from \textnormal{\Cref{thm::2BSDE_wellposed_dom}}.
		\begin{enumerate}
			\item[$(i)$] (Norm bound): There exists a constant $\mathfrak{C} \in (0,\infty)$ that depends only on $\hat\beta$, on $\beta$, and on $\Phi$, such that
			\begin{equation*}
				\sup_{\P \in \fP_\smalltext{0}}
				\Bigg\{\big\|\widehat Y\big\|^2_{\cS^\smalltext{2}_{\smalltext{T}\smalltext{,}\smalltext{\hat\beta}}(\G_\tinytext{+},\P)} 
				+ \big\|\widehat{Z}^\P\big\|^2_{\H^\smalltext{2}_{\smalltext{T}\smalltext{,}\smalltext{\beta}}(X;\G,\P)} 
				+ \big\|\widehat{N}^\P\big\|^2_{\cH^{\smalltext{2}}_{\smalltext{T}\smalltext{,}\smalltext{\beta}}(\G_\tinytext{+},\P)}
				+ \E^\P\bigg[\int_0^T\cE(\beta A)_r \d[\widehat{K}^\P]_r\bigg] \Bigg\}
				\leq \mathfrak{C}\phi^{2,\hat{\beta}}_{\xi, f}.
			\end{equation*}
			
			\item[$(ii)$] \textnormal{(Comparison):} Let $f^\prime$ and $\xi^\prime$ be maps satisfying the same assumptions as $f$ and $\xi$, respectively, and assume that $f^\prime$ has the same Lipschitz coefficients as $f$. Suppose that $\phi^{2,\hat{\beta}}_{\xi^\smalltext{\prime},f^\smalltext{\prime}}<\infty$. Let $\widehat Y^\prime$ denote the first component of the solution to \ref{2BSDE::measurability_dom}--\ref{2BSDE::aggregation_dom} with generator $f^\prime$ and terminal condition $\xi^\prime$. Suppose that $\xi^\prime \leq \xi$, \textnormal{$\fP_0$}--quasi-surely, and that
		\begin{equation*}
			f^{\prime,\P}_r\big(\cY^{\prime,\P}_r,\cY^{\prime,\P}_{r\smallertext{-}},\cZ^{\prime,\P}_r\big) 
			\leq f^\P_r\big(\cY^{\prime,\P}_r,\cY^{\prime,\P}_{r\smallertext{-}},\cZ^{\prime,\P}_r\big), \; \textnormal{$\P\otimes\d C_r$--a.e.}, \; \P \in \fP_0,
		\end{equation*}
		where $(\cY^{\prime,\P},\cZ^{\prime,\P},\cN^{\prime,\P})$ denotes the solution to the corresponding $\P$--\textnormal{BSDE} with generator $f^{\prime,\P}$ and terminal condition $\xi^\prime$. Then, $\widehat Y^{\prime} \leq \widehat Y$, \textnormal{$\fP_0$}--quasi-surely.
		\end{enumerate}
	\end{proposition}

	\section{(2)BSDEs with different decompositions: an informal take on the aggregation issue}\label{sec_aggregation_problems}
	
	Second-order BSDEs can be applied to study control problems, in classical formulation, under model uncertainty, see, for example, \cite{matoussi2015robust,possamai2013robust}, or to study stochastic target problems under model uncertainty, see, for example, \cite{soner2013dual,nutz2012superhedging} or \cite[Remark 12.4.1]{zhang2017backward}. Even \cite{neufeld2013superreplication,nutz2015robust} are related to second-order BSDEs with generator identically zero, although not explicitly stated. A commonality in those works is that the solution of the 2BSDE, except possibly the non-decreasing component, does not depend on the corresponding initial class of probability measures $\P\in\fP_0$ (playing the part of the uncertainty in the model). This feature allows one to construct robust optimal controls to the given problem. This means that the control is optimal for the worst-case outcome of the control problem under study. It is therefore desirable that the solution to the 2BSDEs considered in this work, except possibly the orthogonal martingale and predictable non-increasing components, are all aggregated and defined independently of the probability measures in $\fP_0$. In \Cref{sec_aggregation_without_random_measures}, we show that even the (2)BSDEs introduced \Cref{sec_2bsde_X} inherently suffer, from our point of view, from the issue of non-aggregation. We briefly discuss a different possible decomposition of $X$ in relation to (2)BSDEs in \Cref{sec_2bsdes_different_decompositions}, and explain why, from our point of view, the issue of aggregation persists.
	
	\subsection{The aggregation issue without random measures}\label{sec_aggregation_without_random_measures}
	
	\medskip
	We discuss the problem of aggregating the integrands $(\widehat{Z}^\P)_{\P\in\fP_0}$ from \Cref{thm::2BSDE_wellposed_dom} into a single process $\widehat{Z}$, in the sense that we would desire
	\begin{equation}\label{eq_Z_independent_P}
		\sup_{\P \in \fP_\smalltext{0}}\|\widehat{Z} - \widehat{Z}^\P\|_{\H^\smalltext{2}_\smalltext{T}(X;\G,\P)} = 0.
	\end{equation}
	
	Let us suppose for the moment that $d = 1$. We fix $\P\in\fP_0$. From \ref{2BSDE::dynamics_dom}, it follows that
	\[
		(\widehat{\cY}^\smallertext{+})^{c,\P} = (\widehat{Z}^\P \bcdot X^{c,\P})^{\P} + \widehat{N}^{c,\P}, \; \textnormal{$\P$--a.s.},
	\]
	which yields
	\begin{equation}\label{eq_aggreagtion_X_continuous}
		\langle (\widehat{\cY}^\smallertext{+})^{c,\P}, X^{c,\P} \rangle^{(\P)} = \widehat{Z}^\P \bcdot \langle X \rangle^{(\P)} + \langle \widehat{N}^{c,\P}, X^{c,\P} \rangle^{(\P)}, \; \textnormal{$\P$--a.s.}
	\end{equation}
	If $X$ were $\P$--a.s. continuous, then the $\P$--orthogonality of $X$ and $\widehat{N}$ would imply
	\[
		\langle \widehat{N}^{c,\P}, X^{c,\P} \rangle^{(\P)} = \langle \widehat{N}^{c,\P}, X^{c,\P} \rangle^{(\P)} + \langle \widehat{N}^{d,\P}, X^{c,\P} \rangle^{(\P)} = \langle \widehat{N}^{\P}, X^{c,\P} \rangle^{(\P)} = \langle \widehat{N}^{\P}, X \rangle^{(\P)} = 0, \; \textnormal{$\P$--a.s.}
	\]
	Here, $\widehat{N}^{d,\P}$ denotes the $\P$--purely discontinuous square-integrable martingale part of $\widehat{N}^\P$. It would then follow from \eqref{eq_aggreagtion_X_continuous} that
	\[
		\widehat{Z}^\P = \frac{\d \langle (\widehat{\cY}^\smallertext{+})^{c,\P}, X^{c,\P} \rangle^{(\P)}}{\d \langle X^{c,\P} \rangle}, \; \textnormal{$\P \otimes \d \langle X^{c,\P} \rangle$--a.e.}
	\]
	The predictable quadratic variations on the right-hand side aggregate into real-valued and $\G$-predictable terms (see \cite[Proposition 6.6]{neufeld2014measurability}), and thus the above derivative does so too, which yields a term $\widehat{Z}$ that is independent of the probability measure $\P$ and satisfies \eqref{eq_Z_independent_P}.
	
	\medskip
	However, if $X$ is not $\P$--a.s. continuous for some $\P\in\fP_0$, then $\langle \widehat{N}^{c,\P}, X \rangle^{(\P)}$ might not vanish under $\P$, and the same manipulations as above would yield
	\[
		\widehat{Z}^\P = \frac{\d \langle (\widehat{\cY}^\smallertext{+})^{c,\P}, X^{c,\P} \rangle^{(\P)}}{\d \langle X^{c,\P} \rangle} - \frac{\d \langle \widehat{N}^{c,\P}, X^{c,\P} \rangle^{(\P)}}{\d \langle X^{c,\P} \rangle}, \; \textnormal{$\P \otimes \d \langle X^{c,\P} \rangle$--a.e.}
	\]
	The aggregation of $(\widehat{Z}^\P)_{\P \in \fP_0}$ is thus now equivalent to the aggregation of $(\langle \widehat{N}^{c,\P}, X^{c,\P} \rangle^{(\P)})_{\P \in \fP_0}$; however, it is not clear whether such an aggregation is possible, as each process $\widehat{N}^\P$ arises from a specific $\P$--orthogonal martingale decomposition, and the underlying martingale itself depends on the probability measure $\P$.
	
	\medskip
	In that regard, the weakest condition we were able to come up with for the above arguments to go through is that every probability measure $\P\in\fP_0$ satisfies the following property: if $M^\P = (M^\P_t)_{t \in [0,\infty)}$ is an $(\F_\smallertext{+},\P)$--square-integrable martingale with orthogonal decomposition
	\[
		M^\P = M^\P_0 + (Z^\P\bcdot X)^{(\P)} + N^\P,
	\]
	along $X$, then $\langle N^\P, X^{c,\P} \rangle^{(\P)} = \langle N^{c,\P}, X^{c,\P} \rangle^{(\P)} = 0$.

\medskip
	In that case, we would be able to aggregate, as outlined above, the integrands $(\widehat{Z}^\P)_{\P \in \fP_0}$ of the stochastic integrals $((\widehat{Z}^\P \bcdot X^{c,\P})^{(\P)})_{\P\in\fP_0}$ into a single process $\widehat{Z}$. Together with the dominating diffusion property from \cite{nutz2015robust}, which requires that the compensator $\nu^\P$ of the jump measure of $X$ satisfy $(|x|^2 \land 1)\ast\nu^\P \ll\langle X^{c,\P} \rangle^{(\P)}$ for every $\P \in\fP_0$, we would be able to deduce that the aggregated integrand $\widehat{Z}$ even satisfies
	\[
		(\widehat{Z}\bcdot X)^{(\P)} = (\widehat{Z}^\P\bcdot X)^{(\P)}, \; \textnormal{$\P$--a.s.}, \; \P\in\fP_0,
	\]
	so that $\widehat{Z}$ would indeed be the desired aggregator. 
	
	\medskip
	However, the issue with the assumption $\fP_0 \subseteq \fH^{2,\dagger}_\textnormal{loc}$ is that, under each $\P \in \fH^{2,\dagger}_\textnormal{loc}$, the process $X$ is $\P$--a.s. continuous, and thus we fall back to the setting described at the beginning of this section. Indeed, for $\P \in \fH^{2,\dagger}_\textnormal{loc}$, we would have that any $(\F_\smallertext{+},\P)$--square-integrable martingale $M^\P = (M^\P_t)_{t \in [0,\infty)}$ with corresponding decomposition $M^\P -M^\P_0 = (Z^\P \bcdot X)^{(\P)} +N^\P$ satisfies
	\[
		\langle M^\P,X^{c,\P}\rangle^{(\P)} = Z^\P\bcdot \langle X^{c,\P}\rangle^{(\P)} + \langle N^\P,X^{c,\P} \rangle^{(\P)} = Z^\P\bcdot \langle X^{c,\P}\rangle^{(\P)} + \langle N^{c,\P},X^{c,\P} \rangle^{(\P)} = Z^\P\bcdot \langle X^{c,\P}\rangle^{(\P)},
	\]
	and therefore
	\begin{equation}\label{eq_decomposition_M_dom}
		Z^\P = \frac{\d\langle M^\P,X^{c,\P}\rangle^{(\P)}}{\d\langle X^{c,\P}\rangle^{(\P)}}, \; \textnormal{$\P\otimes\d\langle X^{c,\P}\rangle^{(\P)}$--a.e.}
	\end{equation}
	Since the orthogonal decomposition immediately implies
	\begin{equation*}
		Z^\P = \frac{\d\langle M^\P,X\rangle^{(\P)}}{\d\langle X\rangle^{(\P)}}, \; \textnormal{$\P\otimes\d\langle X\rangle^{(\P)}$--a.e.},
	\end{equation*}
	and the assumption of dominating diffusion implies $\langle X \rangle^{(\P)} = \langle X^{c,\P} \rangle^{(\P)} + x^2\ast\nu^\P \approx\langle X^{c,\P}\rangle^{(\P)}$, $\P$--a.s., we deduce that \eqref{eq_decomposition_M_dom} also holds outside a $\P \otimes \d \langle X \rangle^{(\P)}$--null set, which yields
	\[
		\frac{\d\langle M^\P,X\rangle^{(\P)}}{\d\langle X\rangle^{(\P)}} = Z^\P = \frac{\d\langle M^\P,X^{c,\P}\rangle^{(\P)}}{\d\langle X^{c,\P}\rangle^{(\P)}}, \; \textnormal{$\P\otimes\d\langle X\rangle^{(\P)}$--a.e.}
	\]
	For $M^\P = X^{d,\P}$, together with a localisation argument, we find
	\[
		\frac{\d\langle X^{d,\P},X^{d,\P}\rangle^{(\P)}}{\d\langle X\rangle^{(\P)}} = \frac{\d\langle X^{d,\P},X\rangle^{(\P)}}{\d\langle X\rangle^{(\P)}} = 0, \; \textnormal{$\P\otimes\d\langle X\rangle^{(\P)}$--a.e.},
	\]
	and therefore, since $	\langle X^{d,\P}\rangle^{(\P)} = \langle X^{d,\P},X^{d,\P}\rangle^{(\P)} \ll \langle X \rangle^{(\P)}$, $\P$--a.s., that
	\[
		\langle X^{d,\P}\rangle^{(\P)} = 0, \; \textnormal{$\P$--a.s.},
	\]
	which implies $X = X^{c,\P}$ outside a $\P$--null set. This shows that the problem of aggregation does not vanish in the 2BSDE setting of \Cref{sec_2bsde_X}.
	
	\subsection{(2)BSDEs driven by different decompositions}\label{sec_2bsdes_different_decompositions}
	
	The 2BSDE from \Cref{sec_wellposed_2bsde} and the one from \Cref{sec_2bsde_X} can be regarded as the two extreme cases, where the former is driven by the continuous part of $X$ and its compensated jump measure, and the latter is driven directly by $X$. Thus, at least intuitively, it does not seem that any other decomposition of $X$ would solve the issue of aggregation either.
	
	\medskip
	To illustrate this, let us consider a generic, but fixed, decomposition $X \leadsto  (X^\textnormal{I},\mu^\textnormal{J})$ into an It\^o-integrator part $X^\textnormal{I}$ and an integer-valued random measure $\mu^{\textnormal{J}}$. Suppose further that $X^\textnormal{I}$ is a $\P$--locally square-integrable martingale relative to each $\P\in\fP_0$. We can then consider the $\P$--BSDEs, for $\P\in\fP_0$, of the form
	\begin{align*}
		\cY^\P_t = \xi + \int_t^{T} f^{\P}_r\big(\cY^\P_r,\cY^\P_{r\smallertext{-}},\cZ^\P_r,\cU^\P_r(\cdot)&\big)\d C_r - \bigg(\int_t^{T} \cZ^\P_r \d X^\textnormal{I}_r\bigg)^{(\P)} - \bigg(\int_t^{T}\int_{\R^\smalltext{d}} \cU^\P_r(x)\tilde\mu^{\textnormal{J},\P}(\d r, \d x)\bigg)^{(\P)} - \int_t^{T}\d \cN^\P_r.
	\end{align*}
	Here, we suppose that under each measure $\P\in\fP_0$, the integrals generated by $X^\textnormal{I}$ and $\tilde{\mu}^{\textnormal{J},\P}$ are $\P$--orthogonal; well-posedness of the above BSDE is then ensured by \cite[Theorem 3.7]{possamai2024reflections}. The first component $Y$ of the corresponding 2BSDE should then satisfy
	\[
		Y_t = \underset{\bar{\P} \in \fP_\smalltext{0}(\cG_{\smallertext{t}\smalltext{+}},\P)}{{\esssup}^\P} \cY^{\bar{\P}}_t(T,\xi), \; \textnormal{$\P$--a.s.}, t \in [0,\infty], \; \P \in \fP_0,
	\]
	and the decomposition of $Y$, which is a result of well-posedness of reflected BSDEs, we expect, is of the form
	\begin{align*}
		\displaystyle Y_t &= \xi + \int_t^T f^\P_s\big(Y_s,Y_{s\smallertext{-}},Z^\P_s, U^\P_s(\cdot)\big) \d C_s- \bigg(\int_t^T Z^\P_s \d X^\textnormal{I}_s\bigg)^{(\P)} - \bigg(\int_t^T\int_{\R^\smalltext{d}} U^\P_s(x)\tilde\mu^{\textnormal{J},\P}(\d s,\d x)\bigg)^{(\P)}- \int_t^T \d N^\P_s+ K^\P_T - K^\P_t,
	\end{align*}
	where $N^\P$ is $\P$-orthogonal to $X^\textnormal{I}$ and $\tilde{\mu}^{\textnormal{J},\P}$, and $K^\P$ is predictable and non-decreasing,
	for each $\P\in\fP_0$. Note that our reflected BSDE result \cite[Theorem 3.4]{possamai2024reflections} is applicable in this generic setting, provided that $X^\textnormal{I}$ and $\tilde{\mu}^{\textnormal{J},\P}$ are suitably orthogonal.
	
	\medskip
	We see here that we encounter the same issues described in \Cref{rem::aggregation} for the integrands of $\tilde{\mu}^{\textnormal{J},\P}$, and in \Cref{sec_aggregation_without_random_measures} for the integrands of $X^\textnormal{I}$. More precisely, we are unable to aggregate the integrands of $\tilde{\mu}^{\textnormal{J},\P}$, and that the weakest condition we believe is required in a genuine 2BSDE setting to aggregate the integrands of $X^\textnormal{I}$ implies that the It\^o integrator must be continuous. The issue seems to arise inherently from the fact that the driving martingales, here $X^\textnormal{I}$ and $\tilde{\mu}^{\textnormal{J},\P}$, have jumps and their integrands appear in the generator of the corresponding BSDEs.
	
	\medskip
	Nevertheless, we do not exclude the possibility that the aggregation problem could be solved by other methods. Furthermore, we expect that a theory of 2BSDEs with jumps, as developed in this work, can be established for different decompositions of $X$ into a (local) martingale part with jumps and a compensated integer-valued random measure part, as described above with $X^\textnormal{I}$ and $\mu^\textnormal{J}$. That said, one would need to carefully check the construction in specific cases. The overall approach should be clear, and we expect at most technical challenges that may require further assumptions. This is, however, beyond the intended scope of this work.

	\section{Comparison with \emph{G}-BSDEs}\label{sec:GBSDE}
	We consider the setting of the $G$-expectation from \cite{peng2007g,peng2008multi,soner2011martingale} based on the description in \cite[Section 3]{nutz2013constructing}; we adopt the notation of the latter reference. For a nonempty, compact, and convex subset $\textbf{D} \subseteq \S^{d}_\smallertext{+}$, let $\fP_\textbf{D}$ denote the probability measures on $\Omega$ for which the canonical process $X$ is an $(\F,\P)$--local martingale with $\P$--a.s. continuous paths, such that $\d\langle X\rangle^{(\P)}_t \ll \d t$, $\P$--a.s., and $\d\langle X\rangle^{(\P)}_t/\d t \in \textbf{D}$, $\P\otimes\d t$--almost everywhere. The set $\fP_\textbf{D}$ satisfies \Cref{ass::probabilities2}; more precisely, the collection $(\fP(s,\omega))_{(\omega,s)\in\Omega\times[0,\infty)}$ for $\fP(s,\omega) \coloneqq \fP_\textbf{D}$ does so (see \cite[Proposition 3.1]{nutz2013constructing}). Let $G : \R^{d\times d} \longrightarrow \R$ be given by
	\begin{equation*}
		G(\Gamma) \coloneqq \frac{1}{2}\sup_{A\in\textbf{D}}\textnormal{Tr}[\Gamma A].
	\end{equation*}
	By \cite[Remark 3.6]{dolinsky2012weak}, the functional
	\begin{equation*}
		\cE^\textbf{D}_0[\xi] \coloneqq \sup_{\P \in \fP_\smalltext{\textbf{D}}}\E^\P[\xi],
	\end{equation*}
	which is defined for Borel-measurable functions $\xi$, coincides with the $G$-expectation on the space $\L^1_G(\Omega)$ of $G$-integrable random variables; this space consists exactly of Borel-measurable random variables $\xi : \Omega \longrightarrow \R$ that admit a quasi-continuous version and are $\fP_\textbf{D}$-uniformly integrable (see \cite[Theorem 52 and Remark 53]{denis2011function}), that is,
	\begin{equation*}
		\lim_{n \rightarrow \infty}\sup_{\P \in \fP_\smalltext{\textbf{D}}}\E^\P[|\xi|\1_{\{|\xi|\geq n\}}] = 0.
	\end{equation*}
 That $\xi$ admits a quasi-continuous version means that there exists $\xi^\prime : \Omega \longrightarrow \R$ such that for each $\varepsilon > 0$, there exists an open set $O$ with $\cE^\textbf{D}_0[\1_O] < \varepsilon$ and $\xi^\prime$ restricted to $\Omega \setminus O$ is continuous.
	
	\medskip
	Fix $T \in [0,\infty)$, and denote by $\H^2$ the space of $\R^d$-valued, $\F$-predictable processes $Z = (Z_t)_{t \in [0,T]}$ satisfying
	\begin{equation*}
		\|Z\|^2_{\H^\smalltext{2}} \coloneqq \cE^\textbf{D}_0\bigg[\int_0^T|Z_r|^2 \d r\bigg] < \infty.
	\end{equation*}
	Denote by $\H^{2,0}_G \subseteq \H^2$ the space of $\R^d$-valued, $\F$-predictable processes $Z = (Z_t)_{t \in [0,T]}$, with $Z_t = \sum_{i = 0}^{n-1} Z_{t_\smalltext{i}}\1_{(t_\smalltext{i},t_{\smalltext{i}\smalltext{+}\smalltext{i}}]}(t)$ for some $0 \leq t_1 \leq \ldots \leq t_n \leq T$ with $Z_{t_\smalltext{i}} = \varphi(X_{s_\smalltext{1}},\ldots, X_{s_\smalltext{\ell}})$ for some $0 \leq s_1 \leq \ldots \leq s_\ell \leq t_i$ and some bounded and Lipschitz-continuous function $\varphi : (\R^d)^\ell \longrightarrow \R$. Then $\H^2_G$ denotes the closure of $\H^{2,0}_G$ in $\H^2$. The elements of $\H^2_G$ are therefore $\R^d$-valued, $\F$-predictable processes with $\|Z\|^2_{\H^\smalltext{2}} < \infty$ for which there exists a sequence $(Z^n)_{n \in \N} \subseteq \H^{2,0}_G$ satisfying
	\begin{equation*}
		\lim_{n \rightarrow \infty}\|Z-Z^n\|^2_{\H^\smalltext{2}} = \lim_{n \rightarrow \infty}\cE^\textbf{D}_0\bigg[\int_0^T |Z_r-Z^n_r|^2\d r\bigg] = 0.
	\end{equation*}
	Similarly, we denote by $\S^2$ the space of real-valued, $\F$-optional processes $Y = (Y_t)_{t \in [0,T]}$ satisfying
	\begin{equation*}
		\|Y\|^2_{\S^\smalltext{2}} \coloneqq \cE^\textbf{D}_0\bigg[\sup_{r \in [0,T]}|Y_r|^2\bigg] < \infty,
	\end{equation*}
	and then by $\S^{2,0}_G \subseteq \S^2$ the subspace consisting of those of the form $Y_t = \varphi(t,X_{t_\smalltext{1}\land t},\ldots, X_{t_\smalltext{n}\land t})$, $t \in [0,T]$, for some bounded and Lipschitz-continuous function $\varphi : \R \times (\R^d)^n \longrightarrow \R$. We denote by $\S^2_G$ the closure of $\S^{2,0}_G$ in $\S^2$.

	\medskip
	Suppose now that $\xi$ and $f$ satisfy the conditions of \cite[Theorem 4.1]{hu2014backward} for some $\beta > 2$; this reduces to a standard integrability condition combined with a crucial regularity condition in the $\omega$-variable. Then there exists a (unique) triplet $(Y,Z,K)$ such that $(Y,Z) \in \S^2_G \times \H^2_G$ and
	\begin{equation}\label{eq::g_bsde}
		K_t \coloneqq Y_0 - Y_t - \int_0^t f_s(Y_s,Z_s) \d s + \int_0^t Z_s \d X_s, \; t \in [0,T],
	\end{equation}
	satisfies (see \cite[Proposition 3.4]{soner2011martingale})
	\begin{equation*}
		K_t = \underset{\bar{\P} \in \fP_\smalltext{\textbf{D}}(\cF_\smallertext{t},\P)}{{\essinf}^\P}  \E^{\bar\P}[K_T|\cF_t], \; \textnormal{$\P$--a.s.}, \; t \in [0,T], \; \P \in \fP_\textbf{D},
	\end{equation*}
	and then also (see \cite[Proposition 4.5.$(iii)$]{nutz2012superhedging})
	\begin{equation*}
		K_t = \underset{\bar{\P} \in \fP_\smalltext{\textbf{D}}(\cF_{\smallertext{t}\smalltext{+}},\P)}{{\essinf}^\P}  \E^{\bar\P}[K_T|\cF_{t\smallertext{+}}], \; \textnormal{$\P$--a.s.}, \; t \in [0,T], \; \P \in \fP_\textbf{D}.
	\end{equation*}
	The integral relative to $X$ in \eqref{eq::g_bsde} is defined in terms of the $G$--stochastic calculus and coincides, under every $\P \in \fP_\textbf{D}$, with $(\int_0^\cdot Z_r \d X_r)^{(\P)}$, the It\^o stochastic integral under $\P$ (see \cite[Proposition 3.3]{soner2011martingale}). This shows that the $G$-BSDE solution is also a 2BSDE solution in the sense of \Cref{sec_wellposed_2bsde}, and thus $Z = \widehat{Z}$, $\{\P\otimes \d t : \P \in \fP_\textbf{D}\}$--quasi-surely, and $(Y,K) = (\widehat{Y},\widehat{K}^\P)$ outside a $\P$--null set for every $\P \in \fP_\textbf{D}$.
	
	\medskip
	The converse statement is a bit more subtle. Once one is given a solution to the 2BSDE, the only point left is to prove that the solution belongs to the appropriate $G$-spaces. It is relatively easy to show that the process $Y$---seen as a solution to the 2BSDE---inherits the required regularity from the data of the equation: indeed, this was already proved in \cite{soner2012wellposedness} through standard \emph{a priori} estimates which do not require any notion of $G$--stochastic calculus. However, the regularity for $Z$ is less obvious, and does seem to require techniques from $G$--stochastic calculus; see, for instance, the approach in \cite{hu2014backward}.
	
	\medskip
	In our view, the 2BSDE framework is thus strictly more general than the $G$-BSDE one, particularly because it accommodates significantly more irregular data with regard to the terminal condition and generator. As such, in order to prove well-posedness of a $G$-BSDE, it seems more natural to first prove well-posedness of the corresponding and more general 2BSDE, and then show, as explained above, that under extra regularity requirements on the data, this is also a solution in the $G$-BSDE sense. One should also mention that, from the point of view of applications, the quasi-continuity in $\omega$, which is ubiquitous in the $G$--stochastic calculus, is a very strong requirement. It makes it, for instance, impossible to use $G$-BSDEs to prove the results in contract theory from \cite{cvitanic2018dynamic} (which uses 2BSDEs), and, from the point of view of model uncertainty in finance, since even standard stochastic integrals fail to have this level of regularity, this should typically be inherited by contingent claims, so much so that any duality result for the super-replication of contingent claims is of little practical use.

	{\footnotesize
		\bibliography{bibliographyDylan}}

\end{document}